\pdfoutput=1

\documentclass[a4paper,11pt]{scrbook}

\usepackage[english]{babel}   

\usepackage{scrhack}

\KOMAoptions{BCOR=8mm}

\setkomafont{section}{\rmfamily\bfseries\LARGE}
\setkomafont{subsection}{\rmfamily\bfseries\Large}
\setkomafont{subsubsection}{\rmfamily\bfseries\large}
\setkomafont{chapter}{\rmfamily\bfseries\scshape\huge}
\setkomafont{chapterentry}{\rmfamily\bfseries\scshape}
\setkomafont{partentry}{\rmfamily\bfseries\scshape}
\setkomafont{part}{\rmfamily\bfseries\scshape\huge}
\setkomafont{partnumber}{\rmfamily\bfseries\scshape\huge}
\setkomafont{partentrypagenumber}{\rmfamily\bfseries\scshape}

\setkomafont{descriptionlabel}{\normalfont\bfseries}

\usepackage[automark]{scrpage2}
\KOMAoptions{headsepline=true}
\setheadsepline{.4pt}
\setkomafont{pageheadfoot}{\normalfont\normalcolor\upshape} 

\usepackage[T1]{fontenc} % to make special characters as in Sa\"ibi searchable

\usepackage[protrusion=true,expansion=true]{microtype}

\usepackage[utf8]{inputenc}
\usepackage{amssymb,amsmath,stmaryrd,mathrsfs, amsthm}
\usepackage{mathtools}
\usepackage[all,2cell,matrix,arrow]{xy}
\usepackage{color}
\definecolor{darkgreen}{rgb}{0,0.45,0} 
\usepackage{enumitem}
\usepackage{mathdots}
\usepackage[babel=true]{csquotes}
\usepackage{fancybox} 
\usepackage{footnote}
\usepackage{amsthm}
\usepackage{array}
\usepackage{tikz-cd}
\usetikzlibrary{matrix, calc, arrows}
\usepackage{xspace}
\usepackage{graphicx}
\usepackage[
 backref=true,
 maxbibnames = 99,
 uniquename=false,
firstinits=true
]{biblatex}
 \usepackage[colorlinks,citecolor=darkgreen,linkcolor=darkgreen]{hyperref}

\makesavenoteenv{tabular}
\DeclareMathAlphabet{\mathcalligra}{T1}{calligra}{m}{n}

\let\into\hookrightarrow

\let\fib\twoheadrightarrow

\def\acof{\mathrel{\mathrlap{\hspace{3pt}\raisebox{4pt}{$\scriptscriptstyle\sim$}}\mathord{\rightarrowtail}}}

\newtheorem{thm}{Theorem} 
\newtheorem{prop}[thm]{Proposition}  
\newtheorem{lem}[thm]{Lemma} 
\newtheorem{cor}[thm]{Corollary} 
\newtheorem{notn}[thm]{Notation} 
\newtheorem{eg}[thm]{\textit{Example}} 
\newtheorem{rmk}[thm]{\textit{Remark}} 
\theoremstyle{definition} 
\newtheorem{defi}[thm]{Definition}    
 
\numberwithin{thm}{section}
\newcommand{\fsix}[6]{\begin{array}{rcl} #1&\longrightarrow &#2\\ #3&\longmapsto &#4\\ #5&\longmapsto &#6\\\\ \end{array}} 
\newcommand{\ffour}[4]{\begin{array}{rcl} #1&\longrightarrow &#2\\ #3&\longmapsto &#4 \end{array}}
\newcommand{\pushoutcorner}[1][dr]{\save*!/#1+1.2pc/#1:(1,-1)@^{|-}\restore}
\newcommand{\pullbackcorner}[1][dr]{\save*!/#1-1.2pc/#1:(-1,1)@^{|-}\restore}

\let\ea\expandafter
\def\mdef#1#2{\ea\ea\ea\gdef\ea\ea\noexpand#1\ea{\ea\ensuremath\ea{#2}\xspace}}

\def\pr{\;\vdash\;}
\def\ty{\;\mathsf{type}}
\def\m#1{\llbracket#1\rrbracket}
\def\commutatif{\ar@{}[rrdd]|{\circlearrowright}} 
\def\pcom{\ar@{}[rd]|{\circlearrowright}} 
\def\ppcom{\ar@{}[ld]|{\circlearrowright}} 
\def\coml{\ar@{}[ldd]|{\circlearrowright}} 
\def\comr{\ar@{}[rdd]|{\circlearrowright}}
\def\dia{\ar@{}[rrdd]|{\rotatebox{-45}{$\Longrightarrow$}}} 
\newdir{ >}{{}*!/-5pt/\dir{>}}
\def\sprod#1{{\textstyle\prod_{#1}}}
\def\ssum#1{{\textstyle\sum_{#1}}}
\mdef\Id{\mathsf{Id}}

\def\I{\bold{I}}
\def\tr{\bigtriangledown}
\def\pI{\bold{\dot{I}}}
\def\b0{\bold{0}}
\def\b1{\bold{1}}
\def\P{\bold{P}}
\def\Gpd{\mathbf{Gpd}}
\def\GGpd{\mathbf{Gpd^G}}
\def\f{_{\mathbf{f}}}
\def\G{\mathbf{G}}

\def\cI{\mathbf{\check{I}}}

\def\pr{\;\vdash\;}
\def\ty{\;\mathsf{type}}
\def\m#1{\llbracket#1\rrbracket}
\newdir{ >}{{}*!/-5pt/\dir{>}}
\def\com1{\ar@{}[rd]|{\circlearrowright}} 
\makeatletter
\DeclareRobustCommand{\rvdots}{%
    \vbox{
        \baselineskip4\p@\lineskiplimit\z@
        \kern-\p@
        \hbox{.}\hbox{.}\hbox{.}
    }}
    \makeatother

\let\ea\expandafter

\mdef\el{\mathsf{El}}
\mdef\Id{\mathsf{Id}}
\def\id{\leadsto} 
  
\mdef\iscontr{\mathsf{isContr}} 

\def\mdef#1#2{\ea\ea\ea\gdef\ea\ea\noexpand#1\ea{\ea\ensuremath\ea{#2}\xspace}}
\def\alwaysmath#1{\ea\ea\ea\global\ea\ea\ea\let\ea\ea\csname your@#1\endcsname\csname #1\endcsname
  \ea\def\csname #1\endcsname{\ensuremath{\csname your@#1\endcsname}\xspace}}

\alwaysmath{alpha}
\alwaysmath{beta}
\alwaysmath{gamma}
\alwaysmath{Gamma}
\alwaysmath{delta}
\alwaysmath{Delta}
\alwaysmath{epsilon}
\mdef\ep{\varepsilon}
\alwaysmath{zeta}
\alwaysmath{eta}
\alwaysmath{theta}
\alwaysmath{Theta}
\alwaysmath{iota}
\alwaysmath{kappa}
\alwaysmath{lambda}
\alwaysmath{Lambda}
\alwaysmath{mu}
\alwaysmath{nu}
\alwaysmath{xi}
\alwaysmath{pi}
\alwaysmath{rho}
\alwaysmath{sigma}
\alwaysmath{Sigma}
\alwaysmath{tau}
\alwaysmath{upsilon}
\alwaysmath{Upsilon}
\alwaysmath{phi}
\alwaysmath{Pi}
\alwaysmath{Phi}
\mdef\ph{\varphi}
\alwaysmath{chi}
\alwaysmath{psi}
\alwaysmath{Psi}
\alwaysmath{omega}
\alwaysmath{Omega}
\let\al\alpha
\let\be\beta
\let\gm\gamma

\let\de\delta

\let\la\lambda

\newcommand{\ffive}[5]{\begin{array}{rrcl} #1: & #2&\longrightarrow &#3\\ &#4&\longmapsto &#5  \end{array}}

\title{On lifting univalence to the equivariant setting}
\author{Anthony BORDG}

\begin{document}

% \maketitle

%%%   roman page numbers
\frontmatter 

%%%   front page has special formatting, in particular symmetric (oneside)
 
\KOMAoptions{BCOR=0mm,twoside=off}
\recalctypearea

\begin{titlepage}
 
\centering

%\begin{center}
{\Large UNIVERSIT\'E NICE SOPHIA ANTIPOLIS   --   UFR Sciences}\\
\vspace*{0.5cm}
\'Ecole Doctorale Sciences Fondamentales et Appliqu\'ees\\
\vspace*{1.5cm}
{\Large\bf TH\`ESE}\\
\vspace*{0.3cm}
pour obtenir le titre de\\
\vspace*{0.1cm}
~~~{\Large\bf Docteur en Sciences}\\
\vspace*{0.1cm}
Sp\'ecialit\'e 
{\Large {\sc Math\'ematiques}}\\
\vspace*{0.8cm}
pr\'esent\'ee et soutenue par\\
\vspace*{0.1cm}
{\large\bf Anthony BORDG}\\
\vspace*{1.0cm}
{\LARGE\bf On lifting univalence to the equivariant setting}\\
\vspace*{1.0cm}
Th\`ese dirig\'ee par {\bf Andr\'e HIRSCHOWITZ}\\
\vspace*{0.2cm}
soutenue le 09 novembre 2015\\

\vspace*{2.3cm}
Membres du jury :\\
\vspace*{0.4cm}
\begin{tabular}{llllll}
M. & Denis-Charles CISINSKI & Professeur des universités & examinateur\\
M. & Thierry COQUAND & Professor & examinateur \\
M. & Andr\'e HIRSCHOWITZ & Professeur émérite & Directeur de th\`ese\\
M. & Peter LEFANU LUMSDAINE & Doctor & rapporteur et examinateur\\
M. & Michael SHULMAN & Associate Professor & rapporteur et examinateur\\
M. & Carlos SIMPSON & Directeur de Recherche  & examinateur \\
M. & Bertrand TOËN & Directeur de Recherche & examinateur \\
\end{tabular}

%\vspace*{3cm}
%\pagefill
\vfill
Laboratoire Jean-Alexandre Dieudonn\'e, Universit\'e de Nice, Parc Valrose, 06108 NICE
%\end{center}
\end{titlepage}
\KOMAoptions{BCOR=9mm,twoside=on}
\recalctypearea
\cleardoublepage

%%%   you might want to put your own title page
%  \input{title_page}

%%%    for the abstract we remove page numbers and headers
 \pagestyle{empty}

% \footnotesize
% \small

\textbf{\Large Abstract}\bigskip

This PhD thesis deals with some new models of intensional type theory and the \textit{Univalence Axiom} introduced by Vladimir Voevodsky. Our work takes place in the framework of the definitions of type-theoretic model categories, type-theoretic fibration categories (the notion of model under consideration in this thesis) and universe in a type-theoretic fibration category, definitions due to Michael Shulman. This universe does not correspond to a strict type-theoretic universe in the sense that it is closed up to type-theoretic operations only up to isomorphism. The goal of this thesis consists mainly in the exploration of the stability of the univalence axiom, in particular in the following sense: being given a type-theoretic fibration category $\mathscr{C}$ equipped with a univalent universe U, we are eager to endow the presheaf category $[\mathcal{D}^{op},\mathscr{C}]$, where $\mathcal{D}$ is a small category, with the structure of a type-theoretic fibration category plus a univalent universe.\\   
There are at least two ways to proceed, one using the so-called projective model structure on a functor category and the other by using the so-called injective one.\\
In the very specific case where $\mathscr{C}$ is $\Gpd$ and $\mathcal{D}$ is $\mathbf{B}(\mathbb{Z}/2\mathbb{Z})$ we reach this goal by giving two models in the same underlying category $[\mathbf{B}(\mathbb{Z}/2\mathbb{Z})^{op},\Gpd]$, models that share the construction of a common universe of (small) discrete groupoids with involution. To built the first model we use the projective model structure and prove that our universe is non-univalent. In the second one we use the injective model structure and prove that our universe is univalent.

\newpage

\textbf{\Large Résumé}\bigskip

Cette thèse de doctorat a pour sujet les modèles de la théorie intensionnelle des types avec l'\textit{axiome d'univalence}  introduit par Vladimir Voevodsky. L'auteur prend pour cadre de travail les definitions de \textit{type-theoretic model category}, \textit{type-theoretic fibration category} (cette dernière étant la notion de modèle considérée dans cette thèse) et d'univers dans une \textit{type-theoretic fibration category}, définitions dues à Michael Shulman. Cet univers ne correspond pas à un univers strict de la théorie des types en ce sens que les opérations de formation des types sont closes seulement à isomorphisme près. La problématique principle de cette thèse consiste à approfondir notre compréhension de la stabilité de l'axiome d'univalence en le sens précis suivant: soit $\mathscr{C}$ une \textit{type-theoretic fibration category} équippée d'un univers univalent, nous voulons munir la catégorie de préfaisceaux $[\mathcal{D}^{op}, \mathscr{C}]$ d'une structure de \textit{type-theoretic fibration category} équippée d'un univers univalent.\\
Il y a au moins deux manières de faire, l'une ayant recours à la structure de modèle projective sur une catégorie de foncteurs, l'autre à la structure de modèle injective.\\
Quand $\mathscr{C}$ est la catégorie $\Gpd$ et $\mathcal{D}$ est $\mathbf{B}(\mathbb{Z}/2\mathbb{Z})$ nous atteignons dans le présent travail le but énoncé plus haut en donnant deux modèles dans la même catégorie sous-jacente, à savoir $[\mathbf{B}(\mathbb{Z}/2\mathbb{Z})^{op},\Gpd]$, modèles qui partagent un même univers de (petits) groupoïdes discrets équippés d'une involution. Dans notre premier modèle construit à l'aide de la structure projective, nous prouvons que cet univers n'est pas univalent. Dans le second modèle, construit à l'aide de la structure injective, nous prouvons que cet univers est univalent.            %%%%%%%%%%%%%%     

\newpage

%%%   back to normal size and page style
% \normalsize 
\pagestyle{plain}

\textbf{\Large Acknowledgements}\bigskip

\begin{itemize}

\item First and above all I would like to thank my advisor André Hirschowitz. First, for his bravery  of supervising this thesis at a time the word was not spread about Univalent Foundations at least in France. Second, for uncountable discussions and help throughout this thesis. Last, for his energy and advice on any subject.
\item My parents to allow me to push forward in long studies.
\item Alexandra for her patience and help with LateX.
\item Vladimir Voevodsky for allowing me to take part in the special year on Univalent Foundations in Princeton (IAS) when I was in the first year of my PhD studies.
\item Michael Shulman for suggesting me an avenue of research that prompted this work, for numerous answers to email questions and for inviting me three months in San Diego through a Fulbright grant where I told him about my work and progress.
\item Steve Awodey for hosting me three months in Pittsburgh in the second half of my Fulbright grant and for allowing me to give a few talks at CMU about my work.
\item The Fulbright team in Paris for the opportunity for spending 6 months in the US as a Fulbright Fellow.
\item Yves Bertot, André Hirschowitz and Carlos Simpson for support in the Fulbright process. 
\item My office mate Eduard Balzin for numerous maths discussions.
\item Benedikt Ahrens for advice on LaTeX and for acting as an older academic brother.
\item My office mates Eduard Balzin and Charles Collot for making my stay in Nice more enjoyable.
\item LJAD and EDSFA administrative crew for their nice work.
\item The technical staff, Jean-Marc Lacroix and Roland Ruelle, for the logistics of my PhD defense partly by webconference.
\item I would like to thank the referees, Peter LeFanu Lumsdaine and Michael Shulman, for their careful readings. A special thank goes to André Hirschowitz, Peter LeFanu Lumsdaine and Michael Shulman for their list of suggestions and improvements. I also thank the other members of the dissertation defense, Denis-Charles Cisinski, Thierry Coquand, Carlos Simpson and Bertrand Toën.
\end{itemize}      %%%%%%%%%%%%%%%%%%%%%%%%   

\newpage 

\textbf{\Large Abstract-Long Version}\bigskip

This PhD thesis deals with some new models of intensional type theory and the \textit{univalence axiom} introduced by Vladimir Voevodsky. Our work takes place in the framework of the definitions of type-theoretic model cateories, type-theoretic fibration categories (the notion of model under consideration in this thesis) and universe in a type-theoretic fibration category, definitions due to Michael Shulman. This universe does not correspond to a strict type-theoretic universe in the sense that it is closed up to type-theoretic operations only up to isomorphism. The goal of the author consists mainly in the exploration of the stability of the univalence axiom as initiated by Michael Shulman in a series of papers, in particular in the following sense: being given a type-theoretic fibration category $\mathscr{C}$ equipped with a univalent universe U, we are eager to endow the presheaf category $[\mathcal{D}^{op},\mathscr{C}]$, where $\mathcal{D}$ is a small category, with the structure of a type-theoretic fibration category plus a univalent universe.\\   
Under some slight assumptions on $\mathscr{C}$ there are at least two ways to proceed, one using the so-called projective model structure on a functor category and the other by using the so-called injective one. In some very specific cases the author finds a natural candidate to lift the univalent universe.\\
More specifically, we chose to start with one of the simplest non-trivial case, namely $\mathbf{B}(\mathbb{Z}/2\mathbb{Z})$ the groupoid associated with the group with two elements as an index category and the category of groupoids $\Gpd$ as a target category. Thus, our presheaf category is nothing but the groupoids equipped with an involution. This case is non-trivial since $\mathbf{B}(\mathbb{Z}/2\mathbb{Z})$ has a non-trivial  automorphism, in particular this is not a Reedy category but a generalized Reedy category. As a consequence this case was not covered by Shulman's work that deals with inverse diagrams and elegant Reedy categories which are particular cases of Reedy categories. In a Reedy category non-trivial isomorphisms are not allowed, hence generalized Reedy categories present the interesting difficulty of dealing with non-trivial isomorphisms.\\
The projective model structure requires a non-trivial characterization of the trivial cofibrations to provide the structure of a type-theoretic fibration category. We overcome this first difficulty. We also provide in this setting the categorical structure required for modelling universes. However we prove this projective model structure is not suitable for univalent universes since we are able to break univalence that holds in $\Gpd$ for a universe of small discrete groupoids. Even worse we have been able to break a weaker form of univalence, namely function extensionality, that holds in the internal language of the type-theoretic fibration category $\Gpd$. \\
On the other side the present work confirms that the injective model structure  is suitable for univalence, indeed in the case of functor categories on elegant reedy categories the injective model structure coincides with  the reedy model structure used by Shulman. But the injective model structure presents some technical difficulties when it comes to the universe due to the mysterious nature in general of fibrant objects and fibrations in the injective case. In the case under interest in this thesis we overcome this second difficulty.\bigskip

\textbf{\Large Contributions}\bigskip

To sum up our contributions we have:
\begin{itemize}

\item A model of  intensional type theory with $\ssum{}$, $\sprod{}$, and \Id-types and a universe in $[\mathbf{B}(\mathbb{Z}/2\mathbb{Z})^{op},\Gpd]$ using the projective model structure. We prove that function extensionality  does not hold in the internal type theory of our model (see chapter \ref{sec:chp4}). Moreover we also prove that our universe does not satisfy univalence (see chapter \ref{sec:chp4}). These two proofs rely on a non-trivial characterization of homotopy equivalences in  $[\mathbf{B}(\mathbb{Z}/2\mathbb{Z})^{op},\Gpd]$ equipped with the projective model structure.
\item A model of  intensional type theory with $\ssum{}$, $\sprod{}$, and \Id-types and a univalent universe in $[\mathbf{B}(\mathbb{Z}/2\mathbb{Z})^{op},\Gpd]$ using the injective model structure. In this model we succeed to deal with a generalized Reedy category with a non-trivial automorphism as an index category (see  chapter \ref{sec:chp5}). 
\end{itemize}

\newpage

\textbf{\Large Résumé-Version Longue}\bigskip

Cette thèse de doctorat a pour sujet les modèles de la théorie intensionnelle des types avec l'\textit{axiome d'univalence} introduit par Vladimir Voevodsky. L'auteur prend pour cadre de travail les définitions de \textit{type-theoretic model category}, \textit{type-theoretic fibration category} (cette dernière étant la notion de modèle considérée dans cette thèse) et d'univers dans une \textit{type-theoretic fibration category}, définitions dues à Michael Shulman. Cet univers ne correspond pas à un univers strict de la théorie des types en ce sens que les opérations de formation des types sont closes à isomorphisme près. La problématique principle de cette thèse consiste à approfondir notre compréhension de la stabilité de l'axiome d'univalence initiée par Michael Shulman dans une série d'articles, en le sens précis suivant: soit $\mathscr{C}$ une \textit{type-theoretic fibration category} équippée d'un univers univalent, nous voulons munir la catégorie de préfaisceaux $[\mathcal{D}^{op}, \mathscr{C}]$ d'une structure de \textit{type-theoretic fibration category} équippée d'un univers univalent.\\
A l'aide d'hypothèses raisonnables sur $\mathscr{C}$ il y a au moins deux manières de faire, l'une ayant recours à la structure de modèle projective sur une catégorie de foncteurs, l'autre à la structure de modèle injective. Dans des cas bien particuliers nous avons pu atteindre l'objectif énoncé plus haut et équipper notre catégorie de foncteurs d'un univers (univalent).\\
Plus précisément, l'auteur débute avec l'un des cas non triviaux les plus simples, à savoir $\mathbf{B}(\mathbb{Z}/2\mathbb{Z})$ le groupoïde associé au groupe à deux éléments comme choix pour la catégorie index et la catégorie des groupoïdes $\Gpd$ comme catégorie but. Nous choisissons donc de nous intéresser aux groupoïdes équippés d'une involution. Ce choix est non trivial en ce que la catégorie index $\mathbf{B}(\mathbb{Z}/2\mathbb{Z})$ possède un automorphisme non trivial, en particulier cette catégorie est une catégorie Reedy généralisée mais elle n'est pas Reedy (stricte). On peut notamment souligner que ce cas n'a pas été traité dans les travaux de Michael Shulman mentionnés plus haut qui se cantonnent aux cas des diagrammes inverses et des catégories Reedy élégantes, qui sont dans les deux cas des cas particuliers de catégories Reedy (strictes) en guise de catégories index. Rappelons que dans une catégorie Reedy (stricte) les isomorphismes non triviaux ne sont pas autorisés et leur gestion dans les catégories Reedy généralisées présente donc un défi technique intéressant.\\
La structure de modèle projective nécessite une caractérisation non triviale des cofibrations triviales pour équipper la catégorie $\mathbf{B}(\mathbb{Z}/2\mathbb{Z})$ d'une structure de \textit{type-theoretic fibration category}. Nous surmontons dans cette thèse cette première difficulté. Nous fournissons également dans ce contexte projectif la structure catégorique nécessaire pour avoir des univers. Toutefois, nous prouvons que notre univers de groupoïdes discrets equippés d'une involution n'est pas univalent, ainsi l'univalence qui est satisfaite dans $\Gpd$ pour l'univers des petits groupoïdes discrets a été brisée. De même, dans ce modèle l'extensionnalité des fonctions, pourtant satisfaite dans tout univers de groupoïdes, a été brisée.\\
 D'un autre côté le présent travail confirme le caractère adéquat de la structure de modèle injective par rapport à l'univalence, en effet pour des catégories de foncteurs sur une catégorie de Reedy élégante la structure de modèle injective coïncide avec la structure de modèle Reedy utilisée précédemment par Shulman. Toutefois la structure de modèle injective présente des difficultés techniques dans la construction de l'univers dues à la nature mystérieuse des objets fibrants et des fibrations pour la structure injective. Cette seconde difficulté est également surpassée dans notre travail. \bigskip

\textbf{\Large Contributions}\bigskip

Résumons nos contributions:
\begin{itemize}

\item Un modèle de la théorie des types intensionnelle avec types $\ssum{}$, $\sprod{}$, et \Id plus un univers dans la catégorie $[\mathbf{B}(\mathbb{Z}/2\mathbb{Z})^{op},\Gpd]$ en ayant recours à la structure de modèle projective. De plus, nous donnons une preuve que ni l'univalence ni l'extensionnalité des fonctions ne sont satisfaites dans ce modèle (\textit{cf} chapitre \ref{sec:chp4}) grâce notamment à une caractérisation non triviale des équivalences d'homotopie projectives. 
\item Un modèle de la théorie des types intensionnelle avec types $\ssum{}$, $\sprod{}$, et \Id plus un univers univalent dans la catégorie $[\mathbf{B}(\mathbb{Z}/2\mathbb{Z})^{op},\Gpd]$ en ayant recours à la structure de modèle injective. Notons que dans ce modèle la catégorie index est une catégorie Reedy généralisée avec un automorphisme non trivial (\textit{cf}  chapitre \ref{sec:chp5}). 
\end{itemize}

\newpage

%%%   we switch to the headings style we defined in the preamble
\pagestyle{scrheadings}

%%%  for the table of contents, protrusion should be disabled
\microtypesetup{protrusion=false}
 \tableofcontents
\microtypesetup{protrusion=true}

%%%   switch to regular numbering, starting from 1
\mainmatter

\chapter{Introduction}
\label{intro}

\section{Introduction}
In the seventies Per Martin-L\"of set a framework out, suitable for constructive mathematics, called Martin-L\"of Type Theory (MLTT for short). It is well known that MLTT enjoys very nice computational properties and matches with logic via a famous correspondence, the so-called \textit{Curry-Howard correspondence}.\\
Recently Vladimir Voevodsky added an axiom to MLTT, the so-called \textit{Univalence Axiom}, that roughly asserts an equivalence between the type of propositional equalities between any two small types (\textit{i.e.} two types of a type-theoretic universe) and the type of equivalences between these two small types. This brave new world, MLTT plus the \textit{Univalence Axiom}, was coined \textit{Univalent Foundations} by Voevodsky (UF for short). \\
When a new axiom, like the univalence axiom, is added to a theory then a proof of relative consistency is desirable. To achieve this proof of equiconsistency, the standard way consists in giving a model (in the logical sense). In fact, UF with one universe are at least as consistent as ZFC together with two strongly inaccessible cardinals.
This proof was achieved by Voevodsky \cite{KLV} who gave a model of UF in the Quillen model category of simplicial sets.\\
Models of UF were pursued by Michael Shulman \cite{shulman:invdia,shul13}, David Gepner and Joachim Kock \cite{gepnerkock12} and lately by Denis-Charles Cisinski \cite{Cis14}. The groupoid model of Martin Hofmann and Thomas Streicher that contains a univalent universe of small discrete groupoids can be seen as an ancestor of this line of work \cite{groupoidmodel}.\\
The goal of this thesis consists in the exploration of the stability of the univalence axiom as initiated by Michael Shulman, in particular in the following sense: being given a type-theoretic fibration category $\mathscr{C}$ equipped with a univalent universe U, we are eager to lift this univalent universe in the presheaf category $[\mathcal{D}^{op},\mathscr{C}]$, where $\mathcal{D}$ is a small category, equipped with the structure of a type-theoretic fibration category.\\
This goal was achieved by Michael Shulman in some specific cases, especially in \cite{shulman:invdia} when $\mathcal{D}$ is an inverse category and in \cite{shul13} when $\mathscr{C}$ is $\mathbf{sSet}$ and $\mathcal{D}$ is any elegant Reedy category. Note that inverse categories and elegant Reedy categories being strict Reedy categories they do not allow non-trivial isomorphism.\\
In this thesis we explore the stability of the univalence axiom with respect to the projective and injective model structures on functor categories. More specifically, we treat the case where $\mathscr{C}$ is the category of groupoids $\Gpd$ and $\mathcal{D}$ is $\mathbf{B}(\mathbb{Z}/2\mathbb{Z})$ the groupoid associated with the group with two elements, that contains a non-trivial automorphism. In particular we give a new model of UF since we construct a univalent universe in the full subcategory of fibrant objects of $[\mathbf{B}(\mathbb{Z}/2\mathbb{Z})^{op},\Gpd]$ with respect to the injective structure. Moreover, more surprisingly we give a second model of type theory in the same underlying category with respect to the projective fibrations where the very same universe turns out to be non-univalent. To the best of our knowledge this second model is the first one derived from a Quillen model structure where not all objects are cofibrant. Since the projective and injective model categories, having the same weak equivalences, form a Quillen equivalence, this is also, again to our knowledge, the first instance of two Quillen equivalent model categories that host  different models of type theory. If it should happen that not all objects are fibrant-cofibrant then our method of proof makes it clear that even when a whole model structure is available at hand only the classes of fibrations, acyclic cofibrations and right homotopy equivalences are relevant for models of type theory. Our thesis stresses this fact and shows that a model invariance principle to the effect of "equivalent homotopy theories have equivalent internal type theories", as suggested here \url{https://ncatlab.org/homotopytypetheory/show/open+problems}, should be formulated in a more careful way.   
Last, note that our univalent model in the injective setting is not a special case of later results by Shulman in \cite{shulman:eiuniv}, since therein Shulman constructed a model of UF in a certain model category that presents the homotopy theory of presheaves on an EI-category, but it is not the injective model structure on a functor category (in the case of $\mathbb{Z}/2\mathbb{Z}$ or any other group $G$ Shulman's model specializes to the slice category $\Gpd/\mathbf{B}(G)$ with its natural model structure, which is well-known to be Quillen equivalent to, but not identical to, the injective model structure on $[\mathbf{B}(G),\Gpd]$). This is rightly one of the points of our PhD thesis that Quillen equivalent model categories can host different models of type theory. So in the context of the semantics of type theory through Shulman's notion of type-theoretic fibration category such an equivalence is not obvious or trivial.

\section{Organization}
The chapter 2 gives the necessary background on model categories. The chapter 3 details what we mean by a model of type theory with a univalent universe. In particular we review in this chapter the notions of type-theoretic fibration category and universe in a type-theoretic fibration category. Also we present the notion of type-theoretic model category and make explicit the link between the latter and the former. Though this chapter 3 is short and it could be merged with chapter 2 as one more section, we keep it separate to emphasize the distinction between type-theoretic fibration categories and model categories. Indeed we have already warned the reader familiar with model categories against the appeal of thinking in terms of model structure in the context of type theory since this could be misleading. In accord with this philosophy we introduce the relevant semantical notions to the reader in a separate chapter. The chapter 4 follows the projective way and presents a model of type theory with a non-univalent universe in the underlying category $[\mathbf{B}(\mathbb{Z}/2\mathbb{Z})^{op},\Gpd]$ , while the chapter 5 follows the injective way and proves that the previous universe becomes univalent with respect to the injective structure. 

\section{Notations}
\label{not}

For the rest of this thesis we denote the group $\mathbb{Z}/2\mathbb{Z}$ simply by $\G$ and by a slight abuse of notation we also denote by $\G$ the associated groupoid $\mathbf{B}(\mathbb{Z}/2\mathbb{Z})^{op}$. This choice is for convenience but the reader should be aware that it does not mean $\G$ is any group, however our results should be generalized to the case of a group action by any (commutative) group $G$ in a forthcoming work. As a consequence we will denote the presheaf category $[\mathbf{B}(\mathbb{Z}/2\mathbb{Z})^{op},\text{Gpd}]$ simply by $\GGpd$.\\
The reader should note that a presheaf valued in groupoids on $\G$ is nothing but a groupoid equipped with an involution, and a morphism in $\GGpd$ is nothing but an equivariant functor, namely a functor between groupoids compatible with the involutions on the domain and codomain. Such a groupoid with its involution will be denoted by a capital letter $A$ and the corresponding Greek letter $\al$ will be used to refer to its involution when needed (except when stated otherwise).\\
As usual in category theory, the initial object and the terminal object of $\Gpd$ will be denoted by $\bold{0}$ and $\bold{1}$ respectively.\\
We have an obvious functor from $\G$ to $\bold{1}$ denoted $!$ and an obvious inclusion functor from $\bold{1}$ to $\G$ denoted $I$. These two functors induce by precomposition the two following functors between $\GGpd$ and $\Gpd$,
$$\ffive{\_}{\GGpd}{\Gpd}{G}{\underline{G}}$$
namely the underlying/forgetful functor that maps a groupoid $G$ equipped with an involution to its underlying groupoid denoted $\underline{G}$ which is formaly $G\circ I$;
$$\ffive{() \circ !}{\Gpd}{\GGpd}{G}{G!} $$
that maps, being given the canonical isomorphism between $\Gpd$ and $\Gpd^{\bold{1}}$, a groupoid $G$ to $G\circ !$, shorten by $G!$ for convenience, which is the groupoid $G$ equipped with the identity involution.
The underlying functor has a left adjoint denoted $S$ that maps a groupoid $G$ to $S(G):= G\textstyle\coprod G$ equipped with the swap involution. Last, the trivial $G$-groupoid functor $!$ has a right adjoint, namely the  fixed-points functor, 
$$\ffive{()^{\G}}{\GGpd}{\Gpd}{F}{F^\G}$$ where $F^\G$ is the groupoid of fixed points and fixed morphisms in $F$. Note that $F^\G$ is $\text{lim} F$.\\
Since limits and colimits are pointwise in a presheaf category, $\bold{0}!$ and $\bold{1}!$ are the corresponding initial and terminal objects in $\GGpd$.\\
We will use the letter $\I$ for the groupoid with two distinct points and one isomorphism between them and we will denote by $\cI$ the same groupoid equipped with the involution that swaps the two points and maps the non-identity isomorphism to its inverse. When we need to refer to its elements we will use the symbols 0 and 1 and $\phi$ for its isomorphism. We will denote the obvious pushout $\bold{\I}\textstyle\coprod\limits_{\pI}\bold{\I}$ in $\Gpd$ by $\P$. \\
We will denote by $\tr$ the groupoid with involution which extends $\cI$ and its involution by a fixed third point denoted 2 and a second non-identity isomorphism between 1 and 2 denoted $\psi$ whose image by the involution is $\psi\circ\phi$ (there are only two non-identity isomorphisms and their composition).\\
Last, being given a groupoid $A$ equipped with an involution we will denote by 
$A\f$ the full subgroupoid of $A$ consisting in the fixed points.

\setchapterpreamble{
\begin{quote}
Together with chapter 3 this chapter 2 provides all the background required for this thesis, though without too much details. The reader familiar with model categories should feel free to skip this chapter. 
\end{quote}
}

\chapter{Background on model categories}
\label{sec:somc}

\section{Organization}

In section \ref{sec:motiv} we say a few words about the relevance of model structure with respect to type theory and the connection between homotopy theory and type theory is developed in section \ref{sec:connection}. The sections \ref{sec:basics} and \ref{sec:morebasics} present the basics of model categories. The section \ref{sec:msofc} presents two model structures on functor categories used in the rest of this thesis namely the injective and projective model structures.

\section{Motivations}
\label{sec:motiv}

As motivations the main points are:
\begin{itemize}
\item Model Categories are ubiquitous in Mathematics and they are a powerful tool for using homotopical methods in some abstract setting.\\
\item Any model category $\mathscr{C}$ that satisfies a few additional properties, \textit{i.e.} a type-theoretic model category, provides a model of MLTT with $\ssum,~\sprod{}$, \Id-types by taking the full subcategory of fibrant objects.\\
\item Some key ingredients of type theory (dependent types especially \Id-types, the J-rule...) have a counterpart in some key ingredients of model category (fibrations, path objects, diagonal filler \ldots).\\
\end{itemize}

\section{Basics of model categories}
\label{sec:basics}

This section is based on \cite{dwyer95theories}. The reader should also consult the seminal work of Daniel Quillen, the architect of model categories, for instance his lecture notes \cite{quillen:htpical-alg}. The books by Hirschhorn \cite{hir99} and Hovey \cite{hovey99} on model categories are also very useful references.
\begin{defi}\textbf{(diagonal filler)}\label{defi1}
Given a commutative diagram of the form $$\xymatrix{A\ar[rr]^f \ar[dd]_i&&X\ar[dd]^p\\&\circlearrowright&\\B\ar[rr]_g&&Y}$$ a diagonal filler is a map $h:B\rightarrow X$ such that \begin{align*} p\circ h=g~\\h\circ i=f~.
\end{align*} The maps $i$ and $p$ being fixed if there is a diagonal filler for any commutative diagram of the form above, we say that : $i$ has the left lifting property (LLP) with respect to $p$. Conversely we say that $p$ has the right lifting property (RLP) with respect to $i$.\\
A diagonal filler $h$ is also called a lift (of $g$ along $p$).
\end{defi}
\begin{defi}\textbf{(retract)}\label{defi2}
Let $\mathscr{C}$ be a category and $f,~g$ two morphisms of $\mathscr{C}$. One says that $f$ is a retract of $g$ if there is a commutative diagram $$\xymatrix{X\ar[rr]^r \ar[dd]_f&&Y\ar[dd]^g \ar[rr]^s&&X\ar[dd]^f\\&\circlearrowright&&\circlearrowright&\\X'\ar[rr]_{r'}&&Y'\ar[rr]_{s'}&&X'}$$ such that \begin{align*} s\circ r=1_{X}\quad\\s'\circ r'=1_{X'}~.
\end{align*}\\
\end{defi}
\begin{defi}\textbf{(model category)}\label{defi3}
A model category is a category $\mathscr{C}$ with three classes of maps : \begin{enumerate}[label=(\roman*)] \item the class of weak equivalences denoted $\mathcal{W}$ (in picture $\xrightarrow{\sim}$)\\
\item the class of fibrations denoted $\mathcal{F}$ (in picture $\fib$)\\
\item the class of cofibrations denoted $\mathcal{C}$ (in picture $\hookrightarrow$).\\
\end{enumerate} Maps in $W\cap \mathscr{F}$ are called trivial (or acyclic) fibrations.\\ 
Maps in $\mathcal{W}\cap \mathcal{C}$ are called trivial (or acyclic) cofibrations.\\\\
We require the following axioms : \\\\
MC1 All (small) limits and (small) colimits exist in $\mathscr{C}$.\\\\
MC2 ($2$ out of $3$) If $f$ and $g$ are morphisms of $\mathscr{C}$ such that $g\circ f$ is defined and two of $f,~g,~g\circ f$ are weak equivalences then so is the third.\\\\
MC3 (retract) If $f$ is a retract of $g$ and $g\in\mathcal{F}$ (\textit{resp.} $g\in \mathcal{W}$, \textit{resp.} $g\in \mathcal{C}$) then $f\in\mathcal{F}$ (\textit{resp.} $f\in \mathcal{W}$, \textit{resp.} $f\in \mathcal{C}$).\\\\
MC4 (lift) Given a commutative diagram as in \ref{defi1} a diagonal filler exists in either of the two situations : \begin{enumerate}[label=(\roman*)]
\item $i\in \mathcal{C}~\text{and}~p\in \mathcal{F}\cap \mathcal{W}$\\
\item $i\in \mathcal{C}\cap \mathcal{W}~\text{and}~p\in \mathcal{F}$.\\
\end{enumerate}
MC5 (factorization) Any map $f$ can be factored in two ways : \begin{enumerate}[label=(\roman*)]
\item $f=pi~\text{with}~i\in \mathcal{C} ~\text{and}~p\in\mathcal{F}\cap \mathcal{W}$\\
\item $f=pi~\text{with}~i\in \mathcal{C}\cap \mathcal{W}~\text{and}~p\in \mathcal{F}$.\\
\end{enumerate}
\end{defi}

\begin{rmk}\label{Duality principle}
We have the following \textbf{duality principle}. Let $P$ be a statement about model categories and $P^*$ be the statement obtained by reversing the arrows in $P$ and switching \enquote{cofibration} with \enquote{fibration}. If $P$ is true for all model categories then so is $P^*$.
\end{rmk}

\begin{rmk}\label{rmk1}
The duality principle follows from the fact that the axioms for a model category are self-dual (where \enquote{dual} has to be taken in the sense above).\\
\end{rmk}

\begin{prop}\textbf{(the retract argument)}\label{prop1}
Let $\mathscr{C}$ be a category. Assume we have a factorization $f=pi$ in $\mathscr{C}$ and assume that $f$ has the LLP with respect to $p$ then $f$ is a retract of $i$. In the same way if $f=pi$ has the RLP with respect to $i$ then $f$ is a retract of $p$.\\
\end{prop}
\begin{proof}\label{proofprop1}
Assume $f=pi$ and $f$ has the LLP with respect to $p$. One has the following diagram, $$
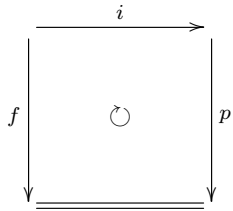
\xymatrix{\ar[rr]^i \ar[dd]_f&&\ar[dd]^p\\&\circlearrowright&\\ \ar@{=}[rr]&&}$$ so one has a lift $j$, 
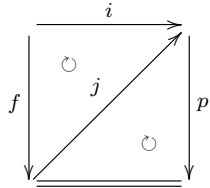
$$\xymatrix{\com1 \ar[rr]^i \ar[dd]_f&&\ar[dd]^p\\&\com1&\\ \ar@{=}[rr] \ar[rruu]^j&&}$$ hence $f$ is a retract of $i$, 
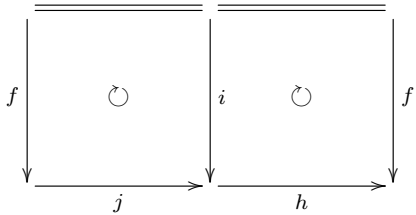
$$\xymatrix{\ar@{=}[rr] \ar[dd]_f&&\ar@{=}[rr] \ar[dd]^i&&\ar[dd]^f\\&\circlearrowright&&\circlearrowright&\\ \ar[rr]_j&&\ar[rr]_h&&}$$ The proof when $f$ has the RLP with respect to $i$ is similar.
\end{proof}\bigskip

\begin{defi}\label{weaklyortho}\textbf{(weakly orthogonal)}\label{defi4}
Let $\mathscr{C}$ be a category and $A$ be a class of morphisms of $\mathscr{C}$. \begin{align*} {}^{\boxslash}A=\lbrace{i~|\forall p\in \text{Mor}(A),~i~\text{has the LLP with respect to}~p}\rbrace~\\A^{\boxslash}=\lbrace{p~|\forall i\in \text{Mor}(A),~p~\text{has the RLP with respect to}~i}\rbrace.\\
\end{align*}
\end{defi}

\begin{prop}\label{2wfs}
Assume that $\mathscr{C}$ is a model category. One has : \begin{enumerate}[label=(\roman*)] \item $C={}^{\boxslash}(\mathcal{F}\cap \mathcal{W})~\text{and}~(\mathcal{C}\cap \mathcal{W})={}^{\boxslash}\mathcal{F}$\\
\item $\mathcal{F}=(\mathcal{C}\cap \mathcal{W})^{\boxslash}~\text{and}~(\mathcal{F}\cap \mathcal{W})=\mathcal{C}^{\boxslash}$.\\
\end{enumerate}
\end{prop}
\begin{proof}\label{proof2wfs}
We prove \textit{(i)}. The statement \textit{(ii)} follows by duality.\\
Let us prove the equality $C={}^{\boxslash}(\mathcal{F}\cap \mathcal{W})$. By CM4 one has $C\subseteq{}^{\boxslash}(\mathcal{F}\cap \mathcal{W})$. Conversely, let $f$ be a map that has LLP with respect to any trivial fibration. Then by MC5 \textit{(i)} $f=pi$ with $i\in \mathcal{C}$ and $p\in (\mathcal{F}\cap \mathcal{W})$, so by the retract argument $f$ is a retract of $i$. From MC3 one concludes $f\in \mathcal{C}$. The proof of the equality $(\mathcal{C}\cap \mathcal{W})={}^{\boxslash}\mathcal{F}$ is similar.
\end{proof}\bigskip

\begin{rmk}\label{rmk2}
The proposition above shows that in a model category any two of the three classes determine the third one. Indeed, if $\mathcal{F}\text{and}~\mathcal{W}$ are determined then one has $\mathcal{C}={}^{\boxslash}(\mathcal{F}\cap \mathcal{W})$; in the same way, if $\mathcal{C}$ and $\mathcal{W}$ are determined then one has $\mathcal{F}=(\mathcal{C}\cap \mathcal{W})^{\boxslash}$. Last, if $\mathcal{F}\text{and}~\mathcal{C}$ are determined then $\mathcal{W}$ is exactly the class of maps $f$ that factors as $pi$ with $i\in \mathcal{C}\cap \mathcal{W}$ and $p\in \mathcal{F}\cap \mathcal{W}$. Indeed, take $f=pi$ with $i\in \mathcal{C}\cap \mathcal{W}$ and $p\in \mathcal{F}\cap \mathcal{W}$, then by the $2$ out of $3$ property $f$ belongs to $\mathcal{W}$. Conversely, if $f$ belongs to $\mathcal{W}$ then by MC5 \textit{(ii)} one has $f=pi$ with $i\in \mathcal{C}\cap \mathcal{W}$ and $p\in \mathcal{F}$, but by the $2$ out of $3$ property $p$ belongs to $\mathcal{W}$ and so $p$ belongs to $\mathcal{F}\cap \mathcal{W}$. It means that the axioms for a model category are overdetermined. The semantics of Type Theory can be seen as a way to relax these axioms.\\
\end{rmk}

\begin{defi}\textbf{(weak factorization system)}\label{defi5}
A weak factorization system (WFS) on a category $\mathscr{C}$ is given by two classes $A,B$ of morphisms of $\mathscr{C}$ such that $$A={}^{\boxslash}B~\text{and}~B=A^{\boxslash}$$ and every morphism of $\mathscr{C}$ factors as a morphism in $A$ followed by a morphism in $B$.\\
\end{defi}
\begin{defi}\textbf{(concise definition of a model category)}\label{defi6}
A model category is a complete and cocomplete category $\mathscr{C}$ with three classes of maps $\mathcal{F},~\mathcal{C},~\mathcal{W}$ such that :\\\\
(WFS) $(\mathcal{C},~\mathcal{F}\cap \mathcal{W})~\text{and}~(\mathcal{C}\cap \mathcal{W},~\mathcal{F})$ form two weak factorization systems.\\\\
($2$ out of $3$) Given composable morphisms $f,g$ then as soon as two of the three morphisms in $\lbrace{f,~g,~g\circ f}\rbrace$ are in $\mathcal{W}$ so is the third one.\\
\end{defi}
\begin{rmk}\label{rmk3}
Thanks to \ref{2wfs} we have already proven that the first definition implies the concise one. It is pretty straightforward to prove that the concise definition implies the first one. Let us prove the less obvious fact in this implication, namely that a retract of a fibration (\textit{resp.} a cofibration, \textit{resp.} a weak equivalence) is a fibration (\textit{resp.} a cofibration, \textit{resp.} a weak equivalence). We do the job for fibrations, trivial fibrations and weak equivalences for instance. \\
First, we treat the case of fibrations. Assume $f$ is a retract of $g$ and $g$ is a fibration, $$\xymatrix{X\ar[rr]^r \ar[dd]_f&&Y\ar@{->>}[dd]^g \ar[rr]^s&&X\ar[dd]^f\\&\circlearrowright&&\circlearrowright&\\X'\ar[rr]_{r'}&&Y'\ar[rr]_{s'}&&X'}$$ with $s\circ r=1_{X}~\text{and}~s'\circ r'=1_{X'}$, since $\mathcal{F}=(\mathcal{C}\cap \mathcal{W})^{\boxslash}$ we need to check that $f$ lifts with respect to any trivial cofibration. Being given a lifting problem, $$\xymatrix{\underset{~}{Z}\ar[rr]^t \ar@{ >->}[dd]_*[@]{\hbox to 4pt{$\sim$}}_-{h\quad}&&X\ar[dd]^f\\&\circlearrowright&\\Z'\ar[rr]_u &&X'}$$ one has $$\xymatrix{\underset{~}{Z}\ar[rr]^t \ar@{ >->}[dd]_*[@]{\hbox to 4pt{$\sim$}}_-{h\quad}&&X\ar[rr]^r \ar[dd]^f&&Y\ar@{->>}[dd]^g\\&\circlearrowright&&\circlearrowright&\\Z'\ar[rr]_u&&X'\ar[rr]_{r'}&&Y'}$$ hence we have a lift, $$\xymatrix{\underset{~}{Z}\com1 \ar[rr]^{r\circ t} \ar@{ >->}[dd]_*[@]{\hbox to 4pt{$\sim$}}_-{h\quad}&&Y\ar@{->>}[dd]^g\\&\com1&\\Z'\ar[rruu]^j \ar[rr]_{r'\circ u}&&Y'\,.}$$ So $s\circ j$ is the lift we are looking for.\\
Second, we deal with the case of trivial fibrations. Assume that  $f$ is a retract of $g$ and $g$ is a trivial fibration, $$\xymatrix{X\ar[rr]^r \ar[dd]_f&&Y\ar@{->>}[dd]^g \ar[rr]^s&&X\ar[dd]^f\\&\circlearrowright&&\circlearrowright&\\X'\ar[rr]_{r'}&&Y'\ar[rr]_{s'}&&X'\,.}$$
Since $(\mathcal{F}\cap\mathcal{W})=\mathcal{C}^{\boxslash}$, we want to prove that $f$ belongs to $\mathcal{C}^{\boxslash}$. Consider the following lifting problem, 
$$\xymatrix@=1,5cm{Z\ar[r]^t \ar@{>->}[d]_h & X\ar[d]^f \\ Z'\ar[r]_u & X'\,. }$$
Since $g$ is a trivial fibration, the following lifting problem can be filled with a diagonal filler $j$,
$$\xymatrix@=1,5cm{Z\ar[r]^t\ar@{>->}[d]_h & X \ar[r]^r & Y\ar@{->>}[d]_*[@]{\hbox to 4pt{$\sim$}}^-{g\quad} \\ Z'\ar[r]_u\ar@{-->}[rru]^j & X'\ar[r]_{r'} & Y'\,.}$$
The reader can check that $s\circ j$ is a diagonal filler for the initial lifting problem.\\
Last, we work out the case of weak equivalences. Assume that  $f$ is a retract of $g$ and $g$ is a weak equivalence, $$\xymatrix{X\ar[rr]^r \ar[dd]_f&&Y\ar@{->>}[dd]^g \ar[rr]^s&&X\ar[dd]^f\\&\circlearrowright&&\circlearrowright&\\X'\ar[rr]_{r'}&&Y'\ar[rr]_{s'}&&X'\,.}$$
We proceed in two steps. The first step deals with the case where we have the additional assumption that $f$ is a fibration. In this case we want to prove  that $f$ is a trivial fibration. Note that the axiom (WFS) implies MC5 (i). So by MC5 (i) we can factor $g$ as $p\circ i$, where $i$ is a cofibration and $p$ is a trivial fibration. Note that by the axiom (2 out of 3) the morphism $i$ is a trivial cofibration. Hence a diagonal filler $j$ exists in the following lifting problem,
$$\xymatrix@=1,5cm{Y\ar[rr]^s\ar@{>->}[d]_*[@]{\hbox to 4pt{$\sim$}}_-{i\quad} && X\ar@{->>}[d]^f \\ Y''\ar@{->>}[r]_p\ar@{-->}[rru]^j & Y'\ar[r]_{s'} & X'\,.}$$
So we have the following retract diagram,
$$\xymatrix@=1,5cm{X\ar[r]^{i\circ r}\ar[d]_f & Y''\ar[r]^j \ar@{->>}[d]_*[@]{\hbox to 4pt{$\sim$}}^-{p\quad} & X\ar[d]^f \\ X'\ar[r]_{r'} & Y'\ar[r]_{s'} & X'}$$
by the previous case one concludes that $f$ is a trivial fibration. \\
The second step is the general case where $f$ is any morphism but of course we still assume that f is a retract of the weak equivalence $g$. Note that (WFS) implies MC5 (ii), hence one can factor $f$ as $p\circ i$, where $i$ is a trivial cofibration and $p$ is a fibration. In order to prove that $f$ is a weak equivalence, it suffices to prove that $p$ is a trivial fibration. Consider $t$ the pushout of $i$ along $r$. By two applications of the universal property of the pushout, one gets two maps denoted below by $v$ and $w$,
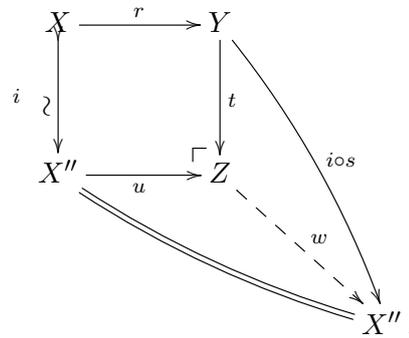
$$\xymatrix@=1,5cm{X\ar[r]^r\ar@{>->}[d]_*[@]{\hbox to 4pt{$\sim$}}_-{i\quad} & Y\ar[d]^t\ar@/
^/[rdd]^g & \\
X''\ar[r]_u\ar@/_/[rrd]_{r'\circ p} & Z\pushoutcorner \ar@{-->}[rd]^v & \\
& & Y'}$$ 
and 
$$\xymatrix@=1,5cm{X\ar[r]^r\ar@{>->}[d]_*[@]{\hbox to 4pt{$\sim$}}_-{i\quad} & Y\ar[d]^t\ar@/
^/[rdd]^{i\circ s} & \\
X''\ar[r]_u\ar@{=}@/_/[rrd] & Z\pushoutcorner\ar@{-->}[rd]^w & \\
& & X''\,.}$$
Now, we can display the initial retract diagram into four squares as follows,
$$
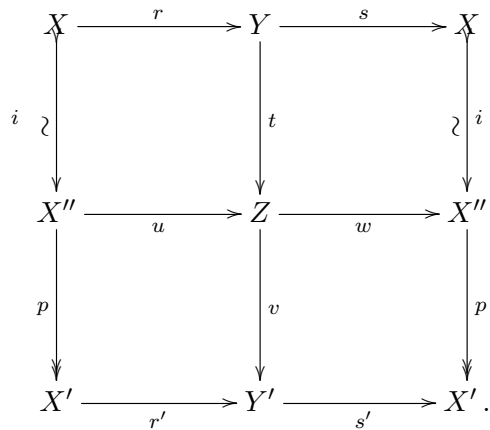
\xymatrix@=2cm{X\ar[r]^r\ar@{>->}[d]_*[@]{\hbox to 4pt{$\sim$}}_-{i\quad} & Y\ar[r]^s\ar[d]^t & X\ar@{>->}[d]_*[@]{\hbox to 4pt{$\sim$}}^-{i\quad} \\
X''\ar[r]_u\ar@{->>}[d]_p & Z\ar[r]_w\ar[d]^v & X''\ar@{->>}[d]^p \\
X'\ar[r]_{r'} & Y'\ar[r]_{s'} & X'\,.}$$
Lastly, note successively that first the south-est square commutes thanks to the uniqueness part of the universal property of the pushout applied to the following universal problem,
$$
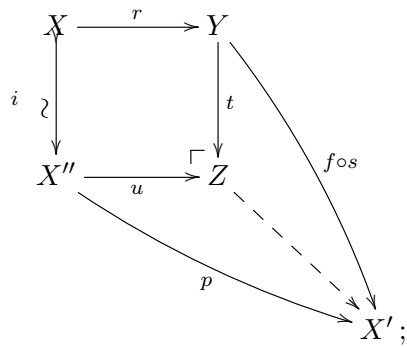
\xymatrix@=1,5cm{X\ar[r]^r\ar@{>->}[d]_*[@]{\hbox to 4pt{$\sim$}}_-{i\quad} & Y\ar[d]^t\ar@/
^/[rdd]^{f\circ s} & \\
X''\ar[r]_u\ar@/_/[rrd]_p & Z\pushoutcorner\ar@{-->}[rd] & \\
& & X'\,;}$$
second $t$ is a trivial cofibration since this is the pushout of a trivial cofibration (the proof of this fact will be given later in \ref{pullbackpushoutstability}). Since one has $v\circ t$ equals $g$, by the (2 out of 3) axiom one concludes that $v$ is a weak equivalence. Hence one exhibits the fibration $p$ as a retract of the weak equivalence $v$, so by the first step $p$ is a trivial fibration.  

\end{rmk}\bigskip

\begin{eg}\label{eg1} 
Any complete and cocomplete category $\mathscr{C}$ can be provided with three model structures by choosing one of the distinguished classes of morphisms to be all the isomorphisms and the other two to be all morphisms of $\mathscr{C}$.\\
\end{eg}
\begin{eg}\label{eg2} On $\mathbf{Set}$:\\
\begin{itemize} \item Take for $\mathcal{W}$ all the maps.\\ \item Take for $\mathcal{F}$ the monomorphisms (\textit{i.e.} the injective maps).\\ \item Take for $\mathcal{C}$ the epimorphisms (\textit{i.e.} the surjective maps).\\
\end{itemize}
\end{eg}
\begin{eg}\label{eg3} On $\mathbf{Gpd}$ there is a natural model structure (this is the model structure on $\mathbf{Gpd}$ under consideration in this thesis) : \\
\begin{itemize} \item Take for $\mathcal{W}$ the equivalences of groupoids.\\ \item Take for $\mathcal{F}$ the isofibrations, where a functor $F:G\rightarrow H$ between two groupoids $G$ and $H$ is an isofibration if for any isomorphism $h:y\rightarrow y'$ in $H$ and any element $x\in G$ such that $F(x)=y$ there exists an isomorphism $g~\text{in}~G$ with domain $x$ such that $F(g)=h$. In picture, $$\xymatrix{x\ar@{-->}[rr]^g&\underset{~}{}\ar[dd]^F&x'\\\\y\ar[rr]_h&\overset{~}{}&y'}$$\\
\item Take for $C$ the functors that are injective on objects.\\
\end{itemize}
Note this model structure is cofibrantly generated (see \cite{hir99} def 11.1.1 for a definition of a cofibrantly generated model category and \cite{rezk00} 4 for a proof that $\Gpd$ is cofibrantly generated). The obvious inclusion :\\
$$i:\b1\hookrightarrow\I$$\\
forms a set of generating trivial cofibrations.\\
Moreover, together the following three morphisms of groupoids amount to a set of generators for cofibrations :\\ 
\begin{align*}
&{u:\bold{0} \rightarrow \b1}\\&{v:\bold{1}\textstyle\coprod\bold{1}\hookrightarrow \I}\\&{w:\P\rightarrow\I\,.}\\
\end{align*}
For details on this central example for us the reader can look at \cite{strickland} 6.1 and \cite{rezk00}.
\end{eg}

\begin{eg}\label{eg4} On $\mathbf{Top}$ : \\
\begin{itemize} \item Take $\mathcal{W}$ to be the weak homotopy equivalences where a map $f:X\rightarrow Y$ is a weak homotopy equivalence if $f_{*}:\pi_{0}(X)\rightarrow \pi_{0}(Y)$ is a bijection of the sets of path components and for each basepoint $x\in X$ and each $n\geqslant 1$ the morphism $f_{*}=\pi_{n}(X,x)\rightarrow \pi_{n}(Y,f(x))$ is an isomorphism of groups.\\ \item Take for $\mathcal{F}$ the Serre fibrations, where a map $f:X\rightarrow Y$ is a Serre fibration if for each $CW$-complex $A$, the morphism $f$ has the RLP with respect to the inclusion $A\times 0\rightarrow A\times [0,1]$.\\ \item Take for $\mathcal{C}$ the morphisms with the LLP with respect to trivial fibrations.\\
\end{itemize}
\end{eg}

\begin{eg}\label{eg5} On $\mathbf{sSet}$ : \\
\begin{itemize} \item Take for $\mathcal{W}$ the morphisms $f$ of simplicial sets such that $|f|$ is a weak homotopy equivalence, where $|~|$ is the functor of geometric realization.\\ \item Take for $\mathcal{C}$ the morphisms $f:X\rightarrow Y$ such that for all $n\geqslant0,~f_{n}:X_{n}\rightarrow Y_{n}$ is  a monomorphism.\\ \item Take for $\mathcal{F}$ the morphisms with the RLP with respect to trivial cofibrations.\\ 
\end{itemize}
This model structure on $\mathbf{sSet}$ due to Quillen is cofibrantly generated. A set 
of generating cofibrations is given by the boundary inclusions:
\begin{center}
$\partial \Delta^{n} \to \Delta^{n}$.
\end{center}
A set of generating trivial cofibrations is given by the horn inclusions:
\begin{center}
$\Lambda^{n}_{k} \to \Delta^{n}$.
\end{center}

\end{eg}
\begin{prop}\label{pullbackpushoutstability}
Let $\mathscr{C}$ be a model category. \begin{enumerate}[label=(\roman*)] \item $\mathcal{F}$ (\textit{resp.} $\mathcal{F}\cap \mathcal{W}$) is stable under pullback.\\
\item $\mathcal{C}$ (\textit{resp.} $\mathcal{C}\cap \mathcal{W}$) is stable under pushout.\\
\end{enumerate}
\end{prop}
\begin{proof}\label{proofpullbackpushoutstability} Let us prove \textit{(i)}. The property \textit{(ii)} follows by duality. Let $g^{*}f$ be the pullback of a fibration $f$ along a morphism $g$, $$\xymatrix{Z\times_{Y}X\pullbackcorner \ar[rr] \ar[dd]_{g^{*}f}&&X\ar@{->>}[dd]^f\\\\Z\ar[rr]_g&&Y\,.}$$ Being given a lifting problem as follows, \def\commutatif{\ar@{}[rrdd]|{\circlearrowright}} \newdir{ >}{{}*!/-5pt/\dir{>}} $$\xymatrix{\underset{~}{V}\commutatif \ar[rr] \ar@{ >->}[dd]_*[@]{\hbox to 4pt{$\sim$}}_-{h\quad}&&Z\times_{Y}X \ar[dd]^{g^{*}f}\\\\W\ar[rr]&&Z\,,}$$ one has a lift, $$\xymatrix{\underset{~}{V} \ar[rr] \ar@{ >->}[dd]_*[@]{\hbox to 4pt{$\sim$}}_-{h\quad}&&Z\times_{Y}X \ar[rr]&&X\ar@{->>}[dd]^f\\\\W\ar[rr] \ar[rrrruu]^{j'}&&Z\ar[rr]&&Y\,.}$$ So using the universal property of the pullback square, one gets a map $j$ as follows, $$\xymatrix{\underset{~}{V} \ar[rr] \ar@{ >->}[dd]_*[@]{\hbox to 4pt{$\sim$}}_-{h\quad}&&Z\times_{Y}X \pullbackcorner \ar[dd]^{g^{*}f} \ar[rr]&&X\ar@{->>}[dd]^f\\\\W\ar[rr] \ar[rruu]^j&&Z\ar[rr]&&Y\,.}$$\\
The reader can easily check that $j$ is a lift for the initial lifting problem.\\
The proof follows the same pattern if the map $f$ one starts with is a trivial fibration then $g^{*}f$ is a trivial fibration.
\end{proof}\medskip
\begin{defi}\textbf{(path object)}\label{defi7} Let $\mathscr{C}$ be a model category and $X$ be an object of $\mathscr{C}$. A (very good) path object for $X$ is an object of $\mathscr{C}$ denoted $PX$ together with a commutative diagram : \newdir{ >}{{}*!/-5pt/\dir{>}} $$\xymatrix{X~\ar@{ >->}[rr]^{\sim} \ar[rrdd]_{\Delta}&&PX\ar@{->>}[dd]\\&&\circlearrowright\qquad\qquad\\&&X\times X}$$ where $\Delta=<id_{X},id_{X}>$.\\
\end{defi}
\begin{rmk}\label{rmk4}
For any $X\in \mathscr{C}$ a path object for $X$ always exists by MC5 \textit{(ii)}.\\
\end{rmk}

\section{The connection with intensional type theory}
\label{sec:connection}

This section relies on \cite{shulman:invdia} and develops in some details the connection with intensional type theory. Recall from 4.2 in \cite{shulman:invdia} that $\m{-}$ denotes the universal morphism from the syntactic category to any model.
One works out the connection with type theory as follows.\\
For $\Gamma\pr A\ty$, the identity type $$\Gamma,(x:A),(y:A)\pr Id_{A}(x,y)\ty$$ is interpreted by a path object $$P_{\m{\Gamma}}\m{\Gamma.A}\twoheadrightarrow \m{\Gamma.A}\times_{\m{\Gamma}}\m{\Gamma.A}$$ \textit{i.e.} a path object for $\m{\Gamma.A}\xrightarrow{\m{A}} \m{\Gamma}$ in the slice category $\mathscr{C}/\m{\Gamma}$, in picture :\\\\
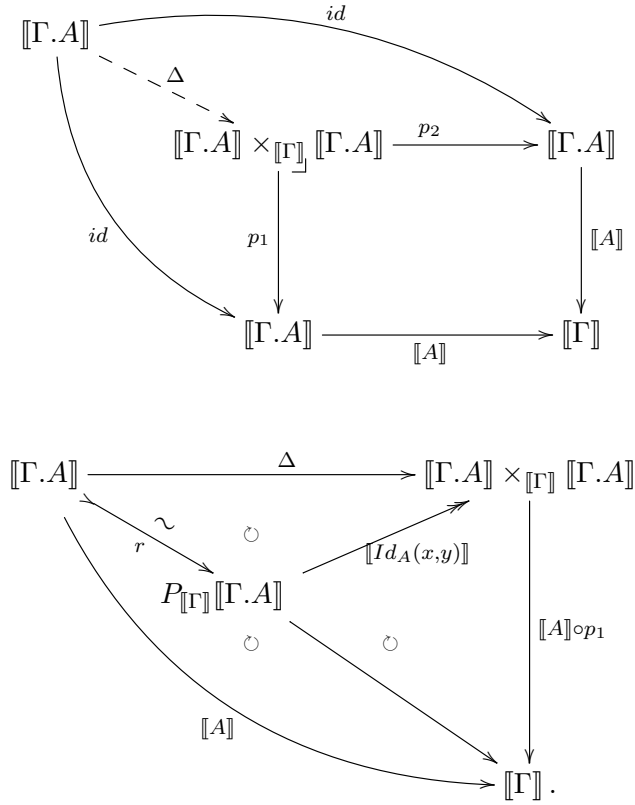
$$\xymatrix{{\m{\Gamma.A}} \ar@{-->}[rd]^{\Delta} \ar@/^2pc/[rrrd]^{id} \ar@/_2pc/[rddd]_{id}&&&\\&{\m{\Gamma.A}\times_{\m{\Gamma}}\m{\Gamma.A}} \ar[rr]^{p_{2}} \ar[dd]_{p_{1}} \pullbackcorner&&{\m{\Gamma.A}} \ar[dd]^{\m{A}}\\\\&{\m{\Gamma.A}} \ar[rr]_{\m{A}}&&{\m{\Gamma}}}$$\\
\def\CD{\ar@{}[d]|{\circlearrowright}}
$$\xymatrix{{\underset{~}{\m{\Gamma.A}}} \ar[rr]^{\Delta} \ar@/_3pc/[rrddd]_{\m{A}} \ar@{ >->}[rd]^[@]{\hbox to 1pt{$\sim$}}_r&\CD&{\m{\Gamma.A}\times_{\m{\Gamma}}\m{\Gamma.A}} \ar[ddd]^{\m{A}\circ p_{1}}\\&{P_{\m{\Gamma}}\m{\Gamma.A}}\qquad\com1 \ar@{->>}[ru]_{\m{Id_{A}(x,y)}} \ar[rdd] \CD&\\&&\\&&{\m{\Gamma}}\,.}$$\\
The reflexivity constructor is interpreted by the acyclic cofibration $r:\xymatrix{\m{\Gamma.A}~~\ar@{ >->}[r]^-{\sim}&P_{\m{\Gamma}}\m{\Gamma.A}}$. Indeed, $r$ is a section of $$\m{\Gamma,(x:A)\pr r_{x}:Id_{A}(x,x)}~,$$ this last fibration being interpreted by the following pullback :\\
$$\xymatrix{{\Delta^{*}P_{\m{\Gamma}}\m{A}} \pullbackcorner \ar[rr] \ar@{->>}[dd]&&{P_{\m{\Gamma}}\m{\Gamma.A}} \ar@{->>}[dd]^{\m{Id_{A}(x,y)}}\\\\{\m{\Gamma.A}} \ar[rr]_-{\Delta}&&{\m{\Gamma.A}\times_{\m{\Gamma.A}}\m{\Gamma.A}}}$$ this pullback is (isomorphic to)\\
$$\xymatrix{{P_{\m{\Gamma}}\m{\Gamma.A}} \ar@{=}[rr] \ar@{->>}[dd]_{p_{1}\circ \m{Id_{A}(x,y)}}&&{P_{\m{\Gamma}}\m{\Gamma.A}} \ar@{->>}[dd]^{\m{Id_{A}(x,y)}}\\\\{\m{\Gamma.A}} \ar[rr]_-{\Delta}&&{\m{\Gamma.A}\times_{\m{\Gamma.A}}\m{\Gamma.A}}}$$ and by the uniqueness condition in the universal property of the pullback we prove that $r$ is indeed a section of $p_{1}\circ \m{Id_{A}(x,y)}=\m{\Gamma,(x:A)\pr r_{x}:Id_{A}(x,x)}$\\
$$\xymatrix{{\m{\Gamma.A}} \ar@{-->}[rd]^r \ar@/^2pc/[rrrd]^r \ar@/_4pc/[rddd]_{id}&&&\\&{P_{\m{\Gamma}}\m{\Gamma.A}} \ar@{=}[rr] \ar@{->>}[dd]_{p_{1}\circ \m{Id_{A}(x,y)}}&&{P_{\m{\Gamma}}\m{\Gamma.A}} \ar@{->>}[dd]^{\m{Id_{A}(x,y)}}\\\\&{\m{\Gamma.A}} \ar[rr]_-{\Delta}&&{\m{\Gamma.A}\times_{\m{\Gamma.A}}\m{\Gamma.A}}\,.}$$\\\\
In particular, on has the J-rule for identity types :\begin{align*} {\Gamma,(x:A),(y:A),(p:Id_{A}(x,y)),\Theta \pr B \ty}\quad\\ {\underline{\Gamma,(x:A),\Theta[x/y,{r_x/p] \pr d:B[x/y,{r_x/p]}}}}\qquad\\ {\Gamma,(x:A),(y:A),(p:Id_{A}(x,y)),\Theta \vdash J_d(x,y,p):B}.
\end{align*}\\
The interpretation of this rule is a lift in the following square,
$$\xymatrix{\underset{~}{\m{\Gamma.A.\Theta[x/y,\Gamma_{x}/p]}}\ar[rr]^d \ar@{ >->}[dd]_{\m{r}}&&{\m{\Gamma.A.A.P_{\Gamma}A.\Theta.B}} \ar@{->>}[dd]^{P_{B}}\\\\{\m{\Gamma.A.A.P_{\Gamma}A.\Theta}} \ar@{=}[rr] \ar@{-->}[rruu]_{\m{Id}}&&{\m{\Gamma.A.A.P_{\Gamma}A.\Theta}}\,.}$$ The left-hand morphism is a pullback of the acyclic cofibration $$\xymatrix{r:\m{P.A}~~ \ar@{ >->}[r]^{\sim}&P_{\m{\Gamma}}\m{\Gamma.A}}$$ hence is an acyclic cofibration (since in a model of type theory we require that acyclic cofibrations are stable under pullback along a fibration, actually the square above is the \textit{raison d'être} of this requirement).\\

\section{More basics of model categories}
\label{sec:morebasics}\bigskip

Through this section we present more basics of model categories and we use the natural model structure on $\Gpd$ introduced in \ref{eg3} as a leitmotiv in order to exemplify the notions introduced and for the acquaintance of the reader with these notions.\bigskip

\subsection{Homotopy equivalences}

\begin{defi}\textbf{(right homotopy)}\label{defi8} Two morphisms $f,g:A\rightarrow X~\text{in}~\mathscr{C}$ are said to be right homotopic (written $f\overset{r}{\sim} g$) if there exists a path object $PX$ for $X$ and a morphism $H:A\rightarrow PX$ such that $$\xymatrix{&&PX\ar@{->>}[dd]\\&&\circlearrowright\qquad\qquad\\A\ar[rr]_{<f,g>} \ar[rruu]^H&&X\times X}$$ commutes.\\
\end{defi}\bigskip

\begin{rmk}
Note that the definition of right homotopy given above does not depend on the choice of a path object. Indeed, the reader can notice that every path object factors through any other. Let $P'X$ (and the corresponding morphisms that factor the diagonal morphism) be any other path object for $X$, we can form the following lifting problem and denote by $j$ a diagonal filler,
$$\xymatrix@=1,5cm{X\ar@{ >->}[r]^*[@]{\hbox to 4pt{$\sim$}}\ar@{ >->}[d]_*[@]{\hbox to 4pt{$\sim$}} & P'X\ar@{->>}[d]\\
PX\ar@{->>}[r]\ar@{-->}[ru]_j & X\times X\,.}$$
Hence the morphism $j\circ H$ is a right homotopy,
$$\xymatrix{ & & P'X\ar@{->>}[dd]\\
 & PX\ar@{-->}[ru]^j & \\
A\ar[rr]_{<f,g>}\ar[ru]^H & & X\times X\,.}$$
\end{rmk}\bigskip

\begin{rmk}\label{rmk6}
We have the dual notions of cylinder object (dual to path object) and left homotopy $\overset{l}{\sim}$ (dual to right homotopy, see \cite{dwyer95theories} 4.1) but there are not relevant for type theory.\\
\end{rmk}\bigskip

\begin{prop}
If $X$ is fibrant (meaning that the unique morphism from $X$ to the terminal object $\bold{1}$ is a fibration) then the relation $\overset{r}{\sim}$ is an equivalence relation on $\mathscr{C}(A,X)$.
\end{prop}\bigskip

\begin{defi}\textbf{(homotopy equivalence)}\label{defi9}
Let $\mathscr{C}$ be a model category. A morphism $f:X\rightarrow Y~\text{in}~\mathscr{C}$ is a homotopy equivalence if there exists $g:Y\rightarrow X~\text{in}~\mathscr{C}$ and (right) homotopies from $f\circ g$ to $1_{Y}$ and from $g\circ f$ to $1_{X}$.\\
\end{defi}\bigskip

\begin{rmk}\label{rmk7}
Categorically the univalence axiom requires some morphism to be a homotopy equivalence as explained in the section \ref{sec:catunivaxiom}.\\
\end{rmk}\bigskip

\subsection{Transfinite sequences}

\begin{defi}\textbf{($\lambda$-sequence)}\label{defi10}
Let $\mathscr{C}$ be a category and $\la$ any ordinal. A $\la$-sequence in $\mathscr{C}$ is a cocontinuous functor $F$ from $[0,\la]$ to $\mathscr{C}$, where $[0,\la]$ is the closed interval of ordinals up to $\la$.\\
We denote the value of $F$ at $\be$ by $F_\be$ for any $\be\leqslant\la$ and the value of $F$ at the unique morphism from $\be$ to $\gm$ by $F_{\be,\gm}$ for every $\be\leqslant\gm\leqslant\la$. 
\end{defi}\bigskip

\begin{defi}\textbf{(transfinite sequence)}
A transfinite sequence is a $\la$-sequence for some ordinal $\la$.
\end{defi}\bigskip

\begin{defi}\textbf{(transfinite composition)}
Let $F$ be a $\la$-sequence. The (transfinite) composition of $F$ is the natural map $F_{0,\la}$ from $F_0$ to $F_\la$.  
\end{defi}\bigskip

\begin{defi}
Let $\mathscr{C}$ be a category, $\mathcal{A}$ a class of morphisms in $\mathscr{C}$ and $\la$ an ordinal. A $\la$-sequence $F$ in $\mathcal{A}$ is a $\la$-sequence in $\mathscr{C}$ such that for every $\be<\la$ the map $X_{\be,\be+1}$ belongs to $\mathcal{A}$.
\end{defi}\bigskip

\begin{defi}\textbf{(closure under transfinite composition)}
Let $\mathscr{C}$ be a category and $\mathcal{A}$ be a class of morphisms in $\mathscr{C}$. One says that the class $\mathcal{A}$ is closed under transfinite composition when the composition of any transfinite sequence in $\mathcal{A}$ is in $\mathcal{A}$ again.
\end{defi}\bigskip

\begin{prop}\label{prop4}
Let $\mathscr{C}$ be a category and $\mathcal{B}$ a class of morphisms in $\mathscr{C}$ and we take $\mathcal{A} = {}^\boxslash\mathcal{B}$. Then the class $\mathcal{A}$ is closed under transfinite composition.\\
In particular, if $\mathscr{C}$ carries the structure of a model category then  the class $\mathcal{C}$ of cofibrations (\textit{resp} the class $\mathcal{C}\cap \mathcal{W}$ of trivial cofibrations) is closed under transfinite composition.
\end{prop}
\begin{proof}\label{proofprop4}
Let $\lambda$ be an ordinal and $X$ a $\la$-sequence in $\mathcal{A}$. We need to prove that the transfinite composition $X_{0,\la}$ belongs to $\mathcal{A}$. It suffices to prove that this last map belongs to ${}^{\boxslash}\mathcal{B}$. Being given a lifting problem : \def\com2{\ar@{}[rrdd]|{\circlearrowright}}  $$\xymatrix{X_{0} \com2 \ar[rr]^g \ar[dd]&&Y\ar[dd]^{f}\\\\ X_\la \ar[rr]_h&&Z}$$ we need to construct a lift.\\
We define a set $S$ whose elements are the pairs $(\beta,j_{\beta})$ with $\beta$ an ordinal less than or equal to $\lambda$ and $j_\be$ a morphism from $X_\be$ to $Y$ satisfying that $f\circ j_\be$ is equal to the composition $h\circ X_{\be,\la}$ and $j_\be\circ X_{0,\be}$ is equal to $g$, in picture
$$\xymatrix@=1,5cm{X_0\ar[r]^g\ar[d]_{X_{0,\be}} &  Y\ar[d]^f \\
X_\be\ar[r]_{h\circ X_{\be,\la}}\ar[ru]_{j_\be} & Z\,.}$$\\
We equip this set with the preorder defined as follows,
$$(\be,j_\be)\leqslant (\gm,j_\gm)~\text{iff}~\be\leqslant\gm~\text{and}~j_\gm\circ X_{\be,\gm} = j_\be\,. $$
Note that  the set $S$ is not empty. Indeed, $(0,g)$ belongs to $S$. Let $C$ be a non-empty chain in $S$. The reader can easily check that $(\underset{(\be,j_\be)\in C}{\text{sup}}~\be,\underset{(\be,j_\be)\in C}{\text{colim}}~j_\be)$ belongs to $S$ and this is an upper bound of $C$.\\
By Zorn's lemma one has a maximal element in $S$ denoted $(\be_{\text{max}},j_{\text{max}})$. If $\be_{\text{max}}=\la$ then $j_{\text{max}}$ is a diagonal filler for our initial lifting problem.\\ Otherwise, $\be_{\text{max}}$ is strictly less than $\la$. Since $X_{\be_{\text{max}},\be_{\text{max}}+1}$ belongs to $^\boxslash B$ by assumption on the $\la$-sequence $X$, one gets a map $j_{\beta_{\text{max}}+1}$ as a diagonal filler in the following lifting problem,
$$\xymatrix@=2cm{X_{\be_{\text{max}}}\ar[r]^{j_{\text{max}}}\ar[d] & Y\ar[dd]^{f} \\  X_{\be_{\text{max}}+1}\ar[d]\ar@{-->}[ru]_{j_{\be_{\text{max}} +1}} & \\ X_\la\ar[r]_h & Z\,.}$$
The reader can easily check that $(\be_{\text{max}}+1,j_{\be_{\text{max}}+1})$ belongs to $S$ and it contradicts the maximality of $(\be_{\text{max}},j_{\text{max}})$.
\end{proof}\bigskip

To close this subsection we add for the curious reader the following proposition about the natural model structure on $\Gpd$ (see \ref{eg3}) that we were unable to find in the literature. Let us denote by $\text{cof}(\lbrace i\rbrace)$ the trivial cofibrations and by $\text{cell}(\lbrace i\rbrace)$ the transfinite compositions of pushouts in $\lbrace i\rbrace$, where $i$ is the generating trivial cofibration introduced \textit{ibid}.\medskip

\begin{prop}
For the natural model structure on $\Gpd$ one has the equality
$$\text{cof}(\lbrace i\rbrace) = \text{cell}(\lbrace i\rbrace)\,.$$
\end{prop}
\begin{proof}
Since $i$ is a trivial cofibration, by stability of trivial cofibrations under pushouts and transfinite compositions one concludes $\text{cell}(\lbrace i\rbrace)\subseteq \text{cof}(\lbrace i\rbrace)$.\\
It remains to prove the reverse inclusion. Let $f:A\rightarrow B$ be a trivial cofibration. Let us denote by $\text{Push}(\lbrace i\rbrace)$ the class of maps in $\Gpd$ that are pushouts of $i$. We define a set $S$ whose set of objects are pairs $(\lambda, X)$ with $\lambda$ any ordinal and $X$ any $\lambda$-sequence in $\text{Push}(\lbrace i\rbrace)$ such that $X_0 = A$ and for every $\beta\leqslant\lambda$ the groupoid $X_\beta$ is a subgroupoid of $B$.\\
We provide $S$ with the structure of a preordered set as follows,
\begin{center}

$(\lambda,X)\leqslant (\lambda',X')$ \textit{iff} $\lambda\leqslant\lambda'$ and $X'_{|[0,\lambda]} = X$.
\end{center}
Note that $S$ is a non-empty set. Indeed, $(0, X)$ with $X_0 := A$ is an element of $S$.\\
Let $C$ be a non-empty chain of $S$. By using the universal property of the colimit $\underset{(\lambda,X^\la)\in C}{\text{colim}}\lambda$, one gets a map denoted $(\displaystyle{\bigcup_{(\la,X^{\la})\in C}X^{\la}})$ from $[0,\underset{(\lambda,X^\la)\in C}{\text{colim}}\lambda]$ to $\Gpd$,
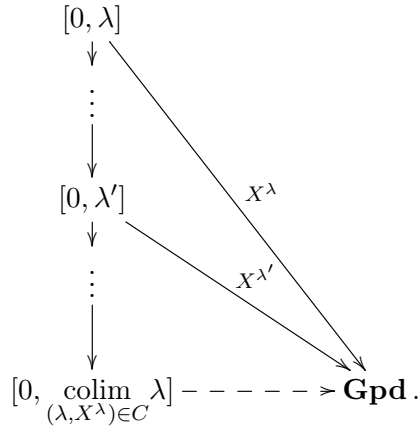
$$\xymatrix@C=2cm@R=0,5cm{[0,\la]\ar[rdddd]^{X^{\la}}\ar[d] & \\
\vdots\ar[d] & \\
[0,\la']\ar[rdd]^{X^{\la'}}\ar[d] & \\
\vdots\ar[d] & \\
[0,\underset{(\lambda,X^\la)\in C}{\text{colim}}\lambda]\ar@{-->}[r] & \Gpd\,.}$$
Clearly $(\underset{(\lambda,X^\la)\in C}{\text{colim}}\lambda,\displaystyle{\bigcup_{(\la,X^{\la})\in C}X^{\la}})$ is an element of $S$ and $(\la,X^\la)\leqslant(\underset{(\lambda,X^\la)\in C}{\text{colim}}\lambda, \displaystyle{\bigcup_{(\la,X^{\la})\in C}X^{\la}})$ for every $(\la,X^\la)\in C$.\\
By Zorn's lemma one gets a maximal element in $S$ denoted $(\la_{\text{max}},X^{\text{max}})$. Note that the transfinite composition $X^{\text{max}}_{0,\la_{\text{max}}}$ is (isomorphic to) the inclusion from $A$ to $X^{\text{max}}_{\la_{\text{max}}}$. Since $f$ is a trivial cofibration, it is an injective-on-objects equivalence hence it is (isomorphic to) the inclusion of a full subgroupoid of B which is equivalent to B. So it suffices to prove that $X^{\text{max}}_{\la_{\text{max}}}$ which is (isomorphic to) a full subgroupoid of $B$ is $B$ itself. Assume this is not the case, then there exists $x$ in $B$ such that $x$ does not belong to $X^{\text{max}}_{\la_{\text{max}}}$. Hence by essential surjectivity of $f$ there exists $y\in A\subseteq X^{\text{max}}_{\la_{\text{max}}}$ and an isomorphism $\phi$ in $B$ from $y$ to $x$. \\
In this case define a $(\la_{\text{max}}+1)$-sequence $X$ with $X_{|[0,\la_{\text{max}}]}:= X^{\text{max}}$ and $X_{\la_{\text{max}}+1}$ is the following pushout,
$$\xymatrix@=1,5cm{\bold{1}\ar[r]^y\ar@{^{(}->}[d]_i & X^{\text{max}}_{\la_{\text{max}}}\ar[d] \\
\I\ar[r] & X_{\la_{\text{max}}+1}\pushoutcorner\,.}$$
Then $(\la_{\text{max}}+1,X)$ belongs to $S$ and $(\la_{\text{max}},X^{\text{max}})<(\la_{\text{max}}+1,X)$ which contradicts the maximality of $(\la_{\text{max}},X^{\text{max}})$.\\
So one concludes that $X^{\text{max}}_{\la_{\text{max}}}$ is (isomorphic to) $B$ and $f$ is (isomorphic to) the transfinite composition $X^{\text{max}}_{0,\la_{\text{max}}}$. 
\end{proof}\bigskip

\begin{rmk}
Note that the proposition above is non-trivial even though under some slight assumptions on a category $\mathscr{C}$ you always have $\text{cof}(I)=\text{cell}(\hat{I})$ for any subset $I$ of morphisms of $\mathscr{C}$ but usually with $I\subsetneqq \hat{I}$ (see \cite{simpson11}, theorem 8.6.1).
\end{rmk}\bigskip

\subsection{The small object argument}\bigskip

It is usually hard to provide a category $\mathscr{C}$ with a model structure. We introduce below a technical device called \enquote{the small object argument} due to Quillen which makes easier to provide the factorizations involved in MC5 with the lifting properties of MC4. We first look at this argument in an informal way in the case of $\Gpd$ equipped with the natural model structure of \ref{eg3}, and afterwards we state and prove the small object argument in a rigourous way.\bigskip

Below we explore in details the factorizations involved in the natural model structure on $\mathbf{Gpd}$ given in \ref{eg3}.\\
Note that an isofibration (\textit{cf.} \ref{eg3}) is nothing but a morphism that has the RLP with respect to the inclusion $i:\b1 \hookrightarrow \I$. Now, starting with any morphism $f:X\rightarrow Y$ in $\Gpd$, one can factor it as a trivial cofibration (the inclusion of a full subgroupoid into an equivalent groupoid) followed by an isofibration by the following gluing construction. First, consider the set $S$ which contains all the pairs of morphisms $(g,h)$ such that the following diagram commutes : $$\xymatrix{\underset{~}{\b1} \ar[rr]^g \ar@{^{(}->}[dd]_{i}&&X \ar[dd]^f \\&\circlearrowright&\\ \I \ar[rr]_h &&Y\,. }$$ Note that such a pair $(g,h)$ is nothing but a (right) lifting problem with respect to $i$, and with the definition of an isofibration in mind this pair can be correspondingly seen as a pair $(g(0),h(\phi))$ which consists in an isomorphism of $Y$ and a point in the fiber of its domain. To provide a solution (\textit{i.e.} a lift) for this lifting problem consists in gluing a cell $\I$ in $X$ at $g(0)$ which lies above $h(\phi)$. Let us define this gluing construction by a simple pushout : $$\xymatrix{\underset{(g,h)\in S}{\coprod}\b1\ar[rr]^{+g} \ar[dd]_{\coprod i}&&X\ar[dd]^j\\\\ \underset{(g,h)\in S}{\coprod}\I \ar[rr]&&X' \pushoutcorner\,.}$$ By the universal property of the pushout $X'$ one has a morphism $p$ as follows : $$\xymatrix{\underset{(g,h)\in S}{\coprod}\b1\ar[rr]^{+g} \ar[dd]_{\coprod i}&&X\ar[dd]^j \ar@/^2pc/[rddd]^f&\\&&\com1&\\ \underset{(g,h)\in S}{\coprod}\I \ar@/_2pc/[rrrd]_{+h} \ar[rr]&\com1&X'\ar@{-->}[rd]^p \pushoutcorner&\\&&&Y\,.}$$ Note that $i$ is a trivial cofibration so by \ref{pullbackpushoutstability} (and by stability under coproduct) $j$ is a trivial cofibration. One can prove that $p$ is a fibration of groupoids (\textit{i.e.} an isofibration) as follows. Being given a lifting problem : $$\xymatrix{\underset{~}{\b1} \ar[rr]^g \ar@{^{(}->}[dd]_{i}&&X' \ar[dd]^p \\&\circlearrowright&\\ \I \ar[rr]_h &&Y\,, }$$ if $g(0)\in X$ then $(g,h)\in S$, \textit{i.e.} $(g,h)$ is an index in the gluing construction. It means that in $X'$ we glued a copy of $\I$ along the map $g$ (copy of $\I$ which lies above $h(\phi)$ hence a lift exists). Otherwise $g(0)$ belongs to $X'\setminus X$, this new point comes from the gluing in $X'$ of a copy of $\I$ corresponding to an index $(g',h')\in S$. Let $\I_{(g',h')}$ denotes this copy. Since we have also $$(g'(0), h(\phi)\circ h'(\phi))\in S$$ we have the following picture :
$$\xymatrix{\\g'(0) \ar@/^2pc/[rrrr]^{\I_{(g'(0),~h(\phi)\circ h'(\phi))}} \ar[rr]^{\I_{(g',h')}}&\underset{~}{} \ar[d]^p&{g(0)=g'(1)}\ar@{-->}[rr]&\underset{~}{} \ar[d]^p&g(1)\\h'(0)\ar[rr]_{h'(\phi)}&\overset{~}{}&{h'(1)=h(0)}\ar[rr]_{h(\phi)}&\overset{~}{}&h(1)\,.}$$ So $\I_{(g'(0),~h(\phi)\circ h'(\phi))}\circ \I^{-1}_{(g',h')}$ is the lift we are looking for.\\
\bigskip

\begin{defi}\textbf{(sequentially small)}\label{defi11}
Let $\mathscr{C}$ be a category with all small colimits. An object A of $\mathscr{C}$ is said to be sequentially small if for every functor $B:\mathbb{N}\rightarrow \mathscr{C}$ the induced map $$\underset{n}{\text{colim}}~\mathscr{C}(A,B(n))\rightarrow ~\mathscr{C}(A,\underset{n}{\text{colim}}B(n))$$ is  an isomorphism.\\
\end{defi}\bigskip

\begin{prop}\textbf{(the small object argument)}\label{prop5}
Let $\mathscr{C}$ be a category with all small colimits and $\mathscr{F}=\lbrace{f_{i}:A_{i}\rightarrow B_{i}}\rbrace_{i\in I}$ be a set of maps in $\mathscr{C}$ such that for all $i\in I,~A_{i}$ is sequentially small. Then any morphism $p$ of $\mathscr{C}$ can be factored as a transfinite composition of pushouts of coproducts of elements in $\mathscr{F}$ followed by a morphism which has the RLP with respect to every element in $\mathscr{F}$.\\
\end{prop}
\begin{proof}\label{proofprop5}
For each $i\in I$ consider the set $S(i)$ which contains all pairs of morphisms $(g,h)$ such that the following diagram commutes : $$\xymatrix{A_{i}\ar[rr]^g \ar[dd]_{f_{i}}&&X \ar[dd]^p\\&\circlearrowright&\\B_{i}\ar[rr]_h&&Y\,.}$$ We define the gluing construction $G^{1}(\mathscr{F},p)$ to be the object of $\mathscr{C}$ given by the pushout diagram $$\xymatrix{\underset{i\in I}\coprod~~\underset{(g,h)\in S(i)}\coprod A_{i} \ar[rr]^{+_{i}+_{(g,h)}g} \ar[dd]_{\coprod f_{i}}&&X \ar[dd]^{i_{1}}\\\\ \underset{i\in I}\coprod~~\underset{(g,h)\in S(i)}\coprod B_{i} \ar[rr]&&G^{1}(\mathscr{F},p)\pushoutcorner\,.}$$\\
Now by the universal property of the pushout one has a morphism $p_1$ as follows\\
$$\xymatrix{\underset{i\in I}\coprod~~\underset{(g,h)\in S(i)}\coprod A_{i} \ar[rr]^{+_{i}+_{(g,h)}g} \ar[dd]_{\coprod f_{i}}&&X\ar@/^2pc/[rddd]^p \ar[dd]^{i_{i}}&\\\\ \underset{i\in I}\coprod~~\underset{(g,h)\in S(i)}\coprod B_{i}\ar@/_2pc/[rrrd]_{+_{i}+_{(g,h)}h} \ar[rr]&&G^{1}(\mathscr{F},p)\pushoutcorner \ar@{-->}[rd]^{p_{1}}&\\&&&Y\,.}$$ Now by induction, for $k>1$ we define objects $G^{k}(\mathscr{F},p)$ and morphisms $$p_{k}:G^{k}(\mathscr{F},p)\rightarrow Y$$ by setting $G^{k}(\mathscr{F},p)=G^{1}(\mathscr{F},p_{k-1})$ and $p_{k}=(p_{k-1})_{1}$. By taking the colimit one has natural morphisms $i_{\infty}$ and $p_{\infty}$ as follows,\\
$$\xymatrix{X\ar[r]^{i_{1}} \ar@/^3pc/[rrrrr]^{i_{\infty}} \ar[dd]^p &G^{1}(\mathscr{F},p) \ar[r]^{i_{2}} \ar[dd]^{p_{1}}&G^{2}(\mathscr{F},p) \ar[r] \ar[dd]^{p_{2}}&\dots&\ar[r]&G^{\infty}(\mathscr{F},p) \ar[dd]^{p_{\infty}}\\\\
Y \ar@{=}[r]&Y \ar@{=}[r]&Y \ar@{=}[r]&\dots& \ar@{=}[r]&Y}$$ where $i_{\infty}$ is the transfinite composition of pushouts of coproducts of elements in $\mathscr{F}$ where $G^{\infty}(\mathscr{F},p):=\underset{n}{\text{colim}}~G^{n}(\mathscr{F},p)$. It remains to prove that $p_{\infty}$ has the RLP with respect to elements in $\mathscr{F}$.\\
Consider a lifting problem \def\com2{\ar@{}[rrdd]|{\circlearrowright}} $$\xymatrix{A_{i}\com2\ar[rr]^g \ar[dd]_{f_{i}}&&G^{\infty}(\mathscr{F},p) \ar[dd]^{p_{\infty}}\\\\B_{i}\ar[rr]_h&&Y\,.}$$ Since $A_{i}$ is sequentially small, there exists an integer $k$ such that $g$ is the composition of $g':A_{i}\rightarrow G^{k}(\mathscr{F},p)$ with the natural morphism $G^{k}(\mathscr{F},p)\rightarrow G^{\infty}(\mathscr{F},p)$. So one has $$\xymatrix{A_{i}\ar[rr]^{g'} \ar[dd]_{f_{i}}&&G^{k}(\mathscr{F},p)\ar[rr]^{i_{k+1}} \ar[dd]_{p_{k}}&&G^{k+1}(\mathscr{F},p)\ar[rr] \ar[dd]^{p_{k+1}}&&G^{\infty}(\mathscr{F},p)\ar[dd]^{p_{\infty}}\\\\B_{i}\ar[rr]_h&&Y\ar@{=}[rr]&&Y\ar@{=}[rr] &&Y}$$ where the composition of the top row is $g$. But $(g',h)$ contributes as an index in the construction of $G^{k+1}(\mathscr{F},p)$ from $G^{k}(\mathscr{F},p)$. Actually it indexes a gluing of $B_{i}$ to $G^{k}(\mathscr{F},p)$ along $A_{i}$ and so there exists a map $B_{i}\rightarrow G^{k}(\mathscr{F},p)$ with the required properties.
\end{proof}\bigskip

\section{Model structures on functor categories}
\label{sec:msofc}

Let $\mathcal{I}$ be a small index category and $\mathscr{C}$ be a model category, under slight assumptions there are at least two model structures on the functor category $[\mathcal{I}, \mathscr{C}]$.

\begin{prop}\textbf{(the projective model structure)}\label{prop6}
If $\mathscr{C}$ is cofibrantly generated there is the projective model structure on $[\mathcal{I},~\mathscr{C}]$ where a natural transformation $\varphi:F\Rightarrow G$ is : \begin{itemize} \item a weak equivalence if $\varphi_{i}:F_{i}\rightarrow G_{i}$ is a weak equivalence in $\mathscr{C}$ for every $i\in \mathcal{I}$. \item a fibration if $\varphi_{i}:F_{i}\rightarrow G_{i}$ is a fibration in $\mathscr{C}$ for every $i\in \mathcal{I}$. \item a cofibration if $\varphi$ has the LLP with respect to any trivial fibration.\\
\end{itemize}
\end{prop}
\begin{proof}\label{proofprop6}
See \cite{luriehtt} proposition A.2.8.2 .
\end{proof}\bigskip

\begin{prop}\textbf{(the injective model structure)}\label{prop7}
If $\mathscr{C}$ is combinatorial (see definition A.2.6.1 in \cite{luriehtt} for \textit{combinatorial}), there exists the injective model structure on 
$[\mathcal{I}, \mathscr{C}]$ where a natural transformation $\varphi : F\Rightarrow G$ is : \begin{itemize} \item a weak equivalence if $\varphi_{i}:F_{i}\rightarrow G_{i}$ is a weak equivalence in $\mathscr{C}$ for every $i\in \mathcal{I}$. \item a cofibration if $\varphi_{i}:F_{i}\rightarrow G_{i}$ is a cofibration in $\mathscr{C}$ for every $i\in \mathcal{I}$. \item a fibration if $\varphi$ has the RLP with respect to any trivial cofibration.\\
\end{itemize} 
\end{prop}
\begin{proof}\label{proofprop7}
See \cite{luriehtt} proposition A.2.8.2 .
\end{proof}\bigskip

\setchapterpreamble{
\begin{quote}
In this third chapter we make clear what we mean by a model of intensional type theory with a universe. For full details we invite the reader to look through all the references provided. For the syntax and semantics of homotopy type theory \cite{shulman:invdia} is highly useful. 
\end{quote}
}

\chapter{From model categories to type-theoretic fibration categories}
\label{sec:chp1}

\section{Organization}
\label{sec:org}

In section \ref{sec:catstruct} we give what we mean by a model of type theory with dependent sums, dependent products and \Id-types. For this purpose we introduce the definition of a type-theoretic fibration category following Michael Shulman \cite{shulman:invdia}. In the next section \ref{sec:univ}  we recall the notion of a universe in a type-theoretic fibration category. The final section \ref{sec:catunivaxiom} underlines the categorical translation of the univalence axiom.

\section{The categorical structure needed to model intensional type theory}
\label{sec:catstruct}

In \cite{shulman:invdia}, under the label type-theoretic fibration category,  Michael Shulman reformulated the categorical structure needed to model intensional type theory (with dependent sums, dependent products and identity types) making its close connection to homotopy theory clearer. We recall here a few definitions and facts.\bigskip

\begin{defi}\label{def:ttfc}(Shulman \cite{shulman:invdia} def. 2.1).
  A \textbf{type-theoretic fibration category} is a category $\mathscr{C}$ with the following structure.
  \begin{enumerate}[leftmargin=*,label=(\arabic*)]
  \item A terminal object $\b1$.\label{item:cat1}
  \item A subcategory $\mathcal{F}\subset\mathscr{C}$ containing all the objects, all the isomorphisms, and all the morphisms with codomain $\b1$.\label{item:cat2a}
    \begin{itemize}[leftmargin=*]
    \item A morphism in $\mathcal{F}$ is called a \textbf{fibration}. We write fibrations as $A\fib B$.
    % \item An object $A$ is called \textbf{fibrant} if $A\to 1$ is a fibration.
    \item A morphism $i$ is called an \textbf{acyclic cofibration} if it has the left lifting property with respect to all fibrations (see \ref{defi1}).\\ We write acyclic cofibrations as $A \acof B$.
    \end{itemize}
  \item All pullbacks of fibrations % between fibrant objects
    exist and are fibrations.\label{item:cat3}
  \item For every fibration $g: A\fib B$, % between fibrant objects,
    the pullback functor $g^*: \mathscr{C}/B \to \mathscr{C}/A$ has a partial right adjoint $\Pi_g$, defined at all fibrations over $A$, and whose values are fibrations over $B$.
    This implies that acyclic cofibrations are stable under pullback along $g$.\label{item:cat4}
  % \item If $f: A\fib B$ is a fibration, % between fibrant objects,
  %   its diagonal $A\to A\times_B A$ factors as $A\acof P_B A \fib A\times_B A$, where $P_BA \fib A\times A$ is a fibration and $A\acof P_B A$ is an acyclic cofibration which is preserved by pullback along \emph{all} maps % from a fibrant object
  %   into $B$.\label{item:cat6}\newcounter{cat6}\setcounter{cat6}{\value{enumi}}
  \item Every morphism factors as an acyclic cofibration followed by a fibration.\label{item:cat7}
  \item In the following commutative diagram:
    \[\vcenter{\xymatrix@C=1,5cm{
        X\ar[r]\ar[dd] \pullbackcorner &
        Y\ar[r]\ar[dd] \pullbackcorner &
        Z\ar[dd]\\\\
        A\ar@{ >->}[r]^{\sim} \ar@{->>}@(dr,dl)[rr] &
        B\ar@{->>}[r] &
        C\\
        & &
      }}\]
    if $B\fib C$ and $A\fib C$ are fibrations, $A\acof B$ is an acyclic cofibration, and both squares are pullbacks (hence $Y\to Z$ and $X\to Z$ are fibrations by~\ref{item:cat3}), then $X\to Y$ is also an acyclic cofibration.\label{item:cat8}
  \end{enumerate}
\end{defi}\bigskip

\begin{rmk}
In a type-theoretic fibration category the pair of classes of morphisms $(\mathcal{C}\cap \mathcal{W},\mathcal{F})$, where $\mathcal{C}\cap\mathcal{W}$ denotes the class of acyclic cofibrations namely $^{\boxslash}\mathcal{F}$, does not need to form a weak factorization system. Indeed, a morphism with the right lifting property with respect to the acyclic cofibrations needs not to be a fibration but by (5) together with the retract argument (see Lemma 1.1.9 in \cite{hovey99}) only a retract of a fibration. This is a difference from the structure involved in a model category that one can emphasize.\\
We could call such a pair $(\mathcal{C}\cap\mathcal{W},\mathcal{F})$ in a type-theoretic fibration category a \textit{very weak factorization system}.  
\end{rmk}\bigskip

The next definition makes explicit the link between type-theoretic fibration categories and type-theoretic model categories for the reader more familiar with homotopy theory. But note that we do not  really need the class of cofibrations nor the class of weak equivalences to model type theory (it is clear from the definition of a type-theoretic fibration category). \bigskip

The following definition and proposition are inspired by \cite{shulman:invdia}, definition 2.12 and proposition 2.13, but the definition of a type-theoretic model category is slightly simplified.

\begin{defi}\label{def:ttmc}
  A \textbf{type-theoretic model category} is a model category $\mathscr{C}$ with the following additional properties.
  \begin{enumerate}
  \item Acyclic cofibrations between fibrant objects are preserved by pullback along any fibration between fibrant objects.
  \item Pullback $g^*$ along any fibration between fibrant objects
    has a right adjoint $\Pi_g$.\label{item:m3}
    % This is automatic if $\mathscr{C}$ is locally cartesian closed.
  \end{enumerate}
\end{defi}\bigskip

The following proposition clarifies the link between type-theoretic fibration categories and type-theoretic model categories.\bigskip

\begin{prop}(Shulman \cite{shulman:invdia}, prop. 2.13)\label{shul2.13}.
  If $\mathscr{C}$ is a type-theoretic model category, then its full subcategory $\mathscr{C}\f$ of fibrant objects (equipped with the subcategory given by fibrations) is a type-theoretic fibration category.
\end{prop}\bigskip

\section{Universes}
\label{sec:univ}

Shulman also encapsulated in \cite{shulman:invdia} the categorical structure for modeling universes such that the coercion \el from terms of type Type to actual types respects type-theoretic operations up to isomorphism (they are not necessarily equal definitionally) as shown in fig. \ref{fig:univopbeta}. For the various options regarding the universes in type theory the reader can look at \url{http://ncatlab.org/homotopytypetheory/show/universe#Tarski}.

\begin{figure}[!h]
  \centering
  \begin{align}
    \el(\Sigma(A,B)) &\cong \sum_{x:\el(A)} \el(B(x)) \\
    \el(\Pi(A,B)) &\cong \prod_{x:\el(A)} \el(B(x)) \\
    \el(\Id(A,x,y)) &\cong (x\id y).
  \end{align}
\caption{Coercion isomorphisms for the universe type}
\label{fig:univopbeta}
\end{figure}

\begin{defi}\label{def:univ}(Shulman \cite{shulman:invdia}, def. 6.12).
  A fibration $p: \widetilde{U}\fib U$ in a type-theoretic fibration category $\mathscr{C}$ is a \textbf{universe} if the following hold, where ``small fibration'' means ``a pullback of $p$''.
  \begin{enumerate}
  \item Small fibrations are closed under composition and contain the identities.\label{item:u1}
  \item If $f: B\fib A$ and $g: A\fib C$ are small fibrations, so is $\Pi_g f \fib C$.\label{item:u2}
  \item If $A\fib C$ and $B\fib C$ are small fibrations, then any morphism $f: A\to B$ over $C$ factors as an acyclic cofibration followed by a small fibration.\label{item:u3}
  \end{enumerate}
\end{defi}

\section{The categorical univalence axiom}
\label{sec:catunivaxiom}

Following Shulman (\cite{shulman:invdia}, part 7), one says that a universe $p: \widetilde{U}\fib U$ in a type-theoretic fibration category is univalent if the morphism from $U$ to $E$ over $U\times U$, that maps a small type to its identity equivalence, is a (right) homotopy equivalence, where $E$ denotes the domain of the fibration over $U\times U$ interpreting the dependent type  $$(A:\text{Type}),(B:\text{Type})\pr\text{Equiv}(A,B)\ty.$$

%\chapter{On derivation of univalent universes}

%\input{Derivability}

\setchapterpreamble{
\begin{quote}
We recall that $\G = \mathbb{Z}/2\mathbb{Z}$.
In this chapter we give a model of intensional type theory with universes in the category $\GGpd$ of groupoids equipped with an involution. Our universe is the natural universe that lifts, with respect to the projective model structure on presheaves, the \textit{univalent} universe in the category of groupoids consisting of small discrete groupoids. But surprisingly our universe is not univalent. Moreover, the projective model structure also allows to break function extensionality (a weak form of the \textit{Univalence Axiom}). 
\end{quote}
}

\chapter{On the inadequacy of the projective model structure for univalence}
\label{sec:chp4}

\section{Organization}
\label{sec:organization}

To prove the inadequacy of the projective model structure with respect to univalence we consider the presheaf category $\GGpd$.  Having first endowed the category of groupoids $\Gpd$ with its natural model structure of section \ref{sec:basics}, we endow the presheaf category $\GGpd$ with its projective model structure. 
In section \ref{sec:tpmsme} we make the projective trivial cofibrations in $\GGpd$ explicit. This way we are able to equip in section \ref{sec:GGpdttfc} the category 
$\GGpd$ with the structure of a type-theoretic fibration category. In section \ref{sec:tcou} we equip this type-theoretic fibration category with one universe $U_{\text{Gpd}_{\Delta}(V_\kappa)}$ for each universe $\text{Gpd}_{\Delta}(V_\kappa)$ in $\Gpd$ consisting of $\kappa$-small discrete groupoids and isomorphisms between them for some inaccessible cardinal $\kappa$. This new universe $U_{\text{Gpd}_{\Delta}(V_\kappa)}$ classifies projective fibrations which are levelwise discrete fibrations of groupoids with $\kappa$-small fibers. For this reason our universe $U_{\text{Gpd}_{\Delta}(V_\kappa)}$ is the natural lift, with respect to projective fibrations, of the universe $\text{Gpd}_{\Delta}(V_\kappa)$. Despite the fact that the universe $\text{Gpd}_{\Delta}(V_{\kappa})$ in $\Gpd$ is univalent we surprisingly prove in section \ref{sec:tfou} that the resulting universe $U_{\text{Gpd}_{\Delta}(V_\kappa)}$ in $\GGpd$ is not univalent. This last proof relies on the characterization of projective homotopy equivalences given in section \ref{sec:heGGpd}.\\
Last, recall that function extensionality holds in the internal language of the type-theoretic fibration category $\Gpd$. But we prove in section \ref{sec:tffe} that function extensionality have been broken in the internal language of our type-theoretic fibration category $\GGpd$.

\section{The projective model structure made explicit}
\label{sec:tpmsme}

Since the model structure on $\Gpd$ (see \ref{eg3}) is cofibrantly generated and $\G$ is a small category, there exists the projective model structure on $\GGpd$ (see \ref{sec:msofc}). By a slight abuse of notation we denote (this notational convention will only hold in this chapter) both the underlying categories and the corresponding model categories by $\Gpd$ and $\GGpd$ respectively. Hereinafter by a levelwise weak equivalence (\textit{resp.} a levelwise fibration) one means a map whose underlying map of groupoids is a weak equivalence (\textit{resp.} a fibration) in $\Gpd$.\\
Recall that one can describe this projective model structure by :\\
\begin{itemize} 
\item Weak equivalences are the levelwise weak equivalences.\\
\item Fibrations are the levelwise fibrations.\\
\item Cofibrations are those maps with the left lifting property with respect to trivial fibrations.\\ 
\end{itemize}

\begin{rmk}\label{rmk1}
Note that in $\GGpd$ all objects are fibrant since all objects in $\Gpd$  are fibrant and projective fibrations are levelwise. So $(\GGpd)\f$ is the same as $\GGpd$, where $(\GGpd)\f$ is the full subcategory of fibrant objects of $\GGpd$.\\
\end{rmk}

\begin{rmk}\label{prop1} 
In $\GGpd$ a cofibrant object has no fixed point. In particular, $\b1!$, the terminal object, is not cofibrant.\\
Indeed, let A be an object of $\GGpd$ with a fixed point and consider the following lifting problem :\\
$$\xymatrix{\bold{0}!\ar[rr]\ar[dd]\commutatif&&\cI\ar@{->>}[dd]^*[@]{\hbox{$\sim$}}\\\\\text{A}\ar[rr]&&\b1!}$$\\
One can easily prove that the unique map from $\cI$ to the terminal object $\b1!$ is surjective on objects and fully faithful and so a projective trivial fibration. But a diagonal filler does not exist since it should map the fixed point in A to a fixed point in $\cI$ and there is no such fixed point. Thus A is not cofibrant.
\end{rmk}\bigskip

\begin{rmk}\label{rmk2} 
The previous remark shows that being a levelwise cofibration is not sufficient to be a projective cofibration.\\
\end{rmk}

\begin{rmk}\label{prop2} 
In $\GGpd$ being a levelwise trivial cofibration is not sufficient to be a projective trivial cofibration.\\
Indeed, consider the following lifting problem :\\
$$\xymatrix{\underset{~}{\cI}\ar@{=}[rr]\ar@{^{(}->}[dd]\commutatif&&\cI\ar@{->>}[dd]\\\\\tr\ar[rr]&&\b1!}$$\\
where the unique map from $\cI\text{ to }\b1!$ is a projective fibration since every object is fibrant. It is clear that $\cI\hookrightarrow \tr$ (see the \textbf{Notations} part at the end of the introduction \ref{intro} for the definition of $\tr$) is a trivial cofibration of groupoids when one forgets the involutions since this is the inclusion of a full subcategory equivalent to the target category. Again a diagonal filler does not exist since there is no fixed point in $\cI$.
\end{rmk}\bigskip

Knowing the generating trivial cofibrations and the generating cofibrations in $\Gpd$, by looking at the construction of the projective model structure one gets the generating projective trivial cofibrations and the generating projective cofibrations in $\GGpd$ (see \cite{luriehtt}, proof of proposition A.2.8.2). A set of generating projective trivial cofibrations is given by the following inclusion : $$S(i): S(\bold{1})\hookrightarrow S(\I)$$ where we recall that $S(\bold{1})$ (\textit{resp} $S(\I)$) is $\bold{1}\textstyle\coprod\bold{1}$ (\textit{resp} $\I\textstyle\coprod\I$) equipped with the involution which swaps the two components of the coproduct and $S(i)$ is $i\textstyle\coprod i$ with $i$ the generating trivial cofibration of groupoids (see \ref{eg3}).

\begin{prop}\label{prop3}
Let $f:A\rightarrow B$ be a morphism in $\GGpd$. The following are equivalent : 
\begin{enumerate}[label=(\roman*)]
\item $f$ is a trivial cofibration.
\item $f$ is a levelwise trivial cofibration and $f$ induces a bijection between the set of fixed points of $A$ and the set of fixed points of $B$.
\item $f$ is a levelwise trivial cofibration and $f$ induces an isomorphism between $A^{\G}$ and $B^{\G}$ the subgroupoids of fixed points and fixed morphisms.
\item $f$ is a levelwise trivial cofibration and $f$ induces an isomorphism between $A\f$ and $B\f$ the full subgroupoids of fixed points. 
\end{enumerate}
\end{prop}

\begin{proof}\label{proofprop3}
We prove successively $(i)\Rightarrow (ii)$, $(ii)\Rightarrow (iii)$, $(iii)\Rightarrow (iv)$ and last $(iv)\Rightarrow (i)$.
\begin{itemize}
\item $(i)\Rightarrow (ii)$: assume that $f$ is a trivial cofibration. It is well-known that $f$ is a levelwise trivial cofibration (see proposition 11.6.2 in \cite{hir99}). Moreover note that if $x$ is a fixed point of $A$ one has $f(x) = f(\al(x)) = \be(f(x))$, where we recall that $\al$ (\textit{resp.} $\be$) denotes the involution on $A$ (\textit{resp.} $B$), hence $f(x)$ is a fixed point of $B$. We also know that $f$ is injective on objects as a cofibration between groupoids. We need to prove that any fixed point in $B$ is the image of a fixed point in $A$. To achieve this the reader can check this fact for the generating trivial cofibration $S(i)$ and he can check the stability of this fact under pushout, transfinite composition and retraction. One concludes that $f$ induces a bijection between the fixed points in $A$ and the fixed points in $B$.
\item $(ii)\Rightarrow (iii)$: it is a straighforward consequence of $f$ being fully faithful.
\item $(iii)\Rightarrow (iv)$: \textit{idem}.
\item $(iv)\Rightarrow (i)$: consider the following lifting problem,
$$\xymatrix@=1,5cm{A\ar[r]^h\ar[d]_f & C\ar@{->>}[d]^g \\
B\ar[r]_k & D\,,}$$
where $g$ is a fibration. We want to find a diagonal filler. Since $f$ is a levelwise trivial cofibration we may assume it is the inclusion of a full subgroupoid of $B$ and this inclusion is an equivalence.\\
We define a set $S$ as the set of triples $(B',j':B'\rightarrow C)$ where $B'$ is a full subgroupoid of $B$ which contains $A$, and $B'$ is stable under the involution $\be$ on $B$ (in other words the restriction of $\be$ to $B'$ is an involution on $B'$), last the morphism $j'$ in $\GGpd$ is a diagonal filler for the following diagram,
$$\xymatrix@=1,5cm{A\ar[rr]^h\ar@{^{(}->}[d] & & C\ar@{->>}[d]^g \\
B'\ar@{^{(}->}[r]\ar[rru]_{j'} &  B \ar[r]_k & D\,.}$$
One provides the set $S$ with the structure of a preordered set as follows:
$$(B',j')\leqslant (B'',j'')\qquad\textit{iff}\qquad B'\subseteq B''\quad\text{and}\quad\xymatrix{B'\ar@{^{(}->}[rr]\ar[rd]_{j'} & & B''\ar[ld]^{j''}\\ & C & }\quad\text{commutes}\,.$$
Note that $S$ is a non-empty set. Indeed, $(A,h)$ is an element of $S$.\\
Let $\mathfrak{C}$ be a non-empty chain of $S$. Consider $\displaystyle{\bigcup_{(B',j')\in \mathfrak{C}}B'}$,  \textit{ie} the full subgroupoid of $B$ whose set of objects is $\displaystyle{\bigcup_{(B',j')\in\mathfrak{C}}\text{Ob}(B')}$. One defines a map $j$ from $\displaystyle{\bigcup_{(B',j')\in\mathfrak{C}}B'}$ to $C$ as $\displaystyle{\bigcup_{(B',j')\in\mathfrak{C}}j'}$. Clearly, $(\displaystyle{\bigcup_{(B',j')\in \mathfrak{C}}B'},j)$ is an element of $S$ and one has $$(B',j')\leqslant (\displaystyle{\bigcup_{(B',j')\in \mathfrak{C}}B'},j)$$ for  every $(B',j')\in \mathfrak{C}$.\\
By Zorn's lemma we get a maximal element in $S$ denoted $(B_{\text{max}},j_{\text{max}})$. It suffices now to prove that $B_{\text{max}}$ is $B$. Assume this is not the case, then there exists an object $x$ in $B$ such that $x$ is not an object of $B_{\text{max}}$. Moreover, since $f$ induces an isomorphism between $A\f$ and $B\f$ and $B_{\text{max}}$ contains $A$, we know that $x$ is not a fixed point.\\
Since $f$ is a levelwise weak equivalence, by the essential surjectivity there exists $y$ in $A$ (hence $y$ is an element of $B_{\text{max}}$) and an isomorphism $\varphi$ in $B$ from $f(y)$, which is equal to $y$, to $x$. Therefore we define $B'$ as the following pushout in $\GGpd$,
$$\xymatrix@=1,5cm{S(\b1)\ar[r]^l\ar@{^{(}->}[d]_{S(i)} & B_{\text{max}}\ar[d] \\
S(\I)\ar[r] & B'\pushoutcorner\,,}$$
where $l(0)$ is $y$ and $l(1)$ is $b(y)$. The groupoid $B'$ is (isomorphic to) the full subgroupoid of $B$ whose set of objects is $\text{Ob}(B_{\text{max}})\bigcup\lbrace x, b(x)\rbrace$. Thanks to the universal property of the pushout one defines a map $j'$ from $B'$ to $C$ as follows. First, we define a map $m$ as a diagonal filler in the following lifting problem,
$$\xymatrix@=1,5cm{S(\b1)\ar[r]^p\ar@{^{(}->}[d]_{S(i)} & C\ar@{->>}[d]_g \\
S(\I)\ar[r]_q\ar@{-->}[ru]_m & D\,,}$$
where $p(0)$ is $h(y)$ and $p(1)$ is $\gm(h(y))$ and $q(\phi)$ is $k(\varphi)$.\\ Then we get the map $j'$,
$$\xymatrix@=1,5cm{S(\b1)\ar[r]^l\ar@{^{(}->}[d]_{S(i)} & B_{\text{max}}\ar[d]\ar@/^/[rdd]^{j_{\text{max}}} & \\
S(\I)\ar[r]\ar@/_/[rrd]_m & B'\ar@{-->}[rd]^{j'}\pushoutcorner & \\
 & & C\,.}$$
One has that $(B',j')$ belongs to the set $S$ and $(B_{\text{max}},j_{\text{max}})<(B',j')$ which contradicts the maximality of $(B_{\text{max}},j_{\text{max}})$. So $B_{\text{max}}$ is $B$ and $j_{\text{max}}$ is the diagonal filler we are looking for, hence we have exhibited $f$ as a trivial cofibration.
\end{itemize}
\end{proof}\bigskip

\begin{cor}\label{cor1}
Let $G$ be a groupoid equipped with an involution and $G'$ a subgroupoid of $G$ stable under this involution such that $G'\f = G\f$ and the inclusion map $\iota$ from $G'$ to $G$ is an equivalence of groupoids. Then the map $\iota$ is a trivial cofibration.
\end{cor}
\begin{proof}
Straighforward with \ref{prop3} since $G'\f = G\f$ and $\iota$ is a levelwise trivial cofibration.
\end{proof}

\section{$\GGpd$ as a type-theoretic fibration category}
\label{sec:GGpdttfc}

Equipped with the proposition \ref{prop3} one can manage to prove that $\GGpd$, equipped with the subcategory given by projective fibrations, is a type-theoretic fibration category (see def \ref{def:ttfc}).\bigskip

\begin{lem}\label{rightproper}
In the natural model structure on $\Gpd$ trivial cofibrations are stable by pullback along any fibration.
\end{lem}
\begin{proof}
Let $g$ be a fibration from $A$ to $B$ and $f$ a trivial cofibration from $C$ to $B$. Consider the pullback of $f$ along $g$ denoted $g^{*}f$,
$$\xymatrix@=1,5cm{A\times_{B}C\pullbackcorner \ar[r]\ar[d]_{g^{*}f} & C\ar@{ >->}[d]^{\rotatebox{90}{$\sim$}}_f\\
A\ar@{->>}[r]_g & B\,.}$$
First, $g^{*}f$ is an injective-on-objects functor. Indeed, let $(x,y)$ and $(x',y')$ be two objects of $A\times_B C$ with $x = x'$. In this case on has
$$ f(y) = g(x) = g(x') = f(y')$$
hence $f(y) = f(y')$ and by the injectivity on objects of $f$, one concludes that $y = y'$.\\
Second, we prove that $g^{*}f$ is a weak equivalence, namely an equivalence of groupoids. \\
The functor $g^{*}f$ is essentially surjective. Indeed, let $y$ be any object of $A$, then $g(y)$ is an object of $B$ hence there exist $x$ in $C$ and an isomorphism $\phi$ in $B$ from $f(x)$ to $g(y)$. Since $g$ is a fibration there exists a lift in $A$, denoted $\widetilde{\phi^{-1}}$, of $\phi^{-1}$ at $y$. Let us denote by $z$ the codomain of this lift. One has $g(z) = f(x)$, hence $(z,x)$ is an element of $A\times_B C$ and $\widetilde{\phi^{-1}}$ is an isomorphism in $A$ from $g^{*}f(z,x) = z$ to $y$.\\
We now prove that $g^{*}f$ is a fully faithful functor. Let $(x,y)$, $(x',y')$ be two elements in $A\times_B C$. We need to prove that the induced map by $g^{*}f$ from the homset $A\times_B C((x,y),(x',y'))$ to the homset $A(x,x')$ that maps a morphism $(\phi,\psi)$ to $\phi$ is bijective. Let us prove it is injective. Let $(\phi,\psi)$,$(\phi',\psi')$ be two elements in this first homset from $(x,y)$ to $(x',y')$ with $\phi = \phi'$. The map induced by $f$ from $C(y,y')$ to $B(f(y),f(y'))$ is in particular injective since $f$ is fully faithful. So one concludes from
$$f(\psi) = g(\phi) = g(\phi') = f(\psi')$$
the equality $\psi = \psi'$. We now prove the surjectivity of the map under consideration. Let $\phi$ be an element in $A(x,x')$ and consider $g(\phi)$ which is an element of $B(g(x),g(x')) = B(f(y),f(y'))$. By surjectivity of the map induced by $f$ there exists a (unique) map $\psi$ in $C(y,y')$ with $f(\psi) = g(\phi)$. Hence $(\phi,\psi)$ is an element of $A\times_B C((x,y),(x',y'))$ with $(g^{*}f)(\phi,\psi) = \phi $. So $g^{*}f$ is fully faithful, being also an injective-on-objects functor it is a trivial cofibration.     
\end{proof}\bigskip

We prove the analogous result for $\GGpd$ with respect to the projective model structure.\medskip

\begin{lem}\label{rightproperbis}
In the projective model structure on $\GGpd$ trivial cofibrations are stable under pullback along any fibration.
\end{lem}
\begin{proof}
Consider the following pullback,
$$\xymatrix@=1,5cm{A\times_{B}C\pullbackcorner \ar[r]\ar[d]_{g^{*}f} & C\ar@{ >->}[d]^{\rotatebox{90}{$\sim$}}_f\\
A\ar@{->>}[r]_g & B\,.}$$
Since the underlying morphism of $g$ is a fibration of groupoids and the underlying morphism of $f$ is a trivial cofibration of groupoids by \ref{rightproper} we conclude that the underlying morphism of $g^{*}f$ is a trivial cofibration of groupoids. Thanks to \ref{prop3} it suffices to prove the surjectivity of $g^{*}f$ on the fixed points. Let $x$ be a fixed point in $A$. Then $g(x)$ is a fixed point in $B$, since $f$ is a trivial cofibration thanks to \ref{prop3} there exists a (unique) fixed point $y$ in $C$ with $f(y) = g(x)$. Hence $(x,y)$ is a fixed point in $A\times_B C$ whose image by $g^{*}f$ is $x$.   
\end{proof}\bigskip

\begin{lem}\label{pullbackstabilitytrivialcof}
The pullback functor along a fibration in $\GGpd$ preserves trivial cofibrations.
\end{lem}
\begin{proof}
Let $g$ be a fibration in $\GGpd$ from $A$ to $B$ and consider the pullback functor along $g$ denoted $g^{*}$ from the slice category $\GGpd/B$ to the slice $\GGpd/A$. Let $\phi$ be a map in $\GGpd/B$,
$$\xymatrix{C\ar[rr]^{\phi}\ar[rd]_f & & D\ar[dl]^h\\
 & B &\,.}$$
The pullback functor maps it to the following dotted arrow $g^{*}\phi$ in the slice category $\GGpd/A$,
$$\xymatrix@=1,5cm{A\times_B C\pullbackcorner\ar@{->>}[rr]\ar[dd]_{g^{*}f}\ar@{-->}[rd]^{g^{*}\phi} & & C\ar[dd]^f\ar[rd]^{\phi} & \\
 & A\times_B D\pullbackcorner\ar@{->>}[rr]\ar[dl]_{g^{*}h} & & D\ar[ld]^h\\
 A\ar@{->>}[rr]_g & & B & \,.}$$
Since $g^{*}\phi$ is the pullback of $\phi$ along the fibration from $A\times_B D$ to $D$ (this last map is a fibration since it is a pullback of the fibration $g$), it follows from \ref{rightproperbis} that $g^{*}\phi$ is a trivial cofibration. 
\end{proof} \bigskip

\begin{lem}\label{thm2}
For every fibration $g:A\twoheadrightarrow B$ in $\GGpd$, the pullback functor 
$$g^{*}:\GGpd/B \rightarrow \GGpd/A$$ 
has a right adjoint $\Pi_g$, and $\Pi_g$ maps fibrations over $A$ to fibrations over $B$.\\
\end{lem}
\begin{proof}\label{proofthm2}
We introduce the following notation : if $u$ is a morphism in $B$ then by $g^{*}u$ we mean the following pullback in $\Gpd$: 
$$\xymatrix{A\times_{B}\I\ar[rr]\ar[dd]_{g^{*}u}\pullbackcorner&&\I\ar[dd]^u\\\\A\ar[rr]_g&&B~.}$$ 
Where $u:\I\rightarrow B$ is the functor that maps the isomorphism $\phi: 0\xrightarrow{\cong}1$ to $u$ in $B$. Let $f:C\rightarrow A$ be a morphism in $\GGpd$. Define $\text{dom}(\Pi_{g}f)$ as the groupoid whose collection of objects denoted $(\text{dom}(\Pi_{g}f))_{0}$ are the pairs $(y,s)$ with 
$y\in B$, and $s:g^{-1}{\lbrace{y}\rbrace}\rightarrow C$ a partial section of $\underline{f}$,
where $g^{-1}\lbrace{y}\rbrace$ is the subgroupoid of $A$ whose objects are objects of $A$ above $y$ and morphisms are morphisms of $A$ above $1_{y}$. Define its collection $(\text{dom}(\Pi_{g}f))_{1}$ of morphisms as the pairs $(u,v)$ with $u$ a morphism in $B$ and $v:g^{*}u\rightarrow \underline{f}$ a morphism in $\Gpd/A$. Define two functions $\mathbf{s},~\mathbf{t}:(\text{dom}(\Pi_{g}f))_{1}\rightarrow(\text{dom}(\Pi_{g}f))_{0}$ as follows :\\
$$\ffour{\mathbf{s}:\quad{(\text{dom}(\Pi_{g}f))_{1}}}{(\text{dom}(\Pi_{g}f))_{0}}{(u,v)}{(\text{dom}~u,~v_{|_{ A\times_{B}{\lbrace{0}\rbrace}}})}$$
and 
$$\ffour{\mathbf{t}:\quad{(\text{dom}(\Pi_{g}f))_{1}}}{(\text{dom}(\Pi_{g}f))_{0}}{(u,v)}{(\text{cod}~u,~v_{|_{ A\times_{B}{\lbrace{1}\rbrace}}})}.$$ 
Define a partial function 
$$\circ:\quad(\text{dom}(\Pi_{g}f))_{1}\times(\text{dom}(\Pi_{g}f))_{1}\rightarrow(\text{dom}(\Pi_{g}f))_{1}$$ 
as follows : being given $(u,v),(u',v')\in(\text{dom}(\Pi_{g}f))_{1}$ such that $\mathbf{s}(u',v')=\mathbf{t}(u,v)$, define $v'':g^{*}(u'\circ u)\rightarrow f$ in $\Gpd/A$ by\\
$$v''(x,0)=v(x,0)\quad\text{for}\quad(x,0)\in A\times_{B}\I$$
$$v''(x,1)=v'(x,1)\quad\text{for}\quad(x,1)\in A\times_{B}\I$$
$$v''(h,1_{0})=v(h,1_{0})\quad\text{for}\quad(h,1_{0})\in A\times_{B}\I$$
$$v''(h,1_{1})=v'(h,1_{1})\quad\text{for}\quad(h,1_{1})\in A\times_{B}\I\,.$$\\
It remains to define $v''(h:x\rightarrow x',\phi)$ with $g(h)=u'\circ u$. Let $\tilde{u}$ be a lift of $u$ at $x$ by $g$ (\textit{i.e.} $\text{dom}~\tilde{u}=x$ and $g(\tilde{u})=u$), such a lift exists since $g$ is an isofibration.\\
One takes 
$$v''(h,\phi)=v'(h\circ{\tilde{u}}^{-1},\phi)\circ v(\tilde{u},\phi)$$ 
which is a well-defined composition in $C$. For the sake of readibility and for the purpose of avoiding lengthy but straighforward calculations we do not give further details but the reader can check that the defined composition in $\text{dom}(\Pi_{g}f)$ is associative. At least note that $v''(h,\phi)$ does not depend on the choice of the lift $\tilde{u}$. Indeed, from the assumption $\mathbf{s}(u',v')=\mathbf{t}(u,v)$ one concludes that $v'_{|A\times_B \lbrace 0\rbrace} =  v_{|A\times_B \lbrace 1\rbrace}$. Let $\hat{u}$ be an other lift of $u$ at $x$, one has the following sequence of equalities,
\begin{align*}
v'(h\circ \hat{u}^{-1},\phi)\circ v(\hat{u},\phi)\circ [v'(h\circ \tilde{u}^{-1},\phi)\circ v(\tilde{u},\phi)]^{-1} &=  v'(h\circ \hat{u}^{-1},\phi)\circ v(\hat{u},\phi) \circ v(\tilde{u},\phi)^{-1} \\
& \circ v'(h\circ \tilde{u}^{-1},\phi)^{-1} \\
 &=  v'(h\circ \hat{u}^{-1},\phi)\circ v(\hat{u},\phi) \circ v(\tilde{u}^{-1},\phi^{-1})\\
 & \circ v'(\tilde{u}\circ h^{-1},\phi^{-1})\\
 &= v'(h\circ \hat{u}^{-1},\phi)\circ v(\hat{u}\circ \tilde{u}^{-1},1_1)\\
 & \circ v'(\tilde{u}\circ h^{-1},\phi^{-1})\\
 &= v'(h\circ \hat{u}^{-1},\phi)\circ v'(\hat{u}\circ \tilde{u}^{-1},1_0)\\
 & \circ v'(\tilde{u} \circ h^{-1},\phi^{-1})\\
&=  v'(h\circ \hat{u}^{-1},\phi)\circ v'(\hat{u}\circ h^{-1},\phi^{-1})\\
&= v'(1_{x'}, 1_1)\\
&= 1_{v''(x',1)}\,.
\end{align*}
Last, define a function $\text{id}:(\text{dom}(\Pi_{g}f))_{0}\rightarrow(\text{dom}(\Pi_{g}f))_{1}$, let $(y,s)\in(\text{dom}(\Pi_{g}f))_{0}$, take $\text{id}_{(y,s)}:=(1_{y},s)$ which is a slight abuse of notation since by the second member in the pair $(1_{y},s)$ we really mean the functor between $g^{*}(1_{y})$ and $f$ in $\Gpd/A$ which maps $(x,i)$ with $i=0,1$ to $s(x)$ and $(h,-)$ to $s(h)$. The reader can check that $(\text{dom}(\Pi_{g}f))$ is a groupoid with $(u,v)^{-1}$ given by : 
$$\text{fst}[(u,v)^{-1}]:=u^{-1}$$ 
$$\text{snd}[(u,v)^{-1}](x,i):=v(x,1-i)$$ $$\text{snd}[(u,v)^{-1}](h,1_{i}):=v(h,1_{1-i})$$
$$\text{snd}[(u,v)^{-1}](h,\phi):=v(h,\phi^{-1})$$\\
(where fst and snd denote the first and the second projections).\\
Recall that we use a Greek letter to denote the involution of a groupoid denoted by the corresponding uppercase letter , for instance $\al$ denotes the involution of the groupoid $A$.\\ 
Now, one can equip $\text{dom}(\Pi_{g}f)$ with an involution denoted $\pi_{g}f$ :\\
$$\fsix{\pi_{g}f:\quad{\text{dom}(\Pi_{g}f)}}{\text{dom}(\Pi_{g}f)}{(y,s)\quad}{(\be(y),~\pi_{g}f(s):g^{-1}\lbrace{\be(y)}\rbrace\rightarrow C)}{(u,v)\quad}{(\be(u),~\pi_{g}f(v):g^{*}\lbrace{\be(u)}\rbrace\rightarrow \underline{f})}$$\\
with  
$$\fsix{\pi_{g}f(s):\quad{g^{-1}\lbrace{\be(y)}\rbrace}}{C}{x\qquad}{\gm(s(\al(x)))}{h\qquad}{\gm(s(\al(h)))}$$ 
and 
$$\fsix{\pi_{g}f(v):\quad{g^{*}(\lbrace{\be(u)}\rbrace)}}{\underline{f}}{(x,i=0,1)}{\gm(v(\al(x),i))}{(h,-)\quad}{\gm(v(\al(h),-))}$$\\
Last, one defines $\Pi_{g}f$ as follows : 
$$\fsix{\Pi_{g}f:\quad{\text{dom}(\Pi_{g}f)}}{B}{(y,s)\quad}{y}{(u,v)\quad}{u}\,.$$ 
This is immediate that $\Pi_{g}f$ is compatible with the involutions $\pi_{g}f$ and $\be$ (respectively on $\text{dom}(\Pi_{g}f)$ and $B$). \\
Now we define $\Pi_g$ on morphisms. Let $i$ be a morphism from $f$ to $h$ in the slice category $\GGpd/A$,
$$\xymatrix{C\ar[rr]^i\ar[rd]_f & & D\ar[dl]^h \\
& A & \,.}$$
We define $\Pi_{g}i$ from $\Pi_{g}f$ to $\Pi_{g}h$ as follows.\\
For any element $(y,s)$ in $\text{dom}(\Pi_{g}f)$ we take 
$$(\Pi_{g}i)(y,s) = (y, i\circ s)$$
and 
$$(\Pi_{g}i)(u,v) = (u, i\circ v).$$
\\It remains to check that $\Pi_{g}$ is \enquote{the} right adjoint to the pullback functor along $g$ denoted $g^{*}$.
Let $h:D\rightarrow B$ be a morphism in $\GGpd$ and define a natural bijection $\varphi_{g,f,h}$ (shorten by $\varphi$):\bigskip
$$\ffour{\varphi: \GGpd/A(g^{*}h,f)}{\GGpd/B(h,\Pi_{g}f)}{\xymatrix{g^{*}D\ar[rr]^{v} \ar[rdd]_{g^{*}h}&&C\ar[ldd]^{f}\\&\circlearrowright&\\&A&}}{\xymatrix{D\ar[rr]^{\varphi(v)}\ar[rdd]_{h}&&**[r]{\text{dom}(\Pi_{g}f)}\ar[ldd]^{\Pi_{g}f}\\&\circlearrowright&\\&B&}}$$
where $\varphi(v)(x)=(h(x),s_{x})$ for $x\in D$ with\\
$$\fsix{s_{x}:\quad{g^{-1}\lbrace{h(x)}\rbrace}}{C}{z\qquad}{v(z,x)}{t\qquad}{v(t,1_{x})}$$ 
and $\varphi(v)(u:x\rightarrow x')=(h(u),~w_{u})$ for $u$ in $D$ with $w_{u}:g^{*}(h(u))\rightarrow \underline{f}$ in $\Gpd/A$ defined by :\\
$$w_{u}(z,0)=s_{x}(z)=v(z,x)$$ 
$$w_{u}(z,1)=s_{x'}(z)=v(z,x')$$ 
$$w_{u}(-,1_{0})=s_{x}(-)=v(-,1_{x})$$
$$w_{u}(-,1_{1})=s_{x'}(-)=v(-,1_{x'})$$ 
$$w_{u}(-, \phi)=v(-,u)~.$$\\
It is a matter of easy calculations to check that $\varphi(v)$ is compatible with the involutions $\de$ (on $D$) and $\pi_{g}f$ (on $\text{dom}(\Pi_{g}f)$). We have to check that $\varphi$ is a bijection. Let define $\varphi^{-1}$ as follows :\\
$$\ffour{\varphi^{-1}: \GGpd/B(h, \Pi_{g}f)}{\GGpd/A(g^{*}h,f)}{\xymatrix{D\ar[rr]^{k}\ar[rdd]_{h}&&**[r]{\text{dom}(\Pi_{g}f)}\ar[ldd]^{\Pi_{g}f}\\&\circlearrowright&\\&B&}}{\xymatrix{g^{*}D\ar[rr]^{\varphi^{-1}(k)} \ar[rdd]_{g^{*}h}&&C\ar[ldd]^{f}\\&\circlearrowright&\\&A&}}$$ 
where $\varphi^{-1}(k)(z,x):=[\text{snd}(k(x)](z)$ for every $z\in A$ and $x\in D$ such that\\
$g(z)=h(x)$ and $\varphi^{-1}(k)(t:z\rightarrow z',~u:x\rightarrow x'):=[\text{snd}(k(u))](t,\phi)$ for every morphisms $t~\text{in}~A$ and $u\text{ in }D$ such that $g(t)=h(u)$. One can easily check that $\varphi^{-1}(k)$ is compatible with the involutions $\al\times \de$ on $g^{*}D$ and $\gm$ on $C$, and futhermore 
$$\varphi^{-1}\circ\varphi=1_{\GGpd/A}(g^{*}h,f)$$
and 
$$\varphi\circ\varphi^{-1}=1_{\GGpd/B}(h,\Pi_{g}f)$$ 
and the bijection $\varphi$ is natural in its arguments. \\
Finally, by \ref{pullbackstabilitytrivialcof} and by adjunction, one concludes that $\Pi_{g}$ preserves fibrations in the slice categories. In particular as a corollary $\Pi_{g}$ maps fibrations over $A$ in $\GGpd$ to fibrations over $B$ in $\GGpd$.\\ When the involutions involved in the statement of our present lemma  are identities we recover a theorem by Giraud (see \cite{Giraud64}, lemma 4.3 and theorem 4.4).
\end{proof}\bigskip

\begin{rmk}\label{rmk5}
The above lemma \ref{thm2} proves the required condition (4) in the definition \ref{def:ttfc} of a type-theoretic fibration category.
\end{rmk}\bigskip

It remains to prove condition (6) in \ref{def:ttfc}. We get rid of this with the two following lemmas.\medskip

\begin{lem}\label{3for2}
In $\GGpd$ if $g\circ f$ and $g$ are trivial cofibrations then $f$ is a trivial cofibration.  
\end{lem}
\begin{proof}
Since weak equivalences are levelwise and they have the 2-out-of-3 property, $f$ is a weak equivalence. Moreover $f$ is an injective-on-objects functor, hence $f$ is a levelwise trivial cofibration. Thanks to \ref{prop3} it suffices to prove that $f$ is surjective on fixed points. Let $x$ be a fixed point in $\text{cod}(f)$ then $g(x)$ is a fixed point in $\text{cod}(g) = \text{cod}(g\circ f)$ so, since $g\circ f$ is a trivial cofibration, there exists a fixed point $y$ in $\text{dom}(g\circ f) = \text{dom}(f)$ such that $g(f(y)) = g(x)$. Since $g$ is a levelwise cofibration one concludes that $f(y) = x$.   
\end{proof}

\begin{lem}\label{condition6}
In the following commutative diagram:
    \[\vcenter{\xymatrix@C=1,5cm{
        X\ar[r]\ar[dd] \pullbackcorner &
        Y\ar[r]\ar[dd] \pullbackcorner &
        Z\ar[dd]\\\\
        A\ar@{ >->}[r]^{\sim} \ar@{->>}@(dr,dl)[rr] &
        B\ar@{->>}[r] &
        C\\
        & &
      }}\]
    if $B\fib C$ and $A\fib C$ are fibrations, $A\acof B$ is an acyclic cofibration, and both squares are pullbacks (hence $Y\to Z$ and $X\to Z$ are fibrations by~\ref{item:cat3}), then $X\to Y$ is also an acyclic cofibration. In other words any pullback along any morphism of a trivial cofibration between fibrations is again a trivial cofibration between fibrations.
\end{lem}
\begin{proof}
According to (5) in \ref{def:ttfc} one can factor the morphism from $Z$ to $C$ as a trivial cofibration followed by a fibration. Thanks to the stability of fibrations under pullback and knowing that trivial cofibrations are stable under pullback along any fibration (\textit{cf} \ref{rightproperbis}), we can display the diagram given in the statement as follows,
$$\xymatrix@C=1,5cm@R=1,2cm{X\pullbackcorner\ar[dd]\ar[rr]\ar@{ >->}[rd]^*[@]{\hbox to -10pt{\rotatebox[origin=c]{-20}{$\sim$}}} & & Y\pullbackcorner\ar[dd]\ar@{ >->}[rd]^*[@]{\hbox to -8pt{\rotatebox[origin=c]{-20}{$\sim$}}}\ar@{->>}[rr] & & Z\ar@{ >->}[d]^*[@]{\hbox to 5pt{\rotatebox[origin=c]{0}{$\sim$}}}\\
 & \pullbackcorner\ar@{->>}[dl]\ar@{ >-->}[rr]^*[@]{\hbox to -10pt{$\sim$}}\ar@(ur,ul)@{->>}[rrr] & & \pullbackcorner\ar@{->>}[dl]\ar@{->>}[r] & Z'\ar@{->>}[d]\\
 A\ar@{ >->}[rr]^{\sim}\ar@{->>}@(dr,dl)[rrrr] & & B\ar@{->>}[rr] & & C\,. }$$\medskip
 
Hence \ref{3for2} allows us to conclude that $X\to Y$ is a trivial cofibration.  
\end{proof}\bigskip

\begin{thm}\label{cor1}
The category $\GGpd$, equipped with the subcategory given by projective fibrations, namely the levelwise isofibrations of groupoids, is a type-theoretic fibration category.\\
We denote this type-theoretic fibration category by $\GGpd_{\mathbf{proj}}$ and one says that the category $\GGpd$ is equipped with its projective type-theoretic fibration structure. \\
Assuming the initiality of the syntactic category of type theory it gives us a model of dependent type theory with dependent sums, dependent products and identity types in $\GGpd$.
\end{thm}
\begin{proof}
The required conditions (1),(2),(3) and (5) are straighforward. The lemmas \ref{thm2} and \ref{condition6} allow us to conclude that conditions (4) and (6) hold.
\end{proof}

\section{The construction of universes}
\label{sec:tcou}

We now move on to constructing universes (in the sense of \ref{def:univ}) in $\GGpd$ for each inaccessible cardinal $\kappa$ and each universe $\text{Gpd}_\Delta(V_\kappa)$ in $\Gpd$. Recall from \cite{groupoidmodel} that the universe $\text{Gpd}_\Delta(V_\kappa)$ in $\Gpd$ consists in the groupoid made of discrete groupoids with $\kappa$-small sets of objects and isomorphisms between them. \\ 
In this section we will lift this specific universe $\text{Gpd}_\Delta(V_\kappa)$ in $\Gpd$ by constructing a universe in $\GGpd$ denoted accordingly by $\widetilde{U_{\text{Gpd}_\Delta(V_\kappa)}}$, $U_{\text{Gpd}_\Delta(V_\kappa)}$ and $p_{\text{Gpd}_\Delta(V_\kappa)}:\widetilde{U_{\text{Gpd}_\Delta(V_\kappa)}}\rightarrow U_{\text{Gpd}_\Delta(V_\kappa)}$ (shorten later by $\widetilde{U_\Delta}$, $U_\Delta$ and $p_\Delta$) and prove the requirement 
$(i)$ of \ref{def:univ} in \ref{lem6}, requirement $(ii)$ in \ref{lem7} and requirement $(iii)$ in \ref{lem8} respectively. \bigskip

\begin{rmk}\label{rmk6} 
One has a \enquote{natural} candidate for lifting in $\GGpd$ any universe $V_\kappa$ in $\Gpd$ consisting in the groupoid whose objects are $\kappa$-small (not necessarily discrete) groupoids with isomorphisms between them, where the $\kappa$-smallness means that both the set of objects and the homsets have cardinality strictly less than $\kappa$. Indeed, for any inaccessible cardinal $\kappa$ the universe $V_\kappa$ in the groupoid model classifies split isofibrations with $\kappa$-small fibers. So, projective fibrations being levelwise fibrations we want to find a universal fibration in $\GGpd$ which classifies projective fibrations that are levelwise split isofibrations with $\kappa$-small fibers. We achieve this result in \ref{lem5}.\\
\end{rmk}
To achieve this goal define ${\widetilde{U_{V_\kappa}},U_{V_\kappa}}\in\GGpd$ and $p_{V_\kappa}:\widetilde{U_{V_\kappa}}\rightarrow U_{V_\kappa}$, shortened immediately by $\widetilde{U}$,$U$ and $p:\widetilde{U}\rightarrow U$ :\\
\begin{itemize}
\item Objects of the groupoid $\widetilde{U}$ are dependent tuples of the form $(A_{0},A_{1},a,\varphi)$ where $A_{0},A_{1}$ are $\kappa$-small groupoids, $a\in A_{0}$, and $\varphi:A_{0}\xrightarrow{\cong}A_{1}$ is an isomorphism in $\Gpd$.\\\item Morphisms in $\widetilde{U}$ between $(A_{0},A_{1},a,\varphi)$ and $(B_{0},B_{1},b,\psi)$ are tuples of the form $(\rho_{0},\rho_{1},\tau,\alpha)$ : 
$$\xymatrix{(A_{0},a)\dia^{\alpha}\ar[rr]^\cong_\varphi\ar[dd]^{\rotatebox[origin=c]{-90}{$\cong$}}_{\rho_{0}}&&A_{1}\ar[dd]_{\rho_{1}}^{\rotatebox[origin=c]{-90}{$\cong$}}\\\\(B_{0},b)\ar[rr]^\cong_\psi&&B_{1}}$$\\
such that :\\
\begin{savenotes}
\begin{tabular}{ll}\\{$\rho_{0}:A_{0}\xrightarrow{\cong}B_{0}$}{~is an isomorphism in $\Gpd$}\\\\{$\rho_{1}:A_{1}\xrightarrow{\cong}B_{1}$}{~is an isomorphism in $\Gpd$}\\\\{$\tau:\rho_{0}(a)\xrightarrow{\cong}b$}{~is an isomorphism in $B_{0}$}\\\\{$\alpha:\psi\circ\rho_{0}\Rightarrow\rho_{1}\circ\varphi$}{~is a natural isomorphism}.\\ 
\end{tabular}
\end{savenotes}
\end{itemize}
The composition in $\tilde{U}$ is given by : 
$$(\rho'_{0},\rho'_{1},\tau',\alpha')\circ(\rho_{0},\rho_{1},\tau,\alpha):=(\rho'_{0}\circ\rho_{0},\rho'_{1}\circ\rho_{1},\tau'\circ\rho'_{0}(\tau),\rho'_{1}(\alpha)\circ\alpha'_{\rho_{0}})$$ 
where the component of $\rho'_{1}(\alpha)\circ\alpha'_{\rho_{0}}$ at $x \in A_{0}$ is $\rho'_{1}(\alpha_{x})\circ\alpha'_{\rho_{0}(x)}$. Note that $\widetilde{U}$ is a groupoid. Indeed, the inverse of the morphism $(\rho_{0},\rho_{1},\tau,\alpha)$ is given by : 
$$(\rho_{0},\rho_{1},\tau,\alpha)^{-1}:=(\rho_{0}^{-1},\rho_{1}^{-1},\rho_{0}^{-1}(\tau^{-1}),\rho_{1}^{-1}(\alpha^{-1}_{\rho_{0}^{-1}}))$$
where the component of $\rho_{1}^{-1}(\alpha^{-1}_{\rho_{0}^{-1}})$ at $x\in B_{0}$ is $\rho_{1}^{-1}(\alpha^{-1}_{\rho_{0}^{-1}(x)})$.\\
We provide $\widetilde{U}$ with the involution $\tilde{u}$ : 
$$\fsix{\tilde{u}:\qquad\qquad\widetilde{U}\qquad}{\qquad\widetilde{U}}{(A_{0},A_{1},a,\varphi)}{(A_{1},A_{0},\varphi(a),\varphi^{-1})}{(A_{0},A_{1},a,\varphi)\xrightarrow{(\rho_{0},\rho_{1},\tau,\alpha)}(B_{0},B_{1},b,\psi)}{(\rho_{1},\rho_{0},\psi(\tau)\circ\alpha^{-1}_{a},\psi^{-1}(\alpha^{-1}_{\varphi^{-1}}))}$$ 
where the component of $\psi^{-1}(\alpha^{-1}_{\varphi^{-1}})$ at $x\in A_{1}$ is $\psi^{-1}(\alpha^{-1}_{\varphi^{-1}(x)})$. One denotes by $U$ the \enquote{unpointed} version of $\widetilde{U}$, \textit{i.e.} objects are of the form $(A_{0},A_{1},\varphi)$ and morphisms of the form $(\rho_{0},\rho_{1},\alpha)$, with its corresponding involution $u$. We define our universal morphism $p$ in $\GGpd$ by : 
$$\fsix{p:\qquad\qquad{\widetilde{U}}\qquad}{\qquad U}{(A_{0},A_{1},a,\varphi)}{(A_{0},A_{1},\varphi)}{(\rho_{0},\rho_{1},\tau,\alpha)}{(\rho_{0},\rho_{1},\alpha)~.}$$
We want to prove that $p:\widetilde{U}\rightarrow U$ is a universe.\\
\begin{lem}\label{lem4}
Our morphism $p:\widetilde{U}\rightarrow U$ is a fibration in $\GGpd$ between fibrant objects.\\
\end{lem}
\begin{proof}\label{prooflem4}
According to our \ref{rmk1} the groupoids $\widetilde{U}$ and $U$ are fibrant objects. Now, $p$ is a levelwise fibration, actually we can even equip $p$ with a split cleavage denoted $c_{p}$ as follows : being given $(\rho_{0},\rho_{1},\alpha)$ an isomorphism in $U$
$$\xymatrix{A_{0}\dia^{\alpha}\ar[rr]^\cong_\varphi\ar[dd]^{\rotatebox[origin=c]{-90}{$\cong$}}_{\rho_{0}}&&A_{1}\ar[dd]_{\rho_{1}}^{\rotatebox[origin=c]{-90}{$\cong$}}\\\\B_{0}\ar[rr]^\cong_\psi&&B_{1}}$$ 
and $(A_{0},A_{1},a,\varphi)$ an element in the fiber of $(A_{0},A_{1},\varphi):=\text{dom}(\rho_{0},\rho_{1},\alpha)$ by $p$, we take $c_{p}((\rho_{0},\rho_{1},\alpha),(A_{0},A_{1},a,\varphi)):=(\rho_{0},\rho_{1},1_{\rho_{0}(a)},\alpha)$ namely 
$$\xymatrix{(A_{0},a)\dia^{\alpha}\ar[rr]^\cong_\varphi\ar[dd]^{\rotatebox[origin=c]{-90}{$\cong$}}_{\rho_{0}}&&A_{1}\ar[dd]_{\rho_{1}}^{\rotatebox[origin=c]{-90}{$\cong$}}\\\\(B_{0},\rho_{0}(a))\ar[rr]^\cong_\psi&&B_{1}~.}$$ 
It is easily seen that $c_{p}$ is a split cleavage. \end{proof}\bigskip
\begin{lem}\label{lem5} 
 The morphism $p$ classifies the projective fibrations in $\GGpd$ that are levelwise split fibrations in $\Gpd$ with $\kappa$-small fibers.
\end{lem}
\begin{proof}\label{prooflem5}
We have to prove that for a morphism $f:A\rightarrow C$ in $\GGpd$, the following are equivalent:\\
\begin{enumerate}[label=(\roman*)] 
\item $f$ is a pullback of $p$\\
\item $\underline{f}$ is a split fibration in $\Gpd$ with $\kappa$-small fibers.
\end{enumerate} 
Assume that $f$ is a pullback of $p$,
$$\xymatrix{A\ar[rr]^{p^{*}g}\ar[dd]_f\pullbackcorner&&\widetilde{U}\ar[dd]^p\\\\C\ar[rr]_g&&U~.}$$ 
Let $\alpha:x\xrightarrow{\cong}y$ be an isomorphism in $C$ and $z\in A$ such that $f(z)=x$. One denotes by $c_{f}$ the intended split cleavage of $\underline{f}$. One takes $c_{f}(\alpha,z):=(\alpha,~c_{p}(g(\alpha),~p^{*}g(z)))$ which is an isomorphism in $A$ above $\alpha$ by $f$ with domain $x$ (note that we identify $A$ with the isomorphic groupoid $C\times_{U}\widetilde{U}$). This cleavage is split, indeed one has
\begin{align*} 
{c_{f}(1_{x},z)}&={(1_{x},~c_{p}(1_{g(x)},~p^{*}g(z)))}\\&={(1_{x},~1_{p^{*}g(z)})}
\end{align*} 
and for $x \xrightarrow{\alpha}y\xrightarrow{\beta}w$ and $z\in f^{-1}\lbrace{x}\rbrace$ one has :\\
\begin{align*} c_{f}(\beta\circ\alpha,z)&=(\beta\circ\alpha,~c_{p}(g(\beta\circ\alpha),~p^{*}g(z)))\\
&=(\beta\circ\alpha,~c_{p}(g(\beta)\circ g(\alpha),~p^{*}g(z)))\\
&=(\beta\circ\alpha,~c_{p}(g(\beta),~\text{cod}(c_{p}(g(\alpha),~p^{*}g(z))))\circ c_{p}(g(\alpha),~p^{*}g(z)))\\
&=(\beta,~c_{p}(g(\beta),~\text{cod}(c_{p}(g(\alpha),~p^{*}g(z)))))\circ(\alpha,~c_{p}(g(\alpha),~p^{*}g(z)))\\
&=(\beta,~c_{p}(g(\beta),~p^{*}g(\text{cod}(c_{f}(\alpha,z)))))\circ(\alpha,~c_{p}(g(\alpha),~p^{*}g(z)))\\
&=c_{f}(\beta,~\text{cod}(c_{f}(\alpha,z)))\circ c_{f}(\alpha,z)~.
\end{align*}\\
The reader can check that the fibers of $\underline{f}$ are $\kappa$-small.
Thus one has $(i)\Rightarrow(ii)$. Conversely, assume that $\underline{f}$ is a fibration of groupoids, with $\kappa$-small fibers, equipped with a split cleavage denoted $c_{f}$. One has to display $f$ as a pullback of $p$ in $\GGpd$ along a morphism we choose to denote $g$. So define $g$ as follows : 
$$\fsix{g:\qquad C\quad}{U}{x\quad}{(f^{-1}\lbrace{x}\rbrace,~f^{-1}\lbrace{\gm(x)}\rbrace,~\varphi_{x})}{x\xrightarrow{\sigma}y}{(\rho_{\sigma,0},\rho_{\sigma,1},\be_{\sigma})}$$ 
where by $f^{-1}\lbrace{x}\rbrace$ (\textit{resp.}~$f^{-1}\lbrace{\gm(x)}\rbrace$) we denote the subgroupoid of $A$ whose objects are objects of $A$ above $x$ (\textit{resp.} $\gm(x)$) and morphisms are morphisms in $A$ above $1_{x}$ (\textit{resp.} $1_{\gm(x)}$). Moreover for $x\in C,~\varphi_{x}$ is the following isomorphism : 
$$\fsix{\varphi_{x}:\qquad f^{-1}\lbrace{x}\rbrace}{f^{-1}\lbrace{\gm(x)}\rbrace}{z~\quad}{\al(z)}{z\xrightarrow{\tau}z'}{\al(\tau)~.}$$ 
The inverse isomorphism of $\varphi_{x}$ is given by : 
$$\fsix{\varphi^{-1}_{x}:\qquad f^{-1}\lbrace{\gm(x)}\rbrace}{f^{-1}\lbrace{x}\rbrace}{z~\quad}{\al(z)}{z\xrightarrow{\tau}z'}{\al(\tau)~.}$$ 
Actually one has $\varphi^{-1}_{x}:=\varphi_{\gm(x)}.$ For $\sigma:x\rightarrow y$ in $C$ one defines $\rho_{\sigma,0}$ by : 
$$\fsix{\rho_{\sigma,0}:\qquad f^{-1}\lbrace{x}\rbrace}{f^{-1}\lbrace{y}\rbrace}{z~\quad}{\text{cod}(c_{f}(\sigma,z))}{z\xrightarrow{\tau}z'}{c_{f}(\sigma,z')\circ\tau\circ c_{f}(\sigma,z)^{-1}~.}$$ 
Since $f(\tau)=1_{x}$ one has $f(c_{f}(\sigma,z')\circ\tau\circ c_{f}(\sigma,z)^{-1})=1_{y}$. In the same way one has 
$$\fsix{\rho_{\sigma,1}:\qquad f^{-1}\lbrace{\gm(x)}\rbrace}{f^{-1}\lbrace{\gm(y)}\rbrace}{z~\quad}{\text{cod}(c_{f}(\gm(\sigma),z))}{z\xrightarrow{\tau}z'}{c_{f}(\gm(\sigma),z')\circ\tau\circ c_{f}(\gm(\sigma),z)^{-1}~.}$$\\
Last, we define the natural isomorphism $\be_{\sigma}$ by choosing for the component at $z\in f^{-1}\lbrace{x}\rbrace$ the isomorphism $\be_{\sigma,z}:=c_{f}(\gm(\sigma),\al(z))\circ \al(c_{f}(\sigma,z))^{-1}$. The suspicious reader can check in order that $\rho_{\sigma,0}$ and $\rho_{\sigma,1}$ are functors and more specifically isomorphisms, that $\be_{\sigma}$ is a natural isomorphism, and $g$ is functorial and compatible with the involutions. It remains to check that $A$ is isomorphic to $C\times_{U}\widetilde{U}$ above $C$, \textit{i.e.} we need to provide an isomorphism $\chi:A\rightarrow C\times_{U}\widetilde{U}$ such that $pr_{1}\circ\chi=f$ where $pr_{1}:C\times_{U}\widetilde{U}\rightarrow C$ is the first projection. Let define $\chi$ as follows : 
$$\fsix{\chi:\qquad A\quad}{C\times_{U}\widetilde{U}}{x\quad}{(f(x),~(f^{-1}\lbrace{f(x)}\rbrace,~f^{-1}\lbrace{\gm(f(x))}\rbrace,x,\varphi_{f(x)}))}{x\xrightarrow{\sigma}y}{(f(\sigma),~(\rho_{f(\sigma),0},~\rho_{f(\sigma),1},~\sigma\circ c_{f}(f(\sigma),x)^{-1},~\be_{f(\sigma)}))~.}$$
The functor $\chi$ is compatible with the involutions $\al$ and $\gm\times\widetilde{u}$ and it is actually an isomorphism with $\chi^{-1}$ given by : 
$$\fsix{\chi^{-1}:\qquad\qquad\qquad\qquad\qquad C\times_{U}\widetilde{U}\qquad}{A}{(x,~(f^{-1}\lbrace{x}\rbrace,~f^{-1}\lbrace{\gm(x)}\rbrace,z,~\varphi_{x}))}{z}{(\sigma,(\rho_{\sigma,0},\rho_{\sigma,1},\tau,\be))}{\tau\circ c_{f}(\sigma,z)~.}$$
Thus one has $(ii)\Rightarrow(i)$. 
\end{proof}\bigskip

\begin{rmk}\label{rmk7}
The previous lemma expresses in which sense we have the \enquote{right} candidate for being a universe in $\GGpd$ (see \ref{rmk6}).\\
\end{rmk}\bigskip
\begin{defi}\label{defi4}(Shulman)
A pullback of the universal fibration $p$ is called a small fibration. We denote such a small fibration with a three ends arrow $\xymatrix{\ar[r]|->{\SelectTips{2cm}{}\object@{>>}}|-->{\SelectTips{eu}{}}&}$.\\
\end{defi}\bigskip

\begin{lem}\label{lem6}
Small fibrations are closed under composition and contain the identities.\\
\end{lem}
\begin{proof}\label{prooflem6}
Straighforward with \ref{lem5} in mind.
\end{proof}\bigskip
\begin{lem}\label{lem7}
If $\xymatrix{f:B\ar[r]|->{\SelectTips{2cm}{}\object@{>>}}|-->{\SelectTips{eu}{}}&A}$ and $\xymatrix{g:A\ar[r]|->{\SelectTips{2cm}{}\object@{>>}}|-->{\SelectTips{eu}{}}&C}$ are small fibrations in $\GGpd$, so is $\Pi_{g}f$.\\
\end{lem}
\begin{proof}\label{prooflem7}
It suffices to prove that $\Pi_{g}f$ can be equipped with a split cleavage (we know this is a fibration thanks to \ref{thm2}). Let $u:y\rightarrow y'$ be an isomorphism in $C$ and $(y,s)$ be an object of $\text{dom}(\Pi_{g}f)$. Since $f$ and $g$ are small fibrations we denote by $c_{f}$ (\textit{resp.} $c_{g}$) a split cleavage for $f$ (\textit{resp.} $g$). Let define $s':g^{-1}\lbrace{y'}\rbrace\rightarrow B$ a partial section of $f$ by :\\\\
\begin{tabular}{lll}
$s':\qquad g^{-1}\lbrace{y'}\rbrace$&$\longrightarrow$&$B$\\
$\qquad\qquad\quad x$&$\longmapsto$&$\text{cod}(c_{f}(c_{g}(u^{-1},x)^{-1},~s(\text{cod}(c_{g}(u^{-1},x)))))$\\
$\qquad \qquad x\xrightarrow{h}x'$&$\longmapsto$&$c_{f}(c_{g}(u^{-1},x')^{-1},s(\text{cod}(c_{g}(u^{-1},x'))))$\\
&&$\circ~s(c_{g}(u^{-1},x')\,\circ\, h\,\circ c_{g}(u^{-1},x)^{-1})$\\
&&$\circ~c_{f}(c_{g}(u^{-1},x)^{-1},~s(\text{cod}(c_{g}(u^{-1},x))))^{-1}$
\end{tabular}\\
In picture one has 
$$\xymatrix{s(\text{cod}(c_{g}(u^{-1},x'))))\ar[rrrrrrrr]^-{c_{f}(c_{g}(u^{-1},x')^{-1},~s(\text{cod}(c_{g}(u^{-1},x'))))}&&&&&&&&s'(x')\\\\s(\text{cod}(c_{g}(u^{-1},x)))\ar[rrrrrrrr]^-{c_{f}(c_{g}(u^{-1},x)^{-1},~s(\text{cod}(c_{g}(u^{-1},x))))}\ar[uu]|{s(c_{g}(u^{-1},x')\,\circ\, h\,\circ c_{g}(u^{-1},x)^{-1})}&&\ar[d]|->{\SelectTips{2cm}{}\object@{>>}}|-->{\SelectTips{eu}{}}_{f}&&&&&&s'(x)\ar@{-->}_{s'(h)}[uu]\\\text{cod}(c_{g}(u^{-1},x))\ar[rrrrrrrr]^-{c_{g}(u^{-1},x)^{-1}}&&\ar[d]|->{\SelectTips{2cm}{}\object@{>>}}|-->{\SelectTips{eu}{}}_{g}&&&&&&x\\y\ar[rrrrrrrr]^-{u}&&&&&&&&y'}$$ 
Now, we want to define $v:g^{*}u\rightarrow \underline{f}$ in $\Gpd/A$ with $g^{*}u$ defined as usual (see the proof of \ref{thm2}) by the following pullback
$$\xymatrix{A\times_{C}\I\ar[rr]\ar[dd]_{g^{*}u}\pullbackcorner&&\I\ar[dd]^u\\\\A\ar[rr]_g&&C}$$ 
so we will get $(u,v):(y,s)\rightarrow(y',s')$ a morphism in $\text{dom}(\Pi_{g}f)$ above $u$. Take $v_{|A\times_{C}\lbrace{0}\rbrace}:=s$ and $v_{|A\times_{C}\lbrace{1}\rbrace}:=s'$ and for every $h:z\rightarrow z'$ in $A$ such that $g(h)=u$ take 
$$v(h,\phi):=c_{f}(c_{g}(u^{-1},z')^{-1},~s(\text{cod}(c_{g}(u^{-1},z'))))~\circ~s(c_{g}(u^{-1},z')~\circ~h)$$
where $\phi$ is the non-identity isomorphism in $\I$.
In picture one has 
$$\xymatrix{s(z)\ar[d]|{s(c_{g}(u^{-1},z')\circ ~h)}&&&&&&&&\\s(\text{cod}(c_{g}(u^{-1},z')))\ar[rrrrrrrr]^-{c_{f}(c_{g}(u^{-1},z')^{-1},~s(\text{cod}(c_{g}(u^{-1},z'))))}&&&&\ar[d]|->{\SelectTips{2cm}{}\object@{>>}}|-->{\SelectTips{eu}{}}^{f}&&&&s'(z')\\z\ar@/^/[rrrrrrrrd]^-{h}\ar@{-->}[dd]|{c_{g}(u^{-1},z')\circ~h}&&&&&&&&\\&&&&&&&&z'\\\text{cod}(c_{g}(u^{-1},z'))\ar@/_/[rrrrrrrru]^-{c_{g}(u^{-1},z')^{-1}}&&&&\ar[d]|->{\SelectTips{2cm}{}\object@{>>}}|-->{\SelectTips{eu}{}}^{g}&&&&\\y\ar[rrrrrrrr]^-{u}&&&&&&&&y'~.}$$\\
By taking $c_{\Pi_{g}f}(u,(y,s)):=(u,v)$ the reader can check that $\Pi_{g}f$ is equipped with a split cleavage. Moreover the map $\Pi_{g}f$ has $\kappa$-small fibers and so is a small fibration. Indeed, remember from the construction in \ref{thm2} that the underlying map of $\Pi_{g}f$ is nothing but $\Pi_{\underline{g}}\underline{f}$ and $\Pi_{\underline{g}}$ stands for the right adjoint of the pullback functor $\underline{g}^*$ in $\Gpd$. Hence one concludes the smallness of $\Pi_{g}f$ from the groupoid model using the smallness of $\underline{g}$ and the smallness of $\underline{f}$.  
\end{proof}\bigskip

In the following paragraph we want to prove for $p_\Delta$ the required condition $(iii)$ in definition 6.12 of the universe in \cite{shulman:invdia} (for a general candidate for universe $p: \widetilde{U}\rightarrow U$ it is unlikely that this required condition $(iii)$ holds due to the smallness requirement of the fibers in the characterization of small fibrations). Be aware that from now on and until \ref{thm4} by a small fibration we mean a pullback of $p_\Delta$. Following the remark 6.13 in \cite{shulman:invdia} we prove the equivalent condition that the following diagonal map $\Delta_f$,
$$\xymatrix@=1,5cm{E\ar@/^/@{=}[rrd]\ar@/_/@{=}[rdd]\ar@{-->}[rd]^{\Delta_f} & & \\
 & E\times_B E\pullbackcorner\ar[r]^{p_2}\ar[d]_{p_1} & E\ar[d]^f \\
 & E\ar[r]_f & B\,,}$$
factors as a trivial cofibration followed by a small fibration, for any small fibration $f$ in $\GGpd$. \\
\bigskip

\begin{defi}\label{smallfib}
A fibration in a model category $\mathscr{C}$ with a unique solution to any right lifting problem with respect to any trivial cofibration is called a discrete fibration.
\end{defi}\bigskip

\begin{notn}
Let $\mathscr{C}$ be a cocomplete category and $J$ be a set of morphisms in $\mathscr{C}$. We denote by $\overline{J}$ the weakly saturated class of morphisms generated by $J$, namely the smallest class of morphisms containing the elements in $J$ and which is weakly saturated (see \cite{luriehtt} def. A.1.2.2). Remember from chapter \ref{sec:somc} (see definition \ref{weaklyortho}) that the notation $J^\boxslash$ denotes the class of morphisms with the right lifting property with respect to any morphism in $J$. Now let $J^{\blacksquare}$ denotes the class of morphisms with the unique right lifting property with respect to any morphism in $J$.
\end{notn}\bigskip

\begin{prop}\label{smallfibbis}
An element of $J^{\blacksquare}$ is also an element of $\overline{J}^{\blacksquare}$
\end{prop}
\begin{proof}
Let $f$ be an element of $J^\blacksquare$. We know (see \cite{luriehtt} Corollary A.1.2.7) that $f$ is an element of $\overline{J}^{\boxslash}$ (or equivalently any element $j$ of $\overline{J}$ belongs to ${}^{\boxslash}f$). It remains to prove  the uniqueness condition. \\
To achieve this we consider successively the case of a pushout of a map in $J$, the case of a transfinite composition of such maps, and last the case of a retract.\bigskip

First, assume that g belongs to $J$ and consider the pushout $i$ of $g$ along a morphism $h$,
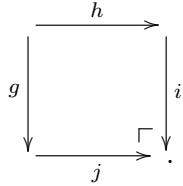
$$\xymatrix@=1,5cm{\ar[r]^h\ar[d]_g & \ar[d]^i \\
\ar[r]_j & \pushoutcorner \;.}$$
Consider the following lifting problem,
$$
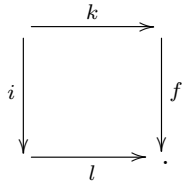
\xymatrix@=1,5cm{\ar[r]^k\ar[d]_i & \ar[d]^f \\
\ar[r]_l & \;.}$$
Let $m$, $n$ be two diagonal fillers for this lifting problem. Note that we can form the following lifting problem,
$$
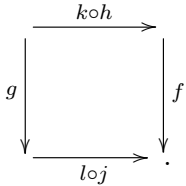
\xymatrix@=1,5cm{\ar[d]_g\ar[r]^{k\circ h} & \ar[d]^f \\
\ar[r]_{l\circ j} & \;.}$$
Also note that $m\circ j$, $n\circ j$ are two diagonal fillers for this lifting problem. Hence $m\circ j$ equals $n\circ j$.\\
By the uniqueness part in the universal property of the pushout and considering the following universal problem, one concludes that $m$ equals $n$,
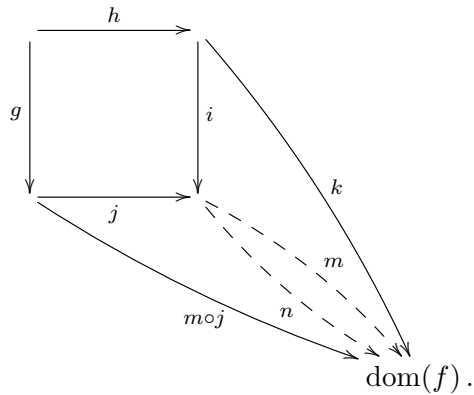
$$\xymatrix@=2cm{\ar[d]_g\ar[r]^h & \ar[d]^i\ar@/^/[rdd]^k & \\
\ar@/_/[rrd]_{m\circ j}\ar[r]_j & \ar@{-->}@/^/[rd]^{m} \ar@{-->}@/_/[rd]_{n} & \\
& & \text{dom}(f)\,.}$$\bigskip

Second, consider an ordinal $\lambda$ and $X$ a $\lambda$-sequence in ${}^{\blacksquare}f$.\\
Our goal consists in proving that $X_{0,\la}$ belongs to ${}^{\blacksquare}f$.\\
We do so by a transfinite induction on $\lambda$.
\begin{itemize}
\item $\la = 0$\\
In this case the transfinite composition is $id_{\bold{0}}$ and the uniqueness is obvious. Indeed, for two diagonal fillers $j_1$ and $j_2$ as follows,
$$
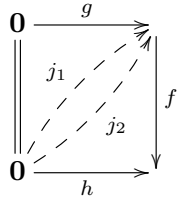
\xymatrix@=1,5cm{\bold{0}\ar[r]^g\ar@{=}[d] & \ar[d]^f \\
\bold{0}\ar[r]_h\ar@{-->}@/^/[ru]^{j_1}\ar@{-->}@/_/[ru]_{j_2} & }$$
one has the equality $j_1 = j_2 = g$.
\item $\la = \gm+1$\\
Consider two diagonal fillers in the following lifting problem,
$$
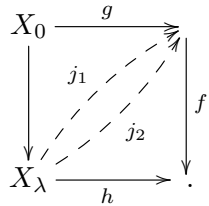
\xymatrix@=1,5cm{X_0\ar[r]^g\ar[d] & \ar[d]^f\\
X_\la \ar[r]_h\ar@/^/@{-->}[ru]^{j_1}\ar@/_/@{-->}[ru]_{j_2} & \;.} $$
The induction hypothesis states that for every $\gm$-sequence $X$ in ${}^{\blacksquare}f$ the composition $X_{0,\gm}$ belongs to ${}^{\blacksquare}f$.\\
We can display this lifting problem as follows,
$$
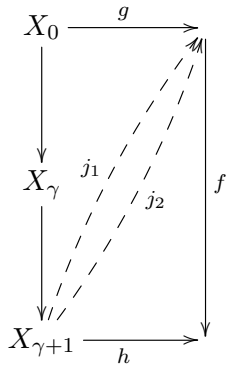
\xymatrix@=1,5cm{X_0\ar[d]\ar[r]^g & \ar[dd]^f \\
X_{\gm}\ar[d] & \\
 X_{\gm+1}\ar@{-->}@/^/[ruu]^{j_1}\ar@{-->}@/_/[ruu]_{j_2}\ar[r]_h & }$$
where $X_{0,\gm}$ is the composition of $X_{|[0,\gm]}$. By induction hypothesis one has that $j_1\circ X_{\gm,\gm+1}$ equals $j_2\circ X_{\gm,\gm+1}$.\\
Now consider the following lifting problem,
$$\xymatrix@=1,5cm{X_{\gm}\ar[d]_{X_{\gm,\gm+1}}\ar[r]^{j_1\circ X_{\gm,\gm+1}} & \ar[d]^f \\
X_{\gm+1}\ar[r]_h & } $$ 
and note that $j_1$ and $j_2$ are two diagonal fillers for this lifting problem. By assumption on the $\la$-sequence $X$ one concludes that $j_1$ equals $j_2$.
\item $\la$ limit\\
The induction hypothesis states that for every $\be<\la$ and every $\be$-sequence $X$ in ${}^{\blacksquare}f$ the composition $X_{0,\be}$ belongs to ${}^{\blacksquare}f$.\\
Let $j_1$ and $j_2$ be two diagonal fillers as follows,
$$\xymatrix@=1,5cm{X_0\ar[d]\ar[r]^g & \ar[d]^f \\
X_\la \ar[r]_h\ar@{-->}@/^/[ru]^{j_1}\ar@{-->}@/_/[ru]_{j_2} & \;.}$$
Let $\be$ be any ordinal strictly less than $\la$, we can display the square above as follows,
$$\xymatrix@=1,5cm{X_0\ar[r]^g\ar[d]_{X_{0,\be}} & \ar[dd]^f \\
X_\be\ar[d] & & \\
X_\la \ar[r]_h\ar@{-->}@/^/[ruu]^{j_1}\ar@{-->}@/_/[ruu]_{j_2} & \;.}$$
Note that $X_{0,\be}$ is the composition of the $\be$-sequence $X_{|[0,\be]}$. The maps $j_1\circ X_{\be,\la}$ and $j_2\circ X_{\be,\la}$ are two diagonal fillers for the following lifting problem,
$$
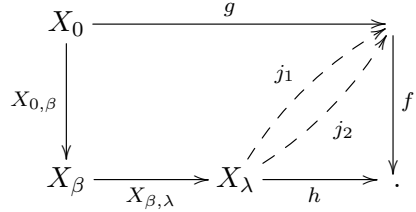
\xymatrix@=1,5cm{X_0\ar[rr]^g \ar[d]_{X_{0,\be}} & & \ar[d]^f \\
X_\be\ar[r]_{X_{\be,\la}} & X_\la \ar[r]_h\ar@{-->}@/^/[ru]^{j_1}\ar@{-->}@/_/[ru]_{j_2} & \;.}$$
Hence by the induction hypothesis $j_1\circ X_{\be,\la}$ equals $j_2\circ X_{\be,\la}$ for any $\be<\la$. Note that $X_\la$ is isomorphic to $\underset{\be<\la}{\text{colim}}X_\be$. Together the maps $j_1\circ X_{\be,\la}$ form a cocone with cobase $\text{dom}(f)$. Indeed, let $\be'$ and $\be$ be two ordinals satisfying $\be'\leqslant\be<\la$ one has 
$$(j_1\circ X_{\be,\la})\circ X_{\be',\be} = j_1\circ (X_{\be,\la}\circ X_{\be',\be}) = j_1\circ X_{\be',\la}\,.$$ 
So by the uniqueness part in the universal property of the cocone one has $j_1$ equals $j_2$.
\end{itemize}\bigskip

Last, assume that $f$ belongs to $\lbrace g\rbrace^{\blacksquare}$ and consider a retract $h$ of $g$,
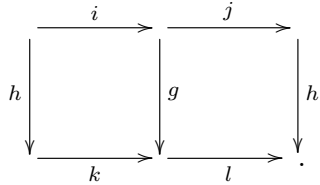
$$\xymatrix@=1,5cm{\ar[r]^i\ar[d]_h & \ar[r]^j\ar[d]^g & \ar[d]^h \\
\ar[r]_k & \ar[r]_l & \;.}$$
Consider the following lifting problem,
$$
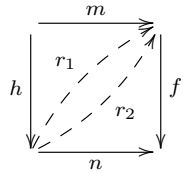
\xymatrix@=1,5cm{\ar[r]^m\ar[d]_h & \ar[d]^f \\
\ar[r]_n\ar@{-->}@/^/[ru]^{r_1}\ar@{-->}@/_/[ru]_{r_2} & }$$
with $r_1$ and $r_2$ two diagonal fillers.\\
We can look at the following lifting problem,
$$
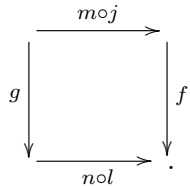
\xymatrix@=1,5cm{\ar[r]^{m\circ j}\ar[d]_g & \ar[d]^f \\
\ar[r]_{n\circ l} & \;.}$$
The maps $r_1\circ l$ and $r_2\circ l$ are two diagonal fillers for this lifting problem. Hence $r_1\circ l$ equals $r_2\circ l$. So $r_1\circ l\circ k$ equals $r_2\circ l\circ k$, namely $r_1$ equals $r_2$.
\end{proof}\bigskip

\begin{rmk}
Note that a small fibration $f$ in $\Gpd$, namely a fibration (with small fibers) equipped with a split cleavage, is nothing but a lluf subcategory $\mathscr{C}$ of $\text{dom}(f)$ such that $f_{|\mathscr{C}}$ is a discrete fibration.
\end{rmk}\bigskip

\begin{lem}\label{minimalfib}
Let $f$ be any morphism in $\GGpd$. The morphism $f$ is a pullback of $p_\Delta$ if and only if $\underline{f}$ is a discrete fibration in $\Gpd$ and it has (discrete) fibers with small sets of objects.
\end{lem}\bigskip
\begin{proof}
Assume that $f$ is a pullback of $p_\Delta$,
$$\xymatrix@=1,5cm{E\pullbackcorner\ar[r]^k\ar[d]_f & \widetilde{U_\Delta}\ar[d]^{p_\Delta}\\
B\ar[r]_l & U_\Delta\,.}$$
Thanks to \ref{lem5} it is enough to prove that two diagonal fillers in any lifting problem with respect to $\underline{f}$ are equal. Hence consider a right lifting problem with two diagonal fillers,
$$\xymatrix@=1,5cm{\bold{1}\ar[r]^g\ar[d]_i & E\ar[d]^{\underline{f}} \\
\I\ar[r]_h\ar@/^/[ru]^{j_1}\ar@/_/[ru]_{j_2} & B\,,}$$
where the map $i$ is our generating trivial cofibration in $\Gpd$. One can display the following lifting problem with respect to $\underline{p_\Delta}$,
$$\xymatrix@=1,5cm{\bold{1}\ar[r]^{\underline{k}\circ g}\ar[d]_i & \widetilde{U_\Delta}\ar[d]^{\underline{p_\Delta}} \\
\I\ar[r]_{\underline{l}\circ h} & U_\Delta\,.}$$
Note that the maps $j_1\circ \underline{k}$ and $j_2\circ \underline{k}$ are two diagonal fillers for this last lifting problem but the reader can easily check that $\underline{p_\Delta}$ is a discrete fibration, hence $j_1\circ \underline{k}$ and $j_2\circ \underline{k}$ are equal. So the maps $j_1$ and $j_2$ are equal by the uniqueness part in the universal property of the pushout applied to the following universal problem,
$$\xymatrix@=1,5cm{\I\ar@{-->}[rd]\ar@/^/[rrd]^{j_1\circ \underline{k}}\ar@/_/[rdd]_h & & \\
& E\pullbackcorner\ar[r]^{\underline{k}}\ar[d]_{\underline{f}} & \widetilde{U_\Delta}\ar[d]^{\underline{p_\Delta}} \\
& B\ar[r]_{\underline{l}} & U_\Delta\,.}$$\medskip

Conversely, assume that the underlying map of $f$ is a discrete fibration with small discrete fibers then the lemma \ref{lem5} allows us to conclude that $f$ is a pullback of $p_{\Delta}$.
\end{proof}

\begin{lem}\label{lem8}
For any small fibration f, the diagonal map $\Delta_f$ is a small fibration. As an obvious consequence $\Delta_f$ factors as a trivial cofibration (the identity map) followed by a small fibration ($\Delta_f$ itself).
\end{lem}
\begin{proof}
Let $f$ be a small fibration from $E$ to $B$. We prove that $\underline{\Delta_f}$ is a discrete fibration in $\Gpd$. Indeed, consider the following lifting problem in $\Gpd$,
$$\xymatrix@=1,5cm{\bold{1}\ar[r]\ar[d]_i & E\ar[d]^{\underline{\Delta_f}} \\
\I\ar[r]_g & E\times_B E\,.}$$
We display the following associated lifting problem,
$$\xymatrix@=1,5cm{\bold{1}\ar[r]\ar[d]_i & E\ar@{=}[r]\ar[d]^{\underline{\Delta_f}} & E\ar[d]^{\underline{f}} \\
\I\ar[r]_g & E\times_B E\ar[r]_{\underline{f}\circ p_1} & B\,,}$$
where $p_1$ is the first projection.\\
The maps $p_1\circ g$ and $p_2\circ g$ are two diagonal fillers for this last lifting problem, since $f$ is a small fibration then thanks to \ref{minimalfib} $\underline{f}$ is a discrete fibration, hence $p_1\circ g$ and $p_2\circ g$ are equal. So the initial lifting problem with respect to $\underline{\Delta_f}$ has a unique solution, namely $p_1\circ g$. Since $\underline{\Delta_f}$ is a discrete fibration its fibers are discrete and their sets of objects are obviously small. By \ref{minimalfib} one concludes that $\Delta_f$ is a pullback of $p_\Delta$, namely a small fibration.
\end{proof}

\begin{thm}\label{thm4}
The morphism $p_\Delta:\widetilde{U_\Delta}\rightarrow U_\Delta$ is a universe for the projective type-theoretic fibration category $\GGpd_{\mathbf{proj}}$.
\end{thm}
\begin{proof}\label{proofthm4}
Straighforward with \textcolor{darkgreen}{Lemmas} \ref{lem6}, \ref{lem7}, \ref{lem8}.
\end{proof}\bigskip

For the record we prove below the fact that one can factor any morphism $f$ in $\GGpd$ as a trivial cofibration followed by a fibration whose underlying fibration of groupoids is equipped with a split cleavage. Unfortunately a priori this last fibration is not a small fibration since there is no guarantee for the fibers to be small even if the morphism $f$ we start with has small fibers. Hence this result does not imply the condition $(iii)$ in \ref{def:univ} for our general candidate $p: \widetilde{U}\rightarrow U$ for universe.

\begin{prop}\label{propgiraudlike}
Let $f:E\rightarrow B$ be any morphism in $\GGpd$, one can factor $f$ as a trivial cofibration $i_{\infty}$ followed by a fibration $f_{\infty}$ whose underlying fibration of groupoids is equipped with a split cleavage
$$\xymatrix{&E_{\infty}\ar@{->>}[rdd]^{f_{\infty}} &\\&\circlearrowright&\\\overset{~}E\ar@{>->}[ruu]^*[@]{\hbox to -1pt{$\sim$}}^-{i_{\infty}~} \ar[rr]_{f}&&B~.}$$
\end{prop}
\begin{proof}
First, we show that one can reduce the proof to the case where $\underline{E}$ is a discrete groupoid. Indeed, assume we have an object $\Delta(E)_{\infty}$ in $\GGpd$ and two maps $i_{\infty}^\Delta$ , $f_{\infty}^\Delta$ such that the following square commutes,
$$\xymatrix{ & \Delta(E)_{\infty}\ar@{->>}[rdd]^{f_{\infty}^\Delta} & \\
& \circlearrowright & \\
\Delta(E)\ar@{>->}[ruu]^*[@]{\hbox to -4pt{$\sim$}}^-{i_{\infty}^\Delta}\ar[rr]_{f_{|\Delta(E)}} & & B\,.}$$
Consider the pushout of $i_{\infty}^\Delta$ along the inclusion from $\Delta(E)$ to $E$,
$$\xymatrix@=1,5cm{\Delta(E)\ar@{>->}[d]_*[@]{\hbox to -10pt{$\sim$}}_-{i_{\infty}^\Delta}\ar@{^{(}->}[r] & E\ar[d]^{i_{\infty}} \\
\Delta(E)_\infty\ar[r] & \pushoutcorner E_\infty\,.}$$
We denote $i_\infty$ the resulting map. As a pushout of a trivial cofibration $i_\infty$ is a trivial cofibration. By using the universal property of this pushout we get a map $f_\infty$ as follows,
$$\xymatrix@=1,5cm{\Delta(E)\ar@{>->}[d]_*[@]{\hbox to -10pt{$\sim$}}_-{i_{\infty}^\Delta}\ar@{^{(}->}[r] & E\ar[d]^{i_{\infty}}\ar@/^/[rdd]^f & \\
\Delta(E)_\infty\ar[r]\ar@/_/[rrd]_{f_\infty^\Delta} & \pushoutcorner E_\infty\ar@{-->}[rd]^{f_\infty}  & \\
 & & B\,.}$$
It remains to prove that $f_\infty$ is a levelwise fibration equipped with a split cleavage. Consider the following lifting problem ,
 $$\xymatrix@=1,5cm{\bold{1}\ar[d]_i\ar[r]^g & E_\infty \ar[d]^{\underline{f_\infty}} \\
 \I\ar[r]_h & B}$$ 
 where $\lbrace i\rbrace$ is a set of generating trivial cofibrations in $\Gpd$ (see \ref{eg3}). Since $\text{Ob}(E_\infty)$ is $\text{Ob}(\Delta(E)_\infty)$, $g$ factorizes through $\Delta(E)_\infty$. So we are able to associate with our initial lifting problem a new lifting problem with respect to $\underline{f_\infty^\Delta}$,
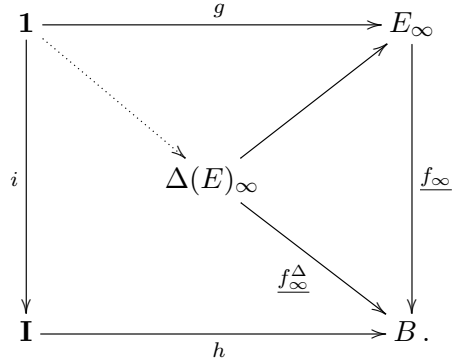
$$\xymatrix@=1,5cm{\bold{1}\ar[dd]_i\ar[rr]^g\ar@{.>}[rd] & & E_\infty \ar[dd]^{\underline{f_\infty}} \\
& \Delta(E)_\infty\ar[ru]\ar[rd]_{\underline{f_\infty^\Delta}} & \\
 \I\ar[rr]_h & & B\,.}$$
Now by postcomposition of the diagonal filler (the dashed arrow below) chosen by the split cleavage of $\underline{f_\infty^\Delta}$ for this last lifting problem we get a diagonal filler for our initial lifting problem,
$$
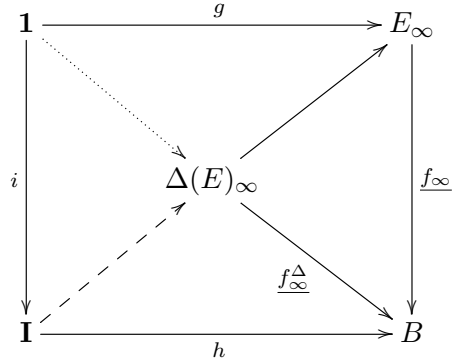
\xymatrix@=1,5cm{\bold{1}\ar[dd]_i\ar[rr]^g\ar@{.>}[rd] & & E_\infty \ar[dd]^{\underline{f_\infty}} \\
& \Delta(E)_\infty\ar[ru]\ar[rd]_{\underline{f_\infty^\Delta}} & \\
 \I\ar[rr]_h\ar@{-->}[ru] & & B}$$
and this process defines a split cleavage for $\underline{f_\infty}$.\bigskip

Second, we prove that one can reduce the case of a discrete groupoid $E$ to two more basic cases, namely the case of a discrete groupoid with only fixed points and the case of a discrete groupoid with no fixed point. Indeed, the discrete groupoid $E$ is (isomorphic to) $E^{\mathbb{Z}/2\mathbb{Z}}\coprod E^0$, where $E^{\mathbb{Z}/2\mathbb{Z}}$ is the discrete subgroupoid of the fixed points in $E$ and $E^0$ is the discrete subgroupoid of $E$ made of the points that are not fixed by the involution. Also the map $f$ is (isomorphic to) $f^{\mathbb{Z}/2\mathbb{Z}}\coprod f^0$. Let us assume that we have the following commutative diagrams,
$$\xymatrix{&E_\infty^{\mathbb{Z}/2\mathbb{Z}}\ar@{->>}[rdd]^{f_{\infty}^{\mathbb{Z}/2\mathbb{Z}}} &\\&\circlearrowright&\\\overset{~}E^{\mathbb{Z}/2\mathbb{Z}}\ar@{>->}[ruu]^*[@]{\hbox to -1pt{$\sim$}}^-{i_{\infty}^{\mathbb{Z}/2\mathbb{Z}}~} \ar[rr]_{f^{\mathbb{Z}/2\mathbb{Z}}}&&B\,,}$$
$$\xymatrix{&E_{\infty}^0\ar@{->>}[rdd]^{f_{\infty}^0} &\\&\circlearrowright&\\\overset{~}E^0\ar@{>->}[ruu]^*[@]{\hbox to -1pt{$\sim$}}^-{i_{\infty}^0~} \ar[rr]_{f^0}&&B\,.}$$
Hence one has the following commutative diagram,
$$\xymatrix{&E_\infty^{\mathbb{Z}/2\mathbb{Z}}\coprod E_\infty^0\ar@{->>}[rdd]^{f_{\infty}^{\mathbb{Z}/2\mathbb{Z}}\coprod f_\infty^0} &\\&\circlearrowright&\\\overset{~}E^{\mathbb{Z}/2\mathbb{Z}}\coprod E^0\ar@{>->}[ruu]^*[@]{\hbox to -1pt{$\sim$}}^-{i_{\infty}^{\mathbb{Z}/2\mathbb{Z}}\coprod i_\infty^0~} \ar[rr]_f&&B\,.}$$   
Clearly $i_{\infty}^{\mathbb{Z}/2\mathbb{Z}}\coprod i_\infty^0$ is a trivial cofibration as a coproduct of trivial cofibrations.  Finally the map $f_{\infty}^{\mathbb{Z}/2\mathbb{Z}}\coprod f_\infty^0$ is a levelwise fibration equipped with a split cleavage. Indeed, consider a lifting problem in $\Gpd$ for $\underline{f_{\infty}^{\mathbb{Z}/2\mathbb{Z}}}\coprod \underline{f_\infty^0}$,
$$\xymatrix@=2cm{\bold{1}\ar[d]_i\ar[r]^g & E_\infty^{\mathbb{Z}/2\mathbb{Z}}\coprod E_\infty^0\ar[d]^{\underline{f_{\infty}^{\mathbb{Z}/2\mathbb{Z}}}\coprod \underline{f_\infty^0}} \\
\I\ar[r]_h & B\,.}$$
There are two mutually exclusive possibilities. Either $g$ takes values in $E_\infty^{\mathbb{Z}/2\mathbb{Z}}$ or it takes values in $E_\infty^0$. If $g$ takes values in $E_\infty^{\mathbb{Z}/2\mathbb{Z}}$ one can display the following diagram,
$$
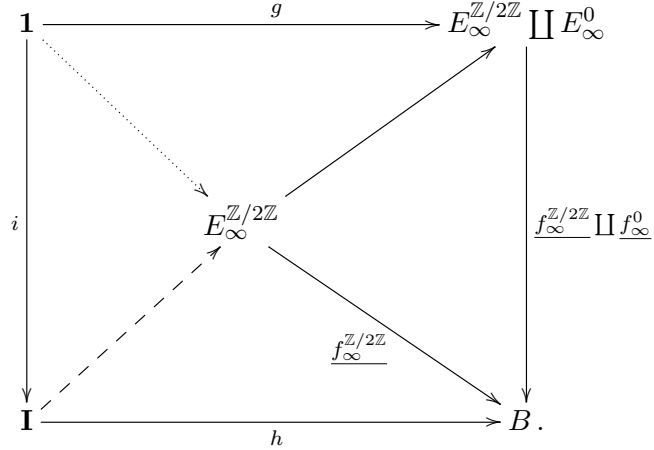
\xymatrix@=2cm{\bold{1}\ar[dd]_i\ar[rr]^g\ar@{.>}[rd] & & E_\infty^{\mathbb{Z}/2\mathbb{Z}}\coprod E_\infty^0\ar[dd]^{\underline{f_{\infty}^{\mathbb{Z}/2\mathbb{Z}}}\coprod \underline{f_\infty^0}} \\
 & E_\infty^{\mathbb{Z}/2\mathbb{Z}}\ar[ru]\ar[rd]_{\underline{f_{\infty}^{\mathbb{Z}/2\mathbb{Z}}}} & \\
\I\ar[rr]_h\ar@{-->}[ru] & & B\,.}$$
Otherwise one can display the following diagram,
$$
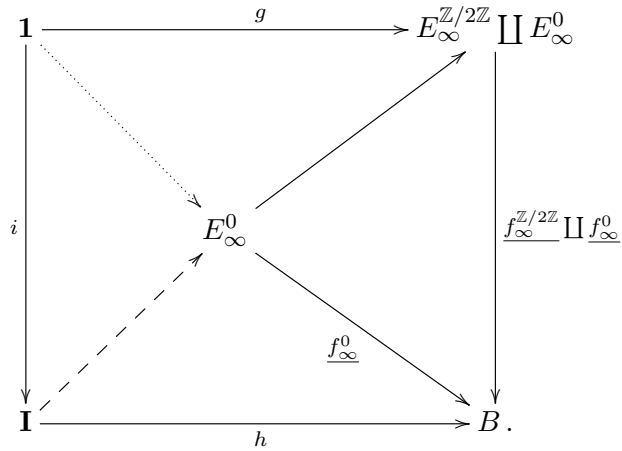
\xymatrix@=2cm{\bold{1}\ar[dd]_i\ar[rr]^g\ar@{.>}[rd] & & E_\infty^{\mathbb{Z}/2\mathbb{Z}}\coprod E_\infty^0\ar[dd]^{\underline{f_{\infty}^{\mathbb{Z}/2\mathbb{Z}}}\coprod \underline{f_\infty^0}} \\
 & E_\infty^0\ar[ru]\ar[rd]_{\underline{f_{\infty}^0}} & \\
\I\ar[rr]_h\ar@{-->}[ru] & & B\,.}$$
These diagrams not only prove that $f_{\infty}^{\mathbb{Z}/2\mathbb{Z}}\coprod f_\infty^0$ is a fibration  but they also prove how to equip its underlying fibration with a split cleavage assuming that the dashed arrow is the diagonal filler chosen by the split cleavage of $\underline{f_{\infty}^{\mathbb{Z}/2\mathbb{Z}}}$ (resp. of $\underline{f_\infty^0}$).\bigskip

So we need to treat these two more basic cases. \\
First, let us assume that $E$ is a discrete groupoid with no fixed point. One defines the groupoid $E_\infty$ as $\displaystyle{\coprod_{x\in \text{Ob}(E)}f(x)/B}$  (where $f(x)/B$ denotes the under groupoid) with the involution that maps the component $f(x)/B$ to $f(\epsilon(x))/B$ thanks to the involution $\be$ on $B$ (where $\epsilon$ denotes the involution on $E$). The map $i_\infty$ maps $x\in \text{Ob}(E)$ to $id_{f(x)}$ in the corresponding copy of $f(x)/B$ and $f_\infty$ is the codomain functor. Using the characterization of trivial cofibrations given in \ref{prop3} the reader can check that $i_\infty$ is a trivial cofibration. Last, we need to prove that $f_\infty$ is a levelwise fibration equipped with a split cleavage denoted $c_\infty$. Let 
$x\in \text{Ob}(E)$ and $\varphi:f(x)\rightarrow y$ in $\text{Ob}(f(x)/B)$ and $h:y\rightarrow y'$ a morphism in $B$ (one has $f_\infty(\varphi) = y$), take $c_\infty(\varphi,h)$ to be the map $h$ from $\varphi$ to $h\circ\varphi$. Clearly $c_\infty$ defined this way is a split cleavage.\medskip

Second, let us assume that $E$ is a discrete groupoid with only fixed points. We start by proving that one can reduce this case to the case of $\bold{1}!$ (the terminal object). Indeed, $E$ is (isomorphic to) $\displaystyle{\coprod_{x\in\text{Ob}(E)}\bold{1}!}$. Also note that $f:E\rightarrow B$ is (isomorphic to) $\displaystyle{\coprod_{x\in\text{Ob}(E)}f_x:(\coprod_{x\in\text{Ob}(E)}\bold{1}!)\rightarrow B}$ and let us assume that for each $x\in \text{Ob}(E)$ we have morphisms $i_\infty^x$,$f_\infty^x$ such that
$$
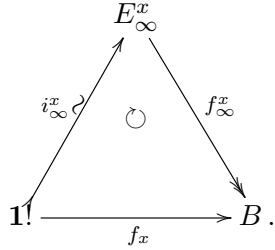
\xymatrix{ & E_{\infty}^x\ar@{->>}[rdd]^{f_{\infty}^x} & \\
& \circlearrowright & \\
\bold{1}!\ar@{>->}[ruu]^*[@]{\hbox to -4pt{$\sim$}}^-{i_{\infty}^x}\ar[rr]_{f_x} & & B\,.}$$
In this case we have the following factorization of $f$,
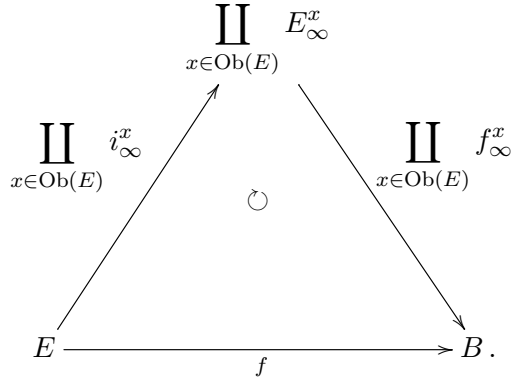
$$\xymatrix@=1,5cm{ & \displaystyle{\coprod_{x\in\text{Ob}(E)}E_{\infty}^x}\ar[rdd]^{\displaystyle{\coprod_{x\in\text{Ob}(E)} f_\infty^x}} & \\
& \circlearrowright & \\
E\ar[ruu]^{\displaystyle{\coprod_{x\in\text{Ob}(E)}i_{\infty}^x}}\ar[rr]_{f} & & B\,.}$$
Moreover $\displaystyle{\coprod_{x\in\text{Ob}(E)} i_\infty^x}$ is a trivial cofibration as a coproduct of trivial cofibrations. It remains to prove that $\displaystyle{\coprod_{x\in\text{Ob}(E)}f_\infty^x}$ is a levelwise fibration equipped with a split cleavage. We proceed as usual. Consider the following lifting problem in $\Gpd$,
$$\xymatrix@=2cm{\bold{1}\ar[r]^g\ar[d]_i & \displaystyle{\coprod_{x\in\text{Ob}(E)}E_\infty^x}\ar[d]^{\displaystyle{\coprod_{x\in\text{Ob}E}\underline{f_\infty^x}}}\\
\I\ar[r]_h & B\,.}$$
We can display the following square with the dashed arrow being the chosen one by the split cleavage of $\underline{f_\infty^x}$,
$$\xymatrix@=2cm{\bold{1}\ar@{.>}[rd]\ar[rr]^g\ar[dd]_i & & \displaystyle{\coprod_{x\in\text{Ob}(E)}E_\infty^x}\ar[dd]^{\displaystyle{\coprod_{x\in\text{Ob}E}\underline{f_\infty^x}}}\\
 & E_\infty^x\ar[ru]\ar[rd]_{\underline{f_\infty^x}} & \\
\I\ar@{-->}[ru]\ar[rr]_h & & B\,.}$$
So $\displaystyle{\coprod_{x\in\text{Ob}(E)}f_\infty^x}$ is a levelwise fibration  equipped with a split cleavage.\bigskip

Thus the ultimate case consists in the case where $E$ is $\bold{1}!$. We define a groupoid $A$ whose set of objects consists in the finite (possibly empty) sequences of composable arrows in $B$ starting at $f(\star)$ (where $\star$ denotes the unique object of $\bold{1}!$) quotiented by the relation that identifies such a sequence with the sequence you get when you erase the identity arrows in it. Moreover there is exactly one arrow for any ordered pair of objects in $A$. For a pair $((g_1,\hdots,g_n), (h_1,\hdots,h_m))$ with $n,m$ elements of $\mathbb{N}$, let us denote this unique map by $u_{g_{i,n}}^{h_{i,m}}$. Let us denote by $\varnothing$ the empty sequence of arrows in $B$. We equip $A$ with the involution induced by the involution $\be$ on $B$. We define $E_\infty$ thanks to the following pushout in $\GGpd$,
$$\xymatrix@=1,5cm{S(\b1)\ar[r]^l\ar[d]_k & A\times S(\b1) \ar[d] \\
A\times S(\b1)\ar[r] & E_\infty \pushoutcorner\,,}$$
where 
\begin{align*}
k(0) &= (\varnothing, 0) \\
k(1) &=  (\varnothing,1)\\
l(0) &= (\varnothing,1)\\
l(1) &= (\varnothing,0)\,.
\end{align*}
We define the functor $i_\infty$ as the functor that maps $\star$ to $(\varnothing,0)$ and the reader can easily check thanks to \ref{prop3} that $i_\infty$ is a trivial cofibration. We define $f_\infty$ thanks to the universal property of the pushout,
$$\xymatrix@=1,5cm{S(\b1)\ar[r]^l\ar[d]_k & A\times S(\b1) \ar[d]\ar@/^/[rdd]^\rho & \\
A\times S(\b1)\ar[r]\ar@/_/[rrd]_\rho & E_\infty \pushoutcorner\ar@{-->}[rd]_{f_\infty} & \\
 & & B\,,}$$
where one defines $\rho$ as follows,
\begin{align*}
\rho(\varnothing,j) &= f(\star) \quad\text{with}\quad j=0,1 \\
\rho(f(\star)\xrightarrow{g_1}y_1\rightarrow\hdots\xrightarrow{g_n}y_n,j) &= y_n\quad\text{for}\quad n\geqslant 1\quad\text{and}\quad j=0,1 \\
\rho(id_\varnothing,id_j) &= id_{f(\star)}\quad\text{with}\quad j=0,1\\
\rho(u_{g_{i,n}}^{h_{i,m}},id_j) &= (h_m\circ\hdots\circ h_1)\circ (g_n\circ\hdots\circ g_1)^{-1}\quad\text{with}\quad j=0,1\\
\rho(u_{\varnothing}^{g_{i,n}},id_j) &= g_n\circ\hdots\circ g_1\quad\text{with}\quad j=0,1\,.
\end{align*}
It remains to prove that $f_\infty$ is a levelwise fibration equipped with a split cleavage. We define a split cleavage $c_\infty$ for $\underline{f_\infty}$ as follows. Consider the following lifting problem in $\Gpd$,
$$\xymatrix@=1,5cm{\bold{1}\ar[r]^g\ar[d]_i & E_\infty\ar[d]^{\underline{f_\infty}}\\
\I\ar[r]_h & B\,.}$$
We are looking for a diagonal filler $c_\infty(g,h)$ in such a way that $c_\infty$ will be a split cleavage. Recall that $\phi$ denotes the unique non trivial isomorphism in $\I$ from $0$ to $1$. If $h(\phi)$ is an identity then take for $c_\infty(g,h)$ the identity arrow of $g(\star)$. Otherwise, if $g(\star)$ is $((g_1,\hdots,g_n),j)$ with $n$ greater than or equal to 0 and $j$ equal to 0 or 1 then our diagonal filler depends on the parity of $n$. If $n$ is even then define $c_\infty(g,h)$ as the unique map in $E_\infty$ from $((g_1,\hdots, g_n),j)$ to $((g_1,\hdots,g_n,h(\phi)),j)$, otherwise ($n$ is odd) define $c_\infty(g,h)$ as the unique map in $E_\infty$ from $((g_1,\hdots,g_n),j)$ to $((g_1,\hdots,g_{n-1},h(\phi)\circ g_n),j)$. The reader can easily check that $c_\infty$ is a split cleavage.
\end{proof}

\section{Homotopy equivalences in $\GGpd$}
\label{sec:heGGpd}

Now, we develop a few basic facts about homotopy equivalences, then we give an explicit characterization of homotopy equivalences in $\GGpd$.\bigskip

\begin{prop}\label{prop5}
If $\mathscr{C}$ is a model category and $\xymatrix{f:A~~\ar@{ >->}[r]^-{\sim}&B}$ is a trivial cofibration such that $A$ is fibrant then $f$ is a (right) homotopy equivalence.\\
\end{prop}
\begin{proof}\label{proofprop5}
One has the following lifting problem : 
$$\xymatrix{\underset{~}{A}\ar@{=}[rr]\ar@{ >->}[dd]_*[@]{\hbox{$\sim$}}_{f\quad}\commutatif&&A\ar@{->>}[dd]\\\\B\ar[rr]&&\bold{1}}$$
Since $f$ is a trivial cofibration and $A$ is fibrant there exists a diagonal filler $g:B\rightarrow A$ such that :
$$
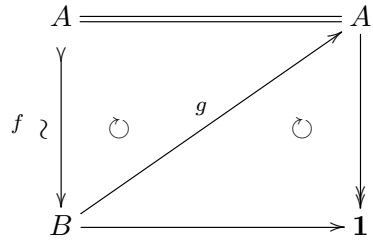
\xymatrix{\underset{~}{A}\ar@{=}[rr]\ar@{ >->}[dd]_*[@]{\hbox{$\sim$}}_{f\quad}&&A\ar@{->>}[dd]\\\qquad\qquad\circlearrowright&&\circlearrowright\qquad\qquad\\B\ar[rruu]^g\ar[rr]&&\bold{1}}$$ 
with $g\circ f=1_{A}$, so one immediately concludes that $g\circ f\overset{r}{\sim}1_{A}$. One defines the following lifting problem :
$$
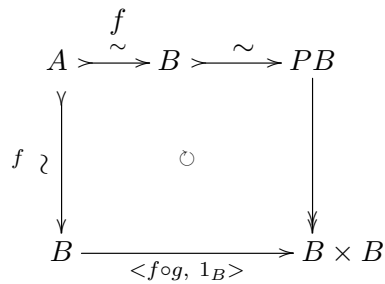
\xymatrix{\underset{~}{A}~\ar@{ >->}[r]^[@]{\hbox{$\underset{\sim}{f}$}}\ar@{ >->}[dd]_[@]{\hbox{$\sim$}}_{f\quad}\commutatif&B~\ar@{ >->}[r]^[@]{\hbox{$\sim$}}&PB\ar@{->>}[dd]\\\\B\ar[rr]_{<f\circ g,~1_{B}>}&&**[r]{B\times B}\\}$$ 
where $PB$ is any path object for $B$. Note that the square above commutes since  
\begin{align*} 
<f\circ g,~1_{B}>\circ f&=~<f\circ g\circ f,f>\\&=~<f,f>
\end{align*} 
and $\Delta\circ f=~<1_{B},~1_{B}>\circ f=~<f,f>$ where $\Delta$ denotes the diagonal map $<1_{B},1_{B}>$ which is the composition
$$\xymatrix{B~\ar@{ >->}[r]^-{\sim}&PB\ar@{->>}[r]&B\times B}.$$ 
So there exists a diagonal filler $h$,
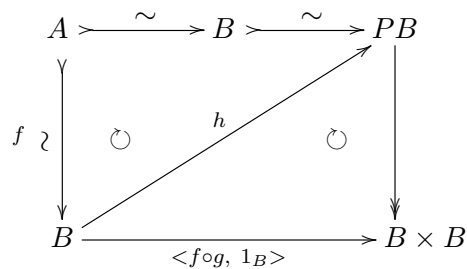
$$\xymatrix{\underset{~}{A~}\ar@{ >->}[r]^[@]{\hbox{$\sim$}}\ar@{ >->}[dd]_[@]{\hbox{$\sim$}}_{f\quad}&B~\ar@{ >->}[r]^[@]{\hbox{$\sim$}}&PB\ar@{->>}[dd]\\\qquad\qquad\circlearrowright&&\circlearrowright\qquad\qquad\\B\ar[rruu]^h\ar[rr]_{<f\circ g,~1_{B}>}&&**[r]{B\times B}\\\\}$$ and such a $h$ is exactly a (right) homotopy between $f\circ g$ and $1_{B}$, so $f\circ g\overset{r}{\sim}1_{B}$ and as a consequence $f$ is a (right) homotopy equivalence.
\end{proof}\bigskip
\begin{prop}\label{prop6}
If $f:A\rightarrow B$ is a trivial cofibration in $\GGpd$ then $f$ is a (right) homotopy equivalence.\\
\end{prop}
\begin{proof}\label{proofprop6}
Straightforward with the previous proposition and the fact that all objects of $\GGpd$ are fibrant.
\end{proof}\bigskip
Since not all objects are cofibrant here, homotopy equivalences are not the same as the weak equivalences of the model structure.\\

\begin{defi}
Let $G$ be a groupoid equipped with an involution $\eta$. On says that $G$ is weakly connected if and only if for every pair $(x,y)$ in $\text{Ob}(G)^2$ either $x$ and $y$ are in the same connected component of the groupoid $\underline{G}$ or $x$ and $\eta(y)$ are in the same connected component of the groupoid $\underline{G}$.
\end{defi}\bigskip

\begin{prop}\label{coprodofweaklyconnected}
Every groupoid equipped with an involution is (isomorphic to) a coproduct in $\GGpd$ of weakly connected groupoids with involutions.
\end{prop}
\begin{proof}
Let $G$ be a groupoid equipped with an involution $\eta$. Being given $x$ in $\text{Ob}(G)$, let us denote by $C_x$ the connected component of $x$ in the groupoid $\underline{G}$. Now we denote by $C_x^{\G}$ the full subgroupoid of $G$ whose set of objects is $\text{Ob}(C_x)\bigcup\text{Ob}(C_{\eta(x)})$. This full subgroupoid, which we call the weak connected component of $x$, has a natural involution, induced by the involution $\eta$ on $G$, hence it turns out to be a groupoid with an involution. By choosing a representative for each set $\text{Ob}(C_x)\bigcup\text{Ob}(C_{\eta(x)})$ one can display $G$ as a coproduct over these representatives of the $C_x^{\G}$'s in $\GGpd$. 
\end{proof}\bigskip

\begin{lem}\label{lem2}
Let $G$ be a groupoid equipped with an involution $\eta$ and $$\varnothing\subset G'\subseteq G''$$ two full subgroupoids of $G$ stable under $\eta$ such that $\text{Ob}(G'') = \text{G'}\cup \lbrace x,\eta(x)\rbrace$ with $x\in G$ and $\eta(x)\neq x$. Moreover assume that $G$ is weakly connected. Since $G'\neq \varnothing$, let $z$ be an element of $G'$. By weak connectedness there exist an element $y$ of $\lbrace z,\eta(z)\rbrace$ and an isomorphism $c$ from $y$ to $x$. Then the following square is a pushout square in $\GGpd$,
$$\xymatrix{S(\b1)\ar[rr]^l\ar@{^{(}->}[dd]_{S(i)}&&G'\ar@{^{(}->}[dd]\\\\ S(\I)\ar[rr]_k&& G''}$$
with $l(0) = y$, $l(1) = \eta(y)$, $k(0) = y$, $k(1) = \eta(y)$, $k(0') = x$, $k(1') = \eta(x)$, $k(\phi) = c$ and $k(\psi) = \eta(c)$ where 
\begin{align*}
S(\b1):= \b1\textstyle\coprod\b1 = \xymatrix{0 & 1}
\end{align*}
and
\begin{align*}
S(\I):= \I\textstyle\coprod\I = \xymatrix{0\ar[d]_\phi & 1\ar[d]^\psi \\ 0' & 1'} 
\end{align*}
with the swap involutions.
\end{lem}
\begin{proof}
We will check that our square satisfies the universal property of a pushout. Hence consider the following commutative square,
$$\xymatrix{S(\b1)\ar[rr]^l\ar@{^{(}->}[dd]_{S(i)}&&G'\ar^{m}[dd]\\\\S(\I)\ar[rr]_n&& H}$$ 
We want to prove there exists a unique map $j:G''\rightarrow H$ such that $j_{|G'} = m$ and $j\circ k = n$. We define $j$ as follows. Take $j_{|G'} = m$, $j(x) = n(0')$, $j(\eta(x)) = n(1')$ and $j(c) = n(\phi)$, $j(\eta(c)) = n(\psi)$. It remains to define consecutively $j$ on morphisms from any $z\in G'$ to $x$, on morphisms from any $z\in G'$ to $\eta(x)$, on the automorphisms of $x$ and $\eta(x)$ and on morphisms from $x$ to $\eta(x)$.\\
Let $f$ be a morphism from $z\in G'$ to $x$. Note that $G'$ being a full subgroupoid and $z$ and $y$ being elements of $G'$, the morphism $c^{-1}\circ f$ belongs to $G'$. Hence take $j(f) = j(c)\circ j(c^{-1}\circ f)$. Now let $f$ be a morphism from any $z\in G'$ to $\eta(x)$, to make sure that $j$ is compatible with the involutions take $j(f) = \theta(j(\eta(f)))$. Next, let $f$ be an automorphism of $x$, take $j(f) = j(c)\circ j(c^{-1}\circ f)$. Again for the sake of the compatibility with the involutions take $j(f) = \theta(j(\eta(f)))$ for any automorphism $f$ of $\eta(x)$. Last, for any morphism $f$ from $x$ to $\eta(x)$, take $j(f) = j(f\circ c)\circ j(c)^{-1}$. The reader can easily check that $j$ is unique.   
\end{proof}\medskip

One gives the following full characterization of the homotopy equivalences in $\GGpd$.
\begin{thm}\label{thm5}
Let $f:G\rightarrow H$ be a morphism in $\GGpd$. The following are equivalent :\\
\begin{enumerate}[label=(\roman*)]
\item $f$ is a (right) homotopy equivalence.
\item $f$ is a levelwise weak equivalence and $f$ induces an isomorphism between the full subgroupoids of fixed points $G\f$ and $H\f$.
\item $f$ is a levelwise weak equivalence and $f$ induces an isomorphism between the subgroupoids of fixed points and fixed morphisms $G^{\G}$ and $H^{\G}$.
\item $f$ is a levelwise weak equivalence and $f$ induces a bijection between the set of fixed points in $G$ and the set of fixed points in $H$.
\end{enumerate}
\end{thm}
\begin{proof}\label{proofthm5}
We will prove successively $(i)\Rightarrow (iv)$, $(iv)\Rightarrow (iii)$, $(iii)\Rightarrow (ii)$ and $(ii)\Rightarrow (i)$.\\
Before turning to $(i)\Rightarrow (iv)$ we prove the following two lemmas.
\begin{itemize}

\item Lemma 1. Let $X$ be a groupoid equipped with an involution such that $X^{\G} = X$ and $w:Y\acof Z$ a projective trivial cofibration in $\GGpd$. Then for any morphism $v: X\rightarrow Z$ there exists a map $\widehat{v}:X\rightarrow Y$ that makes the following square commutes,
$$\xymatrix{X~\ar[rr]^{\widehat{v}} \ar[rrdd]_{v}&&Y\ar^{w}[dd]\\&&\circlearrowright\qquad\qquad\\&&Z\,.}$$
Indeed, define $\widehat{v}$ as follows. Let $x$ be an element of $\text{Ob}(X)$, since $v(x)\in Z\f$ by the characterization of trivial cofibrations given in \ref{prop3} there exists a unique $y\in Y\f$ such that $v(x) = w(y)$. Take $\widehat{v}(x) = y$. Now let $f:x\rightarrow x'$ be a morphism in $X$, since $w$ is fully faithful the induced map by $w$ from $Y(y,y')$ to $Z(v(x),v(x'))$ is a bijection, hence there exists a unique map $\widehat{v}(f)$ such that $w(\widehat{v}(f)) = v(f)$.
\item Lemma 2. Let $f,g: X\rightarrow H$ be two (right) homotopic maps in $\GGpd$ such that $X^{\G} = X$. Then one has $f = g$.\\
Indeed, let $PH$ be a path object for H,
$$\xymatrix{H~\ar@{ >->}[rr]^{\sim}_{w} \ar[rrdd]_{\Delta}&&PH\ar@{->>}[dd]\\&&\circlearrowright\qquad\qquad\\&&H\times H}$$
and $\gamma$ a (right) homotopy,
$$\xymatrix{&&**[r]{PH}\ar@{->>}[dd]\\&\pcom&\\X\ar[rruu]^{\gamma}\ar[rr]_{<f,g>}&&**[r]{H\times H\,.}\\}$$
By the previous lemma 1 applied with $Y = H$, $Z = PH$ and $v = \gamma$, there exists $\widehat{\gamma}$ such that $w\circ \widehat{\gamma} = \gamma$. So we have $\Delta\circ \widehat{\gamma} = <f,g>$, hence $\widehat{\gamma} = f = g$.
\end{itemize}
We prove $(i)\Rightarrow (iv)$. Assume $(i)$, so there exists $g:H\rightarrow G$ in $\GGpd$ such that $f\circ g\overset{r}{\sim}1_{H}$ and $g\circ f\overset{r}{\sim}1_{G}$. Hence $(f\circ g)_{|H^{\G}}\overset{r}{\sim}1_{H^{\G}}$ and $(g\circ f)_{|G^{\G}}\overset{r}{\sim}1_{G^{\G}}$, so by the lemma 2 above one has $(f\circ g)_{|H^{\G}} = 1_{H^{\G}}$ and $(g\circ f)_{|G^{\G}} = 1_{G^{\G}}$ and we conclude that $f_{G^{\G}}: G^{\G}\rightarrow H^{\G}$ is an isomorphism. Hence $f$ induces a bijection between the set of fixed points in $G$ and the set of fixed points in $H$. \medskip

We prove $(iv)\Rightarrow (iii)$. It is straighforward using the fact that $f$ is a levelwise weak equivalence, in particular a fully faithful functor. \medskip

Next we prove $(iii)\Rightarrow (ii)$. Note that $f_{|G\f}$ is bijective on objects since $f_{|G^{\G}}$ is by assumption. Moreover $f$ is fully faithful hence $f_{|G\f}$ is an isomorphism. \medskip 

Last we prove $(ii)\Rightarrow (i)$. Note that by \ref{coprodofweaklyconnected} we can assume without loss of generality that G is weakly connected. Let $f:G\rightarrow H$ in $\GGpd$ be a weak equivalence (it means that $f$ is a levelwise weak equivalence of groupoids) such that $f$ induces an isomorphism between $G\f$ and $H\f$.\\ We start by proving one can assume that $f$ is surjective. Indeed, first note that one can factorize $f$ through its image $\text{Im}~f$, the full subgroupoid of $H$ whose objects are of the form $f(x)$ for some $x\in G$. The groupoid $\text{Im}~f$ can be equipped with an involution thanks to (the restriction of) $\theta$ the involution on $H$. Indeed, being given $y\in H$ such that there exists $x\in G,~f(x)=y$ then $f(\eta(x))=\theta(f(x))=\theta(y)$, hence $\theta(y)\in\text{Im}~f$. First we prove that the inclusion morphism $\text{Im}~f\overset{i}{\hookrightarrow}H$ is a projective trivial cofibration. Indeed since $(\text{Im}~f)\f = H\f$ and $\text{Im}~f$ is equivalent to $H$, we conclude by \ref{cor1}. So thanks to \ref{prop6} this inclusion is a homotopy equivalence.  One concludes that $f:G\rightarrow H$ is a homotopy equivalence if and only if $G\rightarrow\text{Im}~f$ is (since the homotopy equivalences have the \enquote {$2$ out of $3$} property). The morphism $G\rightarrow\text{Im}~f$ is still a levelwise weak equivalence (since weak equivalences have the \enquote {$2$ out of $3$} property) and this morphism still induces an isomorphism between $G\f$ and $(\text{Im}~f)\f$, since $(\text{Im}~f)\f=H\f$. So without loss of generality one can assume that our map $f:G\rightarrow H$ in $\GGpd$, which is a levelwise weak equivalence and induces an isomorphism between $G\f$ and $H\f$, is also surjective on objects (hence on morphisms).\medskip

One wants to prove that $f$ is a homotopy equivalence. We will provide a morphism $g:H\rightarrow G$ in $\GGpd$ such that $f\circ g=1_H$ and $g\circ f\overset{r}{\sim}1_G$. Also note we will use that our projective model structure comes with functorial factorizations of any map as a trivial cofibration followed by a fibration since the small object argument applies in our model category (see for details \url{ncatlab.org/nlab/show/path+space+object}). In the next we use the letter $P$ to refer to our functorial path objects. To achieve our goal we rely on Zorn's lemma, namely we construct a preordered set of partial homotopy equivalences then we apply Zorn's lemma to get a maximal element and last we prove that this maximal element is the required (total) homotopy equivalence. One defines a set $S$ as the set of triples 
$$(G'\subseteq G,~g': f(G')\rightarrow G',~h': G'\rightarrow PG')$$ 
such that $G'$ is a full subgroupoid of $G$ with $\eta'=\eta_{|G'}$ and $G_{\text{f}}\subseteq G'$ and the following squares commute

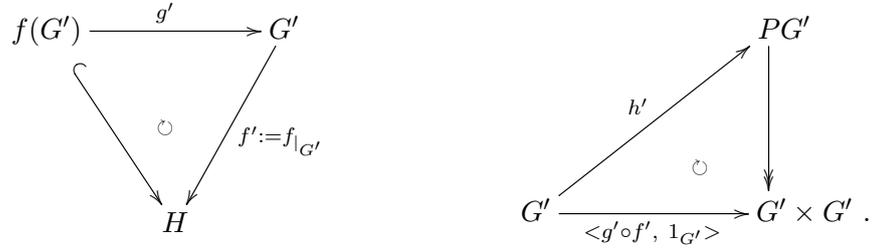
\begin{tabular}{lll} 
{$\qquad\qquad\xymatrix{\underset{~}{f(G')}\commutatif\ar[rr]^{g'}\ar@{^{(}->}[rdd]&&G'\ar[ldd]^{f':=f_{|_{G'}}}\\&&\\&H&\\\\}$}&{$\qquad\qquad$}&{$\xymatrix{&&**[r]{PG'}\ar@{->>}[dd]\\&\pcom&\\G'\ar[rruu]^{h'}\ar[rr]_{<g'\circ f',~1_{G'}>}&&**[r]{G'\times G'}~.}$} 
\end{tabular}\\
Note that by $f(G')$ we denote the full subgroupoid of $H$ whose objects are the $f(x)$'s with $x\in G'$. Note this last groupoid is equipped with an involution, namely $\theta_{|f(G')}$, since $\theta(f(x)) = f(\eta(x)) = f(\eta'(x))$ with $\eta'(x)\in G'$ whenever $x\in G'$ (hence $\theta(f(x))\in f(G')$). 
The two commutative diagrams above mean respectively that $g'$ is a partial section of $f$ and $h'$ is a (right) homotopy between $g'\circ f'$ and $1_{G'}$. One can provide $S$ with the structure of a preordered set : 
$$(G',g',h')\leqslant(G'',g'',h'')\text{ if and only if }(G'\subseteq G'',~g''_{|f(G')} = g',
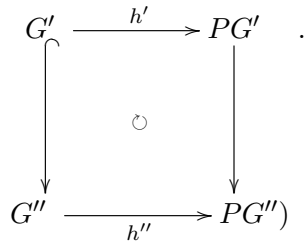
$$ 
$$\xymatrix@C=1,5cm{G'~\comr\ar[r]^{h'}\ar@{^{(}->}[dd]&PG'\ar[dd]\\\\**[l]{G''}~\ar[r]_{h''}&**[r]{PG'')}}.$$.\\
\begin{itemize}
\item Reflexivity :\\\\
One has $(G',g',h')\leqslant(G',g',h')$ since $G'\subseteq G',~g'_{|f(G')} = g',$ 
$$\xymatrix@C=1,5cm{G'\comr\ar[r]^{h'}\ar@{=}[dd]&PG'\ar@{=}[dd]\\\\**[l]{G'}\ar[r]^{h'}&**[r]{PG'.}}$$.\\
\item Transitivity :\\\\
Assume $(G',g',h')\leqslant(G'',g'',h'')$ and $(G'',g'',h'')\leqslant(G''',g''',h''')$
then $G'\subseteq G''\subseteq G'''$ hence $G'\subseteq G'''$,
and $g'''_{|f(G'')}=g'',~g''_{|f(G')}=g'$ hence $g'''_{|f(G')}=g'$,
last the commutativity of the diagram with $h', h'''$ is obvious by functoriality of $P$. 
So one concludes $(G',g',h')\leqslant (G''',g''',h''')$.
\end{itemize}
Let $C\subseteq S$ be a chain of $S\textit{ i.e. }C$ is a subset of $S$ which is totally ordered by $\leqslant$. One has to distinguish two cases depending on the emptiness of $C$.
\begin{itemize}
\item First, assume that $C=\varnothing$. Takes $G' = G_f$. In this case $f(G') = f(G_{f}) = H_{f}$ since $f$ induces an isomorphism between $G_f$ and $H_f$. Takes $f_{|G_f}^{-1}$ for $g'$ then $g'\circ f' = 1_{G'}$ and $<g'\circ f', 1_{G'}> = \Delta$, hence takes for $h'$ the trivial cofibration that comes with the path object $PG_f$.\\
\item Second, assume $C\ne\varnothing$. One takes $G'_{C}:=\underset{(G',g',h')\in C}{\text{colim}} G'$, more specifically $G'_C$ is the full subgroupoid of $G$ whose set of objects is given by $$\text{Ob}(G'_C) = \underset{(G',g',h')\in C}{\bigcup}\text{Ob}(G')\,.$$ This is easy to check that $G'_C$ contains $G_f$ (since $C\neq \varnothing$) and is equipped with the restriction of $\eta$ as an involution.\\
In this case $f(G'_C)$ is the full subgroupoid of $H$ whose set of objects is $$\text{Ob}(f(G'_C))~=~\underset{(G',g',h')\in C}{\bigcup} f(G')\,.$$ \\
We are looking for $g'_C: f(G'_C)\rightarrow G'_C$. For each $(G',g',h')\in C$ one has $g'_{(G',g',h')}$ from $f(G')$ to $G'$. Recall that a functor $F:\mathscr{C}\rightarrow\mathscr{D}$ from a category $\mathscr{C}$ to a category $\mathscr{D}$ consists in the following functions: 

\begin{center}
$F_0: \text{Ob}(\mathscr{C})\rightarrow\text{Ob}(\mathscr{D})$ \\
$\forall(x,y)\in \text{Ob}(\mathscr{C})^2, F_{(x,y)}: \mathscr{C}(x,y)\rightarrow\mathscr{D}(F_0(x),F_0(y))$
\end{center}

such that \begin{center}
$\forall x\in \text{Ob}(\mathscr{C}), F_{(x,x)}(1_x) = 1_{F_0(x)}$\\
$\forall (x,y,z)\in \text{Ob}(\mathscr{C})^3,\forall(f,g), F_{(x,z)}(g\circ f) = F_{(y,z)}(g)\circ F_{(x,y)}(f)$
\end{center} 
where in the last equation $f$ goes from $x$ to $y$ and $g$ from $y$ to $z$.\\
In other words a functor is nothing than a bunch of functions satisfying some conditions, hence the union below has to be understood as the functor whose underlying functions are obtained by the union of the graphs of the functions involved. With this in mind takes $g'_C = \underset{(G',g',h')\in C}{\bigcup}g'$ which makes sense since $C$ is totally ordered and for $(G',g',h')\leqslant (G'',g'',h'')$ two elements of $\mathscr{C}$, one has $g''_{|f(G')} = g'$.\\  One has the following commutative square, 
$$\xymatrix{\underset{~}{f(G'_C)}\ar[rr]^{g'_{C}}\ar@{^{(}->}[rdd]&&G'_C\ar[ldd]^{f}\\&\circlearrowright&\\&H \, .&}$$ 
Now we are looking for a homotopy $h'_C$  such that the following square commutes, 
$$\xymatrix@C=1,5cm{&&PG'_{C}\ar@{->>}[dd]\\&&\circlearrowright\qquad\qquad\\G'_{C}\ar[rruu]^{h'_{C}}\ar[rr]_{<g'_{C}\circ f_{|G'_C},~1_{G'_{C}}>}&&**[r]{G'_{C}\times G'_{C}\, .}\\}$$\\
For each $(G',g',h')\in C$ one has an inclusion $G'\xhookrightarrow{} G'_C$, hence by functoriality one has a map $PG'\rightarrow PG'_C$ and by precomposition of this last morphism with $h':G'\rightarrow PG'$ we get a map from $G'$ to $PG'_C$. By taking the colimit over these morphisms one gets a map $h'_C: G'_C\rightarrow PG'_C$ satisfying the following commutative square,
$$\xymatrix@C=1,5cm{&&PG'_{C}\ar@{->>}[dd]\\&&\circlearrowright\qquad\qquad\\G'_{C}\ar[rruu]^{h'_{C}}\ar[rr]_{<g'_{C}\circ f_{|G'_C},~1_{G'_{C}}>}&&**[r]{G'_{C}\times G'_{C}\, .}\\}$$
(the commutativity of this square results from the fact that for each $(G',g',h')\in C$ one has the commutative square 
$$\xymatrix@C=1,5cm{&&PG'\ar@{->>}[dd]\\&&\circlearrowright\qquad\qquad\\G'\ar[rruu]^{h'}\ar[rr]_{<g'\circ f_{|G'},~1_{G'}>}&&**[r]{G'\times G'\, .}\\}$$
and $<g'_C\circ f_{|G'_C}, 1_{G'_C}>_{|G'} = <g'\circ f_{|G'}, 1_{G'}>$).\\ Moreover $(G',g',h')\leqslant (G'_C,g'_C,h'_C)$ for each $(G',g',h')\in C$ since the diagram
$$\xymatrix@C=1,5cm{G'~\comr\ar[r]^{h'}\ar@{^{(}->}[dd]&PG'\ar[dd]\\\\**[l]{G'_C}~\ar[r]_{h'_C}&**[r]{PG'_C}}.$$
commutes by construction of $h'_C$.
\end{itemize}
By Zorn's lemma $S$ has a maximal element $(G_{\text{max}},~g_{\text{max}},~h_{\text{max}})$. We want to prove that  $G_{\text{max}} = G$ (it will follow that $f$ is a homotopy equivalence).\medskip

Assume that $G_{\text{max}}\ne G$. We can distinguish two cases:
\begin{itemize}
\item First, let assume that $G_{\text{max}} = \varnothing$.\medskip

Since $G_{\text{max}}\ne G$ then there exists $x\in G$ such that $x\not\in G_{\text{max}}$ and $x$ is not a fixed point since $G_{\text{f}}\subseteq G_{\text{max}}$.\\ 
Now consider $\widetilde{G}$ the full subgroupoid of $G$ whose set of objects is \\
$\lbrace{x,\eta(x)}\rbrace$. This groupoid has a natural involution namely the restriction of $\eta$. Since $G_{\text{f}}\subseteq G_{\text{max}}$ and $f_{|G_f}: G_f\rightarrow H_f$ is an isomorphism, we conclude that neither $G$ nor $H$ has a fixed point.
We are looking for $\widetilde{g}:f(\widetilde{G})\rightarrow \widetilde{G}$ such that, 
$$\xymatrix{\underset{~}{f(\widetilde{G})}\ar[rr]^{\widetilde{g}}\ar@{^{(}->}[rdd]&&\widetilde{G}\ar[ldd]^{f_{|\widetilde{G}}}\\&\circlearrowright&\\&H& .}$$ 
Note that $f(\widetilde{G})$ is the full subgroupoid of $H$ with objects $f(x)$ and $f(\eta(x))$.  Since $f$ is fully faithful, its restriction $f_{|\widetilde{G}}: \widetilde{G}\rightarrow f(\widetilde{G})$ is an isomorphism from $\widetilde{G}$ to $f(\widetilde{G})$. Thus take $\widetilde{g} = [f_{|\widetilde{G}}]^{-1}$. Finally, we have a (right) homotopy between $\widetilde{g}\circ f_{|\widetilde{G}}$ and $1_{\widetilde{G}}$ since actually we have the equality $\widetilde{g}\circ f_{|\widetilde{G}} = 1_{\widetilde{G}}$.\medskip

\item Second case, let us assume that $G_{\text{max}}\ne\varnothing$.\medskip

Thanks to $G_{\text{max}}\ne G$ there exists $x\in \text{Ob}(G)-\text{Ob}(G_{\text{max}})$. We denote by $\widetilde{G}$ the full subgroupoid of $G$ generated by $\text{Ob}(G_{\text{max}})\cup \lbrace x,\eta(x)\rbrace$. Note that $x\ne \eta(x)$ since $G\f\subseteq G_{\text{max}}$.\\
We denote by $H_{\text{max}}$ the full subgroupoid of $H$ generated by the set of objects $\lbrace f(z)|z\in \text{Ob}(G_{\text{max}}) \rbrace$.\\
We need to distinguish two cases depending on whether $f(x)$ belongs to $H_{\text{max}}$ or it does not.
\begin{itemize}\medskip

\item First, assume that $f(x)\notin H_{\text{max}}$.\medskip

One has $f(x)\notin H\f$ since $H\f\subseteq H_{\text{max}}$. We denote by $\widetilde{H}$ the full subgroupoid of $H$ generated by $\text{Ob}(H_{\text{max}})\cup\lbrace f(x), \theta(f(x))\rbrace$. Note that $G_{\text{max}}\ne \varnothing$, hence let $z$ denote an element of $G_{\text{max}}$. By weak connectedness, there exist an element $y$ of $\lbrace z,\eta(z)\rbrace$ and an isomorphism $c$ in $G$ from $g_{\text{max}}(f(y))$ to $x$. Thanks to lemma \ref{lem2} the following square is a pushout square in $\GGpd$, 
$$\xymatrix{\underset{~}{S(\b1)}\ar[rr]^l\ar@{^{(}->}[dd]_{S(i)}&&G_{\text{max}}\ar[dd]\\\\S(\I)\ar[rr]_k&&\widetilde{G}\pushoutcorner\\}$$
where with the notations of lemma \ref{lem2} for the domain and codomain of $S(i)$ one has
\begin{align*}
l(0) = g_{\text{max}}(f(y)), l(1)= \eta(g_{\text{max}}(f(y))) \\
k(0) = l(0), k(1) = l(1), k(0') = x, k(1') = \eta(x) \\
k(\phi) = c, k(\psi) = \eta(c).
\end{align*}
Again thanks to lemma \ref{lem2} the following square is a pushout square,
$$\xymatrix{\underset{~}{S(\b1)}\ar[rr]^l\ar@{^{(}->}[dd]_{S(i)}&&H_{\text{max}}\ar[dd]\\\\S(\I)\ar[rr]_k&&\widetilde{H}\pushoutcorner\\}$$
where 
\begin{align*}
l(0) = f(y), l(1) = \theta(f(y))\\
k(0)= l(0), k(1)= l(1), k(0')= f(x), k(1')= \theta(f(x))\\
k(\phi)= f(c), k(\psi)= \theta(f(c)).
\end{align*} 
Now we want to use the universal property of the pushout square above to get $\widetilde{g}$ as required. One has the following commutative square, 
$$\xymatrix{\underset{~}{S(\b1)}\ar[rr]^l\ar@{^{(}->}[dd]_{S(i)}&&H_{\text{max}}\ar[dd]^{\iota\circ g_{\text{max}}}\\\\S(\I)\ar[rr]_j&&\widetilde{G}}$$
where $\iota$ is the inclusion from $G_{\text{max}}$ to $\widetilde{G}$ and $j$ is defined by
\begin{align*}
j(0) = g_{\text{max}}(l(0)), j(1)= \eta(j(0))\\
j(0')= x, j(1') = \eta(x)\\
j(\phi)= c, j(\psi) = \eta(c).
\end{align*}
Now thanks to the universal property of the pushout we get a map $\widetilde{g}$ as follows, 
$$\xymatrix{\underset{~}{S(\b1)}\ar[rr]^l\ar[dd]_{S(i)}&&H_{\text{max}}\ar@{^{(}->}[dd]\ar@/^/[rddd]^{\iota\circ g_{\text{max}}}\\\\S(\I)\ar[rr]^k \ar@/_/[rrrd]_j&&\widetilde{H}\pushoutcorner\ar@{-->}[rd]^{\widetilde{g}}\\&&&\widetilde{G}\,.}$$ 
The reader can easily check that $f\circ \widetilde{g}$ is the inclusion from $\widetilde{H}$ to $H$. 
Last, we need to provide a homotopy $\widetilde{h}$ between $\widetilde{g}\circ f_{|\widetilde{G}}$ and $1_{\widetilde{G}}$. \\
Recall that one has $h_{\text{max}}:G_{\text{max}}\rightarrow PG_{\text{max}}$ such that
$$\xymatrix@C=1,5cm{&&PG_{\text{max}}\coml\ar@{->>}[dd]\\\\G_{\text{max}}\ar[rruu]^{h_{\text{max}}}\ar[rr]_-{<g_{\text{max}}\circ f_{\text{max}},~1_{G_{\text{max}}}>}&&**[r]{G_{\text{max}}\times G_{\text{max}}}}$$ 
where $f_{\text{max}}:= f_{|{G_{\text{max}}}}$.\medskip

Now, consider the following square, 
$$\xymatrix{\underset{~}{G_{\text{max}}}\ar[r]^{h_{\text{max}}}\ar@{^{(}->}_{\iota}[dd]&PG_{\text{max}}\ar[r]^{P\iota}&P\widetilde{G}\ar@{->>}[dd]\\\\\widetilde{G} \ar[rr]_{<\widetilde{g}\circ f_{|\widetilde{G}},~1_{\widetilde{G}}>}&&\widetilde{G}\times\widetilde{G}\,.\\}$$ 
The inclusion $\iota$ from $G_{\text{max}}$ to $\widetilde{G}$ is a trivial cofibration thanks to our characterization of trivial cofibrations (see proposition \ref{prop3}). The above square can be rewritten as  
$$\xymatrix @C=3cm {\underset{~}{G_{\text{max}}}\ar@/_2pc/[rrdd]|{<g_{\text{max}}\circ f_{\text{max}},~1_{G_{\text{max}}}>}\ar[rr]^{h_{\text{max}}}\ar@{^{(}->}_{\iota}[dddd]&&PG_{\text{max}}\comr\coml\ar@{->>}[dd]\ar[r]^{P\iota}&P\widetilde{G}\ar@{->>}[dddd]\\\\&&G_{\text{max}}\times G_{\text{max}}\ar@/-1pc/[rdd]&&\\\\\widetilde{G}\ar[rrr]_{<\widetilde{g}\circ
f_{|\widetilde{G}},~1_{\widetilde{G}}>}&&&\widetilde{G}\times\widetilde{G}\,.\\}$$ To prove this square is commutative it suffices to prove that its bottom triangle commutes, 
$$\xymatrix{G_{\text{max}}\ar[rrr]^{<g_{\text{max}}\circ f_{\text{max}},~1_{G_{\text{max}}}>}
\ar@{^{(}->}_{\iota}[dd]&&&**[r]{G_{\text{max}}\times G_{\text{max}}} \ar[dd]\\\\\widetilde{G}\ar[rrr]_{<\widetilde{g}\circ f_{|\widetilde{G}},~1_{\widetilde{G}}>}&&&**[r]{\widetilde{G}\times\widetilde{G}}\\}$$ 
and this diagram indeed commutes since $\widetilde{g}_{|_{H_{\text{max}}}}= g_{\text{max}}$\\
and $(f_{|{\widetilde{G}}})_{|{G_{\text{max}}}}= f_{|{G_{\text{max}}}}:= f_{\text{max}}$. So one has a diagonal filler $\widetilde{h}$, 
$$\xymatrix @C=2cm {\underset{~}{G_{\text{max}}}\ar[r]^{h_{\text{max}}}\ar@{^{(}->}_{\iota}[dd]&PG_{\text{max}}\ar[r]^{P\iota}&P\widetilde{G}\ar@{->>}[dd]\\\\\widetilde{G} \ar[rruu]^{\widetilde{h}}\ar[rr]_{<\widetilde{g}\circ f_{|\widetilde{G}},~1_{\widetilde{G}}>}&&
\widetilde{G}\times\widetilde{G}\\}$$
and $\widetilde{h}$ is the homotopy we are looking for.\\ We have $(G_{\text{max}}, g_{\text{max}}, h_{\text{max}})<(\widetilde{G},\widetilde{g},\widetilde{h})\in S$ which contradicts the maximality of $(G_{\text{max}}, g_{\text{max}}, h_{\text{max}})$.\medskip

\item Second case,  assume that $f(x)$ belongs to $H_{\text{max}}$.\medskip

In this case the full subgroupoid of $H$ generated by $\text{Ob}(H_{\text{max}})\cup \lbrace f(x), \theta(f(x))\rbrace$ is still $H_{\text{max}}$. We still denote by $\iota$ the inclusion from $G_{\text{max}}$ to $\widetilde{G}$.\\
Take $\widetilde{g} = \iota\circ g_{\text{max}}$ which makes the following square commutes,
$$\xymatrix{\underset{~}{H_{\text{max}}}\ar[rr]^{\widetilde{g}}\ar@{^{(}->}[rdd]&&\widetilde{G}\ar[ldd]^{f_{|\widetilde{G}}}\\&\circlearrowright&\\&H\, .&}$$
Note that we have a homotopy $h_{\text{max}}$,
$$\xymatrix@C=2cm{&&PG_{\text{max}}\coml\ar@{->>}[dd]\\\\G_{\text{max}}\ar[rruu]^{h_{\text{max}}}\ar[rr]_-{<g_{\text{max}}\circ f_{\text{max}},~1_{G_{\text{max}}}>}&&**[r]{G_{\text{max}}\times G_{\text{max}}}\,.}$$ 
By a previous argument the following square commutes,
$$\xymatrix@C=2cm{\underset{~}{G_{\text{max}}}\ar[r]^{h_{\text{max}}}\ar@{^{(}->}[dd]_{\iota}&PG_{\text{max}}\ar[r]^{P\iota}&P\widetilde{G}\ar@{->>}[dd]\\\\\widetilde{G} \ar[rr]_{<\widetilde{g}\circ f_{|\widetilde{G}},~1_{\widetilde{G}}>}&&\widetilde{G}\times\widetilde{G}\,.\\}$$ 
Since the inclusion $\iota$ is a trivial cofibration we get the desired homotopy $\widetilde{h}$ as a diagonal filler,
$$\xymatrix @C=2cm {\underset{~}{G_{\text{max}}}\ar[r]^{h_{\text{max}}}\ar@{^{(}->}[dd]_{\iota}&PG_{\text{max}}\ar[r]^{P\iota}&P\widetilde{G}\ar@{->>}[dd]\\\\\widetilde{G} \ar[rruu]^{\widetilde{h}}\ar[rr]_{<\widetilde{g}\circ f_{|\widetilde{G}},~1_{\widetilde{G}}>}&&
\widetilde{G}\times\widetilde{G}\\}$$
with $(G_{\text{max}}, g_{\text{max}}, h_{\text{max}})< (\widetilde{G},\widetilde{g},\widetilde{h})$
which contradicts the maximality of $(G_{\text{max}}, g_{\text{max}}, h_{\text{max}})$.\\
Thus, eventually $G_{\text{max}} = G$ and $(G_{\text{max}}, g_{\text{max}}, h_{\text{max}})$ exhibits $f$ as a (right) homotopy equivalence.

\end{itemize}
\end{itemize} 
\end{proof}

\section{The failure of univalence}
\label{sec:tfou}

Now, we want to determine in our model the space of equivalences, \textit{i.e.} we want to determine the fibration denoted by $E\rightarrow U\times U$ in $\GGpd$ that corresponds to the interpretation of the dependent type 
$$(A:\text{Type}),(B:\text{Type})\pr\text{Equiv}(A,B)\ty.$$ It will allow us to check if univalence holds for the universe of the previous section.\bigskip

\begin{notn}
When $V_\kappa$ is $\text{Gpd}_\Delta(V_\kappa)$ we denote $U_{V_\kappa}$ (respectively $\widetilde{U_{V_\kappa}}$) by $U_\Delta$ (respectively $\widetilde{U}_\Delta$) and $p_\Delta:\widetilde{U}_{\Delta}\rightarrow U_\Delta$ our universe. Recall that $\text{Gpd}_\Delta(V_\kappa)$ denotes the univalent universe of $\kappa$-small discrete groupoids in the groupoid model of Hofmann and Streicher.
\end{notn}\bigskip

\begin{prop}\label{prop:fiberspaceofeq}
The set of objects of the fiber of $E\rightarrow U_{\Delta}\times U_{\Delta}$ over a pair $(A,B)\in U_{\Delta}\times U_{\Delta}$ is the set of isomorphisms in $U_{\Delta}$ between $A$ and $B$. Moreover, the involution on $E$ maps $(A,B,\rho)$, where $\rho$ is an isomorphism from $A$ to $B$, to $(u(A),u(B),u(\rho))$.
\end{prop}
\begin{proof}
To achieve this goal we have to unfold all the interpretation of this judgement. Recall that 
$$\text{Equiv}(A,B):=\underset{f:A\rightarrow B}{\textstyle\sum\limits}[\:(\underset{s:B\rightarrow A}{\textstyle\sum\limits}~\underset{b:B}{\textstyle\prod\limits}f(s(b))=b)\times(\underset{r:B\rightarrow A}{\textstyle\sum\limits}~\underset{a:A}{\textstyle\prod\limits}r(f(a))=a)\:]$$ 
where $f(s(b))=b$ (\textit{resp. }$r(f(a))=a)$ stands for the identity type $\text{Id}_{B}(f(s(b)),b)$\\
(\textit{resp. }$\text{Id}_{A}(r(f(a)),a))$.\\\\
We shorten $\m{(A:\text{Type}),(B:\text{Type})\pr\text{Equiv}(A,B)\ty}\text{ by }\m{\text{Equiv}(A,B)}$ \\
(and similary for other types). One has 
$$\m{\text{Equiv}(A,B)}=\m{A\rightarrow B}\circ\m{(\underset{s:B\rightarrow A}{\textstyle\sum\limits}~\underset{b:B}{\textstyle\prod\limits}f(s(b))=b)\times(\underset{r:B\rightarrow A}{\textstyle\sum\limits}~\underset{a:A}{\textstyle\prod\limits}r(f(a))=a)}$$ 
where 
$$\m{A\rightarrow B}:\m{\text{Type}.\text{Type}.A\rightarrow B}\twoheadrightarrow U\times U$$\\
is the interpretation of the judgement 
$$(A:\text{Type}),(B:\text{Type})\pr A\rightarrow B\ty.$$ 
Moreover\\\\
$\m{(\underset{s:B\rightarrow A}{\textstyle\sum\limits}~\underset{b:B}{\textstyle\prod\limits}f(s(b))=b)\times(\underset{r:B\rightarrow A}{\textstyle\sum\limits}~\underset{a:A}{\textstyle\prod\limits}r(f(a))=a)}:\\
\m{\text{Type}.\text{Type}.A\rightarrow B.(\textstyle\sum\limits-)\times(\textstyle\sum\limits-)}\twoheadrightarrow\m{\text{Type}.\text{Type}.A\rightarrow B}$\\\\
is the interpretation of the following judgement 
$$(A:\text{Type}),(B:\text{Type}),(f:A\rightarrow B)\pr(\underset{s:B\rightarrow A}{\textstyle\sum\limits}~\underset{b:B}{\textstyle\prod\limits}f(s(b))=b)\times(\underset{r:B\rightarrow A}{\textstyle\sum\limits}~\underset{a:A}{\textstyle\prod\limits}r(f(a))=a)\ty\,.$$
One has $\m{A\rightarrow B}=\Pi_{\m{A}}\m{B}$ where $\m{A}:\m{\text{Type}.\text{Type}.A}\twoheadrightarrow U\times U$ corresponds to the interpretation of $(A:\text{Type}),(B:\text{Type})\pr A\ty$,\\\\
and $\m{B}:\m{\text{Type}.\text{Type}.A.B}\twoheadrightarrow\m{\text{Type}.\text{Type}.A}$ corresponds to the interpretation of $(A:\text{Type}),(B:\text{Type}),(a:A)\pr B\ty$.\\\\
Actually, $\m{A}$ is the pullback of $p$ along the first projection $pr_{1}$, 
$$\xymatrix{{U\times\widetilde{U}}\pullbackcorner\ar[rr]^{pr_{2}}\ar@{->>}[dd]_{\m{A}}&&{\widetilde{U}}\ar@{->>}[dd]^p\\\\{U\times U}\ar[rr]_{pr_{1}}&&U}$$
with $\m{A}=<p\circ pr_{2},~pr_{1}>$.
\\Moreover, $\m{B}$ is the pullback of $p$ along the first projection $pr_{1}$, 
$$\xymatrix{{\widetilde{U}\times\widetilde{U}}\ar[rr]^{pr_{2}}\ar@{->>}[dd]_{\m{B}}&&{\widetilde{U}}\ar@{->>}[dd]^p\\\\{U\times\widetilde{U}}\ar[rr]_{pr_{1}}&&U\\}$$ 
with $\m{B}=<p\circ pr_{2},~pr_{1}>$.\\
Hence $\m{A\rightarrow B}=\Pi_{\m{A}}\m{B}:\m{\text{Type}.\text{Type}.A\rightarrow B}\twoheadrightarrow U\times U$ is such that above $(A,B)\in U\times U$ there is the groupoid of partial sections $s:\m{A}^{-1}\lbrace{(A,B)}\rbrace\rightarrow\widetilde{U}\times\widetilde{U}$ of $\m{B}$.\\\\
Next, note that : 
$$\m{(\underset{s:B\rightarrow A}{\textstyle\sum\limits}~\underset{b:B}{\textstyle\prod\limits}f(s(b))=b)\times(\underset{r:B\rightarrow A}{\textstyle\sum\limits}~\underset{a:A}{\textstyle\prod\limits}r(f(a))=a)}$$
$$=\m{\underset{s:B\rightarrow A}{\textstyle\sum\limits}~\underset{b:B}{\textstyle\prod\limits}f(s(b))=b}\circ\m{\underset{r:B\rightarrow A}{\textstyle\sum\limits}~\underset{a:A}{\textstyle\prod\limits}r(f(a))=a}$$\\\\\\
with $\m{\underset{r:B\rightarrow A}{\textstyle\sum\limits}~\underset{a:A}{\textstyle\prod\limits}r(f(a))=a}:\\\m{\text{Type}.\text{Type}.A\rightarrow B.\underset{s:B\rightarrow A}{\textstyle\sum\limits}~\underset{b:B}{\textstyle\prod\limits}f(s(b))=b.\underset{r:B\rightarrow A}{\textstyle\sum\limits}~\underset{a:A}{\textstyle\prod\limits}r(f(a))=a}\\
\twoheadrightarrow \m{\text{Type}.\text{Type}.A\rightarrow B.\underset{s:B\rightarrow A}{\textstyle\sum\limits}~\underset{b:B}{\textstyle\prod\limits}f(s(b))=b}$\\\\\\
and $\m{\underset{s:B\rightarrow A}{\textstyle\sum\limits}~\underset{b:B}{\textstyle\prod\limits}f(s(b))=b}:\\\m{\text{Type}.\text{Type}.A\rightarrow B.\underset{s:B\rightarrow A}{\textstyle\sum\limits}~\underset{b:B}{\textstyle\prod\limits}f(s(b))=b}\twoheadrightarrow \m{\text{Type}.\text{Type}.A\rightarrow B}.$\\\\\\
Since these two expressions are very similar we determine only\\
$\m{\underset{s:B\rightarrow A}{\textstyle\sum\limits}~\underset{b:B}{\textstyle\prod\limits}f(s(b))=b}$.\\\\\\
Once again, $\m{\underset{s:B\rightarrow A}{\textstyle\sum\limits}~\underset{b:B}{\textstyle\prod\limits}f(s(b))=b}=\m{B\rightarrow A}\circ\m{\underset{b:B}{\textstyle\prod\limits}f(s(b))=b}$\\\\\\
where $\m{B\rightarrow A}:\m{\text{Type}.\text{Type}.A\rightarrow B.B\rightarrow A}\twoheadrightarrow\m{\text{Type}.\text{Type}.A\rightarrow B}$\\\\\\
and $\m{\underset{b:B}{\textstyle\prod\limits}f(s(b))=b}:\\\m{\text{Type}.\text{Type}.A\rightarrow B.B\rightarrow A.\underset{b:B}{\textstyle\prod\limits}f(s(b))=b}\twoheadrightarrow \m{\text{Type}.\text{Type}.A\rightarrow B.B\rightarrow A}.$\\\\\\
One has 
$$\m{B\rightarrow A}=\Pi_{\m{B}}\m{A}$$
where 
$$\m{B}:\m{\text{Type}.\text{Type}.A\rightarrow B.B}\twoheadrightarrow\m{\text{Type}.\text{Type}.A\rightarrow B}$$ 
is the interpretation of 
$$(A:\text{Type}),(B:\text{Type}),(f:A\rightarrow B)\pr B\ty$$
and
$$\m{A}:\m{\text{Type}.\text{Type}.A\rightarrow B.B.A}\twoheadrightarrow\m{\text{Type}.\text{Type}.A\rightarrow B.B}$$ 
is the interpretation of 
$$(A:\text{Type}),(B:\text{Type}),(f:A\rightarrow B),(b:B)\pr A\ty.$$\\
The morphism $\m{B}$ is the pullback of $p$ along $pr_{2}\circ\m{A\rightarrow B}$, 
$$\xymatrix{\m{\text{Type}.\text{Type}.A\rightarrow B.B}\ar[rr]\ar@{->>}[dd]_{\m{B}}&&{\widetilde{U}}\ar@{->>}[dd]^p\\\\{\m{\text{Type}.\text{Type}.A\rightarrow B}}\ar[rr]_{\qquad\qquad pr_{2}\circ\m{A\rightarrow B}}&&U\\}$$\\
where $pr_{2}\circ\m{A\rightarrow B}$ maps an element $(A,B,s)$ to $B$. More explicitely above $(A,B,s)$ by $\m{B}$ one has the elements of the form $(A,\dot{B},s)$ where $\dot{B}$ denotes an element of $\widetilde{U}$ such that $p(\dot{B})=B\in U$. To be specific $s$ is such that it maps an element  of the form $(B,\dot{A})\in U\times\widetilde{U}$ to an element $s(B,\dot{A})\in\widetilde{U}\times\widetilde{U}$ such that $s(B,\dot{A})$ is of the form $(\dot{A},\dot{B})$.\\\\
The morphism $\m{A}$ is the pullback of $p$ along the map : 
$$\ffour{pr_{1}\circ\m{A\rightarrow B}\circ\m{B}:{\m{\text{Type}.\text{Type}.A\rightarrow B.B}}}{U}{(A,\dot{B},s)}{A}$$\\
$$\xymatrix{\m{\text{Type}.\text{Type}.A\rightarrow B.B.A}\pullbackcorner\ar[rr]\ar@{->>}[dd]_{\m{A}}&&{\widetilde{U}}\ar@{->>}[dd]^p\\\\{\m{\text{Type}.\text{Type}.A\rightarrow B.B}}\ar[rr]_{\quad\qquad\qquad pr_{1}\circ\m{A\rightarrow B}\circ\m{B}}&&U\\}$$\\\\
so above $(A,\dot{B},s)\in \m{\text{Type}.\text{Type}.A\rightarrow B.B}$ by $\m{A}$ one has the elements of the form $(\dot{A},\dot{B},s)$.\\
Hence $\m{B\rightarrow A}:\m{\text{Type}.\text{Type}.A\rightarrow B.B\rightarrow A}\twoheadrightarrow\m{\text{Type}.\text{Type}.A\rightarrow B}$ is such that above an element $(A,B,s)\in\m{\text{Type}.\text{Type}.A\rightarrow B}$ there is the groupoid of partial sections $t$, 
$$\xymatrix{\underset{~}{\m{B}^{-1}\lbrace{(A,B,s)}\rbrace}\ar[rr]^t\ar@{^{(}->}[rdd]&&\m{\text{Type}.\text{Type}.A\rightarrow B.B\rightarrow A} \ar[ldd]^{\m{A}}\\&\circlearrowright&\\&\m{\text{Type}.\text{Type}.A\rightarrow B.B}&\\}$$ 
in other words, such a section $t$ maps an element of the form $(A,\dot{B},s)$ to an element of  the form $(\dot{A},\dot{B},s)$.\\\\
One has to determine $\m{\underset{b:B}{\textstyle\prod\limits}f(s(b))=b}$ which is the interpretation of the judgement 
$$(A:\text{Type}),(B:\text{Type}),(f:A\rightarrow B),(s:B\rightarrow A)\pr\underset{b:B}{\textstyle\prod\limits}f(s(b))=b\ty~~.$$
One has 
$$\m{\underset{b:B}{\textstyle\prod\limits}f(s(b))=b}={\textstyle\prod\limits}_{\m{B}}\m{f(s(b))=b}$$ 
where 
$$\m{B}:\m{\text{Type}.\text{Type}.A\rightarrow B.B\rightarrow A.B}\twoheadrightarrow\m{\text{Type}.\text{Type}.A\rightarrow B.B\rightarrow A}$$\\
is the interpretation of the judgement 
$$(A:\text{Type}),(B:\text{Type}),(f:A\rightarrow B),(s:B\rightarrow A)\pr B\ty\,,$$\\
and $\m{f(s(b))=b}:\\
\m{\text{Type}.\text{Type}.A\rightarrow B.B\rightarrow A.B.f(s(b))=b}\twoheadrightarrow\m{\text{Type}.\text{Type}.A\rightarrow B.B\rightarrow A.B}$\\\\
is the interpretation of 
$$(A:\text{Type}),(B:\text{Type}),(f:A\rightarrow B),(s:B\rightarrow A)(b:B)\pr {f(s(b))=b}\ty~~.$$\\
The morphism $\m{B}$ is the pullback of $p$ along the projection 
$$\ffour{\m{\text{Type}.\text{Type}.A\rightarrow B.B\rightarrow A}}{U}{(A,B,s,t)}{B}$$ 
$$\xymatrix{\m{\text{Type}.\text{Type}.A\rightarrow B.B\rightarrow A.B}\pullbackcorner\ar[rr]\ar@{->>}[dd]_{\m{B}}&&{\widetilde{U}}\ar@{->>}[dd]^p\\\\{\m{\text{Type}.\text{Type}.A\rightarrow B.B\rightarrow A}}\ar[rr]&&U\\\\}$$
in other words, $\m{B}$ maps $(A,B,s,t,b)$ to $(A,B,s,t)$ with $b\in B$. The interpretation of  the judgement 
$$(B:\text{Type}),(b:B),(b':B)\pr{b=b'}\ty$$ is given by the small fibration $\xymatrix{P_{U}\widetilde{U}\ar[r]|->{\SelectTips{2cm}{}\object@{>>}}|-->{\SelectTips{eu}{}}&\widetilde{U}\times_{U}\widetilde{U}}$ in the following diagram, 
$$
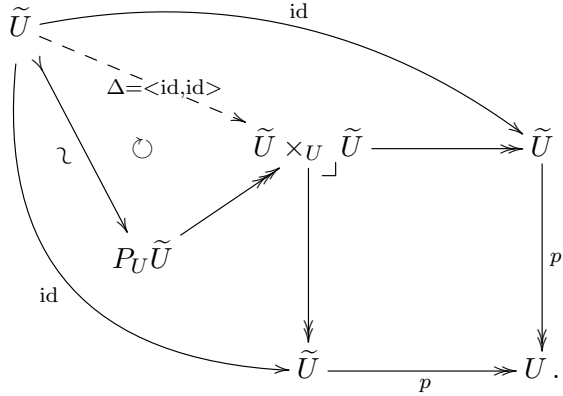
\xymatrix{\underset{~}{\widetilde{U}}\ar@{-->}[rrd]|{\Delta=<\text{id},\text{id}>}\ar@/^2pc/[rrrrd]^{\text{id}}\ar@/_4pc/[rrddd]_{\text{id}}\ar@{ >->}[rdd]_[@]{\hbox{$\sim$}}&&&&\\&\circlearrowright&\widetilde{U}\times_{U}~\widetilde{U}\pullbackcorner\ar@{->>}[rr]\ar@{->>}[dd]&&\widetilde{U}\ar@{->>}[dd]^p\\&P_{U}\widetilde{U}\ar[ru]|->{\SelectTips{2cm}{}\object@{>>}}|-->{\SelectTips{eu}{}}&&&\\&&\widetilde{U}\ar@{->>}[rr]_p&&U\,.}$$ 
Now assume that the groupoids involved in the definition of $U$ and $\widetilde{U}$ are discrete, in this case $\Delta$ is itself a small fibration (\textit{i.e.} an isofibration with $\kappa$-small fibers equipped with a split cleavage), thus for $\xymatrix{P_{U}\widetilde{U}\ar[r]|->{\SelectTips{2cm}{}\object@{>>}}|-->{\SelectTips{eu}{}}&\widetilde{U}\times_{U}\widetilde{U}}$ on can choose $\Delta$ itself. The fibration $\m{f(s(b))=b}$ is the pullback of $\Delta:\widetilde{U}\rightarrow\widetilde{U}\times_{U}\widetilde{U}$ along the projection $$\ffour{\m{\text{Type}.\text{Type}.A\rightarrow B.B\rightarrow A.B}}{\widetilde{U}\times_{U}\widetilde{U}}{(A,\dot{B},s,t)}{(pr_{2}[s(B,~pr_{1}[t(A,\dot{B},s)])],~\dot{B})}$$\\
$$\xymatrix{\m{\text{Type}.\text{Type}.A\rightarrow B.B\rightarrow A.B.f(s(b))=b}\pullbackcorner\ar[rr]\ar[dd]_{\m{f(s(b))=b}}|->{\SelectTips{2cm}{}\object@{>>}}|-->{\SelectTips{eu}{}}&&{\widetilde{U}}\ar[dd]^{\Delta}|->{\SelectTips{2cm}{}\object@{>>}}|-->{\SelectTips{eu}{}}\\\\{\m{\text{Type}.\text{Type}.A\rightarrow B.B\rightarrow A.B}}\ar[rr]&&\widetilde{U}\times_{U}\widetilde{U}\\\\}$$ in other words $\m{\text{Type}.\text{Type}.A\rightarrow B.B\rightarrow A.B.f(s(b))=b}$ consists in the tuples $(A,\dot{B},s,t)$ with $A\in U$, let say $A:=(A_{0},A_{1},\varphi_{A}),~B\in U$, let say $B:=(B_{0},B_{1},\varphi_{B})$ and $s:A_{0}\rightarrow B_{0},~t:B_{0}\rightarrow A_{0}$, such that $s\circ t=1_{B_{0}}$, and $\m{f(s(b))=b}$ maps such an element $(A,\dot{B},s,t)$ to itself.\\
Hence $\m{\underset{b:B}{\textstyle\prod\limits}f(s(b))=b}={\textstyle\prod\limits}_{\m{B}}\m{f(s(b))=b}$ is such that above an element $(A,B,s,t)$ one has $(A,B,s,t)$ itself if $s\circ t=1_{B_{0}}$ (otherwise the fiber is empty).\\
Last, above $(A,B,s)$ by $\m{\underset{s:B\rightarrow A}{\sum}~\underset{b:B}{\textstyle\prod\limits}f(s(b))=b}$ one has the right inverses of $s$ (with $s$ seen as a morphism between the discrete groupoids $A_{0}$ and $B_{0}$). In the same way, above $(A,B,s)$ by $\m{(\underset{s:B\rightarrow A}{\sum}~\underset{b:B}{\textstyle\prod\limits}f(s(b))=b)\times(\underset{r:A\rightarrow B}{\sum}~\underset{a:A}{\textstyle\prod\limits}r(f(a))=a)}$ one has the tuples $(A,B,s,t,u)$ where $t$ is a right inverse of $s$ and $u$ is a left inverse of $s$. But $A_{0}$ and $B_{0}$ being sets (discrete groupoids) one has $t=u$.\\
Finally, $\m{\text{Equiv}(A,B)}$ is a fibration such that above $(A,B)\in U\times U$ one has the isomorphisms in $U$ between $A$ and $B$. We can determine the involution on $E:=\text{dom}(\m{\text{Equiv}(A,B)})$.\\
Indeed, in order from the beginning, $\text{dom}(\m{A})=\text{dom}(<p\circ pr_{2},~pr_{1}>)$ is equipped with the involution $u\times\tilde{u}$ on $U\times \widetilde{U}$ and $\text{dom}(\m{B})=\text{dom}(<p\circ pr_{2},~pr_{1}>)$ is equipped with the involution $\tilde{u}\times\tilde{u},$\\
and $\text{dom}(\m{A\rightarrow B})=\text{dom}(\Pi_{\m{A}}\m{B})=\m{\text{Type}.\text{Type}.A\rightarrow B}$ is equipped with the involution that maps $(A,B,s)$ to $(u(A),u(B),s')$ where $s'$ maps an element of the form $(u(B),\dot{u(A)})$ to $(\dot{u(A)},~\dot{u(B)})$, more specifically :\\ 
\begin{align*} 
s'(u(B),\dot{u(A)})&=(\tilde{u}\times\tilde{u})(s(u\times\tilde{u}(u(B),\dot{u(A)})))\\&=(\tilde{u}\times\tilde{u})(s(B,\tilde{u}(\dot{u(A)})))\\&=(\dot{u(A)},~\tilde{u}(\dot{B})).\\
\end{align*}
Let determine the involution on $\text{dom}(\m{B\rightarrow A})=\m{\text{Type}.\text{Type}.A\rightarrow B.B\rightarrow A}$. One has $\text{dom}(\m{B\rightarrow A})=\text{dom}(\Pi_{\m{B}}\m{A})$ with $\text{dom}(\m{B})=\m{\text{Type}.\text{Type}.A\rightarrow B.B}$ equipped with the involution that maps $(A,\dot{B},s)$ to $(u(A),\tilde{u}(\dot{B}),s')$ with $s'$ as above. Hence $\text{dom}(\m{B\rightarrow A})=\text{dom}(\Pi_{\m{B}}\m{A})=\m{\text{Type}.\text{Type}.A\rightarrow B.B\rightarrow A}$ is equipped with the involution that maps $(A,B,s,t)$ to $(u(A),u(B),s',t')$ with $t'$ such that the following square commutes, 
$$\xymatrix{\underset{~}{\m{B}^{-1}\lbrace{(u(A),u(B),s')}\rbrace}\ar[rr]^{t'}\ar@{^{(}->}[rdd]&&\m{\text{Type}.\text{Type}.A\rightarrow B.B.A} \ar[ldd]^{\m{A}}\\&\circlearrowright&\\&\m{\text{Type}.\text{Type}.A\rightarrow B.B}&\\}$$\\
in other words $t'$ maps an element of the form $(u(A),\dot{u(B)},s')$ to an element of the form $(\dot{u(A)},\dot{u(B)},s')$, to be specific this last element is the one you get by applying the involution on $\m{\text{Type}.\text{Type}.A\rightarrow B.B.A}$ to $t(A,\tilde{u}(\dot{u(B)}),s)$, namely $(\tilde{u}(\dot{A}),\tilde{u}(\dot{u(B)}),s')$. Let determine the involution on 
$$\text{dom}(\m{\underset{b:B}{\textstyle\prod\limits}f(s(b))=b})=\text{dom}({\textstyle\prod\limits}_{\m{B}}\m{f(s(b))=b}).$$ 
First, the involution on $\text{dom}(\m{B})=\m{\text{Type}.\text{Type}.A\rightarrow B.B\rightarrow A.B}$ maps $(A,B,s,t,b)$ to $(u(A),u(B),s',t',\varphi_{B}(b))$ where $B:=(B_{0},B_{1},\varphi_{B})$ and $b\in B_{0}$. The involution on $\text{dom}(\m{f(s(b))=b})=\m{\text{Type}.\text{Type}.A\rightarrow B.B\rightarrow A.B.f(s(b))=b}$ maps $(A,\dot{B},s,t)$ to $(u(A),\tilde{u}(\dot{B}),s',t')$.\\
Last, the involution on $E$ maps $(A,B,\rho)$, where $\rho=(\rho_{0},\rho_{1})$ is an isomorphism in $U$ between $A$ and $B$, to $(u(A),u(B),(\rho_{1},\rho_{0}))=(u(A),u(B),u(\rho)).$
\end{proof}\bigskip

\begin{prop}\label{prop7}
Univalence does not hold for our universe $p_{\Delta}:\widetilde{U}_{\Delta}\rightarrow U_{\Delta}$ in the projective type-theoretic fibration category $\GGpd_{\mathbf{proj}}$.\\
\end{prop}
\begin{proof}\label{proofprop7}
Recall that according to \cite{shulman:invdia} (see part 7 on the univalence axiom) and our proposition \ref{prop:fiberspaceofeq} the morphism $U\to E$ is defined as follows, 
$$\ffour{U}{E}{A}{(A,A,1_{A})}$$\\
\textit{i.e.} it maps an object $A$ to the identity isomorphism of $A$ in $U$. Note that this morphism is not surjective on the fixed points of $E$. Indeed, it is easy to find a fixed point $(A,B,\rho)$ of $E$ such that $\rho\ne$id.\\
For instance, take $A:=(\mathbb{N},\mathbb{N},1_{\mathbb{N}}),~B:=(2\mathbb{N},2\mathbb{N},1_{2\mathbb{N}})$\\
and $\rho=(\mathbb{N}\xrightarrow{\cong}2\mathbb{N},\mathbb{N}\xrightarrow{\cong}2\mathbb{N})$ where $\mathbb{N}\xrightarrow{\cong}2\mathbb{N}$ is the bijection that maps $n$ to $2n$, then $(A,B,\rho)$ is a fixed point of $E$ where $\rho$ is not the identity, so $(A,B,\rho)$ does not belong to the image of the morphism above. Note that we can even take $A=B$ and still find a fixed point of $E$ that does not belong to the image of our map $U\rightarrow E$. Indeed, take $A=B:=(\mathbb{Z},\mathbb{Z},1_{\mathbb{Z}})$ and $\rho=(\mathbb{Z}\xrightarrow{\cong}\mathbb{Z},\mathbb{Z}\xrightarrow{\cong}\mathbb{Z})$ where $\mathbb{Z}\xrightarrow{\cong}\mathbb{Z}$ is the bijection that maps $n$ to $-n$. According to \ref{thm5} the map $U\rightarrow E$ is not a homotopy equivalence, so univalence does not hold in our model.
\end{proof}

\section{The failure of function extensionality}
\label{sec:tffe}\bigskip

For some details about the meaning of function extensionality in a type-theoretic fibration category see (5.8) in \cite{shulman:invdia}. In particular, following Shulman, function extensionality in the internal language of a type-theoretic fibration category means that for every fibrations $f:P\fib X$ and $g:X\fib A$, there is a map
\begin{align*}
  \Pi_g (\text{iscontr}_X(P)) \to \text{iscontr}_A(\Pi_g P).\label{eq:catfunext}
\end{align*}.  \bigskip

\begin{prop}\label{prop8}
Function extensionality does not hold in the internal type theory of our projective type-theoretic fibration category $\GGpd_{\mathbf{proj}}$.\\
\end{prop}
\begin{proof}\label{proofprop8}
According to lemma 5.9 in \cite{shulman:invdia} (be aware that in this lemma an acyclic fibration means a fibration which is also a homotopy equivalence), it suffices to prove that there exists a fibration $g$ and a fibration $f$ such that $f$ is a homotopy equivalence and $\Pi_{g}f$ is not a homotopy equivalence. In order to achieve this, take $g:S(\b1)\rightarrow\mathbf{1}!$ where we recall that $S(\b1)$ is the discrete groupoid with two objects $0$ and $1$ equipped with the involution that swaps the two points. Note that $g$ is a fibration since all objects are fibrant in $\GGpd$. Take $f:S(\I)\rightarrow S(\b1)$ where we recall that $S(\I)$ denotes the following groupoid 
$$\xymatrix{0\ar[d]_\phi & 1\ar[d]^\psi\\
0' & 1'}$$
with the swap involution and $f$ is defined by $f(0)= f(0')= 0,\\
f(\phi)= 1_{0}, f(1)= f(1')= 1~\text{and}~f(\psi)= 1_{1}$. The reader can easily check that $f$ is compatible with the involution, $f$ is fully faithful and surjective (and so is a trivial fibration in $\GGpd$) and $f$ restricted to fixed points is the identity (since $(S(\I))\f= (S(\b1))\f=\mathbf{0}$). So, according to \ref{thm5}, $f$ is a homotopy equivalence. Now, since $\Pi_{g}f:\text{dom}(\Pi_{g}f)\rightarrow\mathbf{1}!$ and $(\mathbf{1}!)\f= \mathbf{1}$ it suffices to prove that $\text{dom}(\Pi_{g}f)$ has at least two fixed points. A fixed point of $\text{dom}(\Pi_{g}f)$ is a couple $(*,s)$ where 
$$\xymatrix{S(\b1)\ar[rr]^s\ar@{=}[rdd]&&**[r]{S(\I)}\ar[ldd]^f\\&\circlearrowright&\\&S(\b1)&}$$ 
with $\pi_{g}f(s)=s$. In other words, $s$ is a section of $f$ such that $\pi_{g}f(s)=s$. But we have two such sections $s_{1}$ and $s_{2}$, indeed take 
$$\fsix{s_{1}:S(\b1)}{S(\I)}{0}{0}{1}{1}$$ 
one has $f\circ s_{1}= 1$ and $\begin{cases}s_{1}(0)= \text{swap}(s_{1}(1))\\s_{1}(1)=\text{swap}(s_{1}(0))\end{cases}$\\
where swap denotes the swap involution on $S(\I)$, so $\pi_{g}f(s_{1})=s_{1}$ and 
$$\fsix{s_{2}:S(\b1)}{S(\I)}{0}{0'}{1}{1'}$$ 
one has $f\circ{s_{2}}= 1$ and $\begin{cases}{s_{2}(0)}=\text{swap}(s_{2}(1))\\s_{2}(1)=\text{swap}(s_{2}(0))\end{cases}$ so $\pi_{g}f(s_{2})=s_{2}$.
\end{proof}\bigskip

\begin{rmk}\label{rmk8}
We recall that the univalence axiom implies function extensionality (\textit{cf} \cite{voevodsky:github}). 
Note that the above proposition involving non-discrete groupoids does not give us a proof that univalence does not hold in our universe of discrete groupoids equipped with involution. This is the reason why our characterization of homotopy equivalences in \ref{thm5} was mandatory. Also note that according to Remark 5.10 in \cite{shulman:invdia} function extensionality holds in the type-theoretic fibration category $\Gpd$, so we have broken function extensionality.  \\
\end{rmk}\bigskip

\section{Conclusion}
This chapter that treats the case of groupoids equipped with an involution gives support for drawing the conclusion that the projective model structure is not suitable for univalence. This is essentially due to homotopy equivalences being very strong in the projective setting.

\setchapterpreamble{
\begin{quote}
So far, known models of Voevodsky's \textit{Univalence Axiom} in functor categories have not allowed the existence of a non-trivial isomorphism in the index category. In this chapter we remedy this limitation by giving a model with one \textit{univalent universe} in the injective type-theoretic fibration category $\GGpd_{\mathbf{inj}}$, namely the type-theoretic fibration structure on $\GGpd$ provided by the subcategory given by the injective fibrations.\\
\end{quote}

}

\chapter{A univalent model in the injective type-theoretic fibration category $\GGpd_{\mathbf{inj}}$}
\label{sec:chp5}

\section{Organization}

We try to remedy the limitation on the index category by looking at one of the simplest index category with a non-trivial automorphism, namely  $\mathbf{B}(\mathbb{Z}/2\mathbb{Z})$ (recall that according to our notations we denote for convenience in the rest of this chapter this category simply by $\G$). We consider the category of groupoids $\mathbf{Gpd}$ as a simple target category. Thus, like the previous chapter we consider the presheaf category $\GGpd$. Recall this presheaf category is nothing but the category of groupoids with involution and equivariant functors between them. We endow this presheaf category with the injective fibrations, in particular in the first section \ref{sec:injttfc} we give a proof of the fact that the full-subcategory of injectively fibrant objects of $\GGpd$ equipped with the subcategory given by injective fibrations between fibrant objects is a type-theoretic fibration category. We denote this injective type-theoretic fibration category by $\GGpd_{\mathbf{inj}}$. In the section \ref{sec:explicit} we give a fairly explicit description of the injective fibrations.
In section \ref{sec:injmodel} we equip the injective type-theoretic fibration category $\GGpd_{\mathbf{inj}}$ with a univalent universe.\\

\section{The injective type-theoretic fibration category $\GGpd_{\mathbf{inj}}$}
\label{sec:injttfc}\bigskip

\begin{prop}\label{lem14}
Assume one has a category $\mathscr{C}$ equipped with a distinguish class of morphisms called fibrations such that
$\mathscr{C}$ is locally presentable and satisfies the additional assumption that for every fibration $g:A\fib B$, the pullback functor $g^{\star}:\mathscr{C}/B\rightarrow\mathscr{C}/A$ has a right adjoint $\Pi_{g}$. Moreover, assume that $\mathcal{D}$ is any small category. Then for any levelwise fibration $g:A\twoheadrightarrow B\text{~in~}[\:\mathcal{D},\mathscr{C}\:]$ the pullback functor $g^{\star}:[\:\mathcal{D},\mathscr{C}\:]/B\rightarrow[\:\mathcal{D},\mathscr{C}\:]/A$ has a right adjoint $\Pi_{g}$.\\
\end{prop}
\begin{proof}\label{prooflem14}
Since $\mathscr{C}$ is locally presentable and local presentability is preserved by taking functor category over any small category and is also preserved by slicing, in order to have a right adjoint it is enough that $g^{\star}:[\:\mathcal{D},\mathscr{C}\:]/B\rightarrow[\:\mathcal{D},\mathscr{C}\:]/A$ preserves small colimits. So, let $\mathcal{E}$ be any small category and $F:\mathcal{E}\rightarrow[\:\mathcal{D},\mathscr{C}\:]/B$ any  functor such that colim$F$ exists in $[\:\mathcal{D},\mathscr{C}\:]/B$. One wants to prove that 
$$g^{\star}(\text{colim}F)\cong\text{colim}(g^{\star}\circ F)~.$$
Due to the nature of a colimit in a slice category it is enough to check that dom$(g^{\star}(\text{colim}F))\in[\:\mathcal{D},\mathscr{C}\:]$ is (isomorphic to) colim(dom$~\circ~g^{\star}\circ F)$, where\\
dom : $[\:\mathcal{D},\mathscr{C}\:]/A\rightarrow[\:\mathcal{D},\mathscr{C}\:]$ is the domain functor.\\
Now, since colimits in a presheaf category are pointwise, \textit{i.e. }for every $d\in\mathcal{D}$,
$$[\:\text{colim~(dom~}\circ~g^{\star}\circ F)\:](d)\cong \underset{e}{\text{colim}~}[\:\text{dom}(g^{\star}(F(e)))(d)\:]~,$$ 
it is enough to check that dom$(g^{\star}(\text{colim}F))(d)\cong \underset{e}{\text{colim}~}[\:\text{dom}(g^{\star}(F(e)))(d)\:]$.\\
But pullbacks are pointwise (like any limit) so
$$\text{dom}(g^{\star}(\text{colim}F))(d)\cong\text{dom}(g^{\star}_d(\text{colim}F)_d)~.$$
Now by assumption on $\mathscr{C}$ and the fact that $g_d$ is a fibration, $g^{\star}_d$ has a right adjoint and so it preserves colimits.\\
Moreover, one has (colim$F)_d\cong\underset{e}{\text{colim}}(F(e))_d$.\\
Indeed, to prove that (colim$F)_d\cong\underset{e}{\text{colim}}(F(e))_d$ in $\mathscr{C}/{B(d)}$ it suffices to prove that dom((colim$F)_d)=\text{dom(colim}F)(d)\in\mathscr{C}$ is the colimit $\underset{e}{\text{colim}}~\text{dom}(F(e)_d)$, but one has dom(colim$F)\cong\text{colim~(dom}\circ F)\in [\:\mathcal{D},\mathscr{C}\:]$, so we conclude\\
dom(colim$F)(d)\cong\underset{e}{\text{colim}}~\text{(dom}(F(e))(d))=\underset{e}{\text{colim}}~(\text{dom}(F(e)_d))$.\\
As a consequence one has
\begin{align*}
g_d^{\star}\ (\text{colim}F)_d&\cong g_d^{\star}[\:\underset{e}{\text{colim}}(F(e))_d\:]\\
&\cong\underset{e}{\text{colim}}~[\:g_{d}^{\star}\ (F(e))_d\:],\\
\end{align*}
and finally, 
\begin{align*}
\text{dom}(g^{\star}_d\ (\text{colim}F)_d)&\cong \text{dom}(\underset{e}{\text{colim}}~g^{\star}_d\ (F(e))_d)\\
&\cong\underset{e}{\text{colim}}~[\:\text{dom}(g^{\star}_d\ (F(e))_d)\:]\\
&\cong\underset{e}{\text{colim}}~[\:\text{dom}(g^{\star}\,F(e))\,(d)\:]~.\\
\end{align*}
\end{proof}\bigskip

\begin{rmk}
By taking $\mathscr{C}= \Gpd$ and $\mathcal{D}= \mathbf{B}(\mathbb{Z}/2\mathbb{Z})$, the proposition \ref{lem14} gives us a proof that the pullback functor along any injective fibration (indeed injective fibrations are in particular levelwise fibrations) has a right adjoint. But note that in this specific case the construction of the right adjoint in chapter \ref{sec:chp4} (see \ref{thm2}) is still valid because right adjointness is a categorical notion, hence does not depend on the choice of the class of fibrations, moreover for any fibration $g:A\fib B$ the right adjoint $\Pi_g$ constructed there still maps injective fibrations over $A$ to injective fibrations over $B$. Indeed, this is equivalent to the fact that the pullback functor $g^{*}$ preserves injective trivial cofibrations, and this is the case since injective trivial cofibrations are levelwise, injective fibrations are in particular levelwise fibrations, pullbacks are pointwise and this fact is true in $\Gpd$ (see \ref{rightproper}). So it proves the required condition (4) in \ref{def:ttfc}.
\end{rmk}\bigskip

\begin{prop}\label{prop13}
Let $\mathscr{C}$ be a type-theoretic model category whose underlying model category is combinatorial, and let $\mathcal{I}$ be any small category. The injective model category $[\mathcal{I},\mathscr{C}]$ is a type-theoretic model category.\\
\end{prop}
\begin{proof}\label{proofprop13}
By assumption trivial cofibrations (between fibrant objects) are preserved by pullbacks along any fibration (between fibrant objects) in the target category $\mathscr{C}$, since pullbacks are pointwise, injective trivial cofibrations are levelwise and injective fibrations are in particular levelwise fibrations, we conclude the first required condition for the injective structure on $[\mathcal{I},\mathscr{C}]$.\\
Last, since the underlying category of $\mathscr{C}$ is locally presentable and injective fibrations are in particular levelwise fibrations, by applying \ref{lem14} we conclude that the pullback functor $g^*$ along any fibration has a right adjoint $\Pi_g$.
\end{proof}\bigskip

\begin{prop}\label{prop17}
The full subcategory of injectively fibrant objects of $\GGpd$, namely $\GGpd\f$, equipped with the subcategory given by injective fibrations between fibrant objects is a type-theoretic fibration category. 
\end{prop}
\begin{proof}\label{proofprop17}
It follows from the fact that the injective model structure on $\GGpd$ provides a type-theoretic model category. Indeed, since $\Gpd$ is combinatorial we can apply \ref{prop13} and the proposition \ref{shul2.13} in chapter 3 allows us to conclude.
\end{proof}\bigskip

\begin{rmk}\label{rmk13}
Since Quillen's model category $\mathbf{sSet_{Quillen}}$ is a type-theoretic model category whose underlying category is combinatorial, the proposition \ref{prop13} gives us a whole class of models of type theory with $\sum,\prod$ and Id-types by taking the full subcategory of injectively fibrant objects. Namely we have such a model in $[\:\mathcal{I},\mathbf{sSet_{Quillen}}\:]\f$, where $\mathcal{I}$ is any small category. This class of models has already been noticed by Shulman (last bullet of Examples 2.16 in \cite{shulman:invdia}).\\
\end{rmk}\bigskip

\section{The injective fibrations in $\GGpd$ made explicit}
\label{sec:explicit}\bigskip

\begin{rmk}\label{rmk14}
Note that we cannot hope without further restrictions on the universe $V_\kappa$ in $\Gpd$ that the map $p:\widetilde{U}\rightarrow U$ as defined in \ref{sec:tcou} is an injective fibration between injectively fibrant objects. Indeed, the map $\cI\to\mathbf{1}!$ is a split isofibration with small fibers after the forgetting of the involutions, so this map is a small projective fibration, \textit{i.e.} a pullback of $p$. But this map is not an injective fibration (\textit{cf.} the proof of \ref{prop14} below), hence $p$ cannot be an injective fibration.
\end{rmk}\bigskip

\begin{rmk}\label{prop14}
Consider the presheaf category $\GGpd$. There exists a projective fibration in $\GGpd$ which is not an injective fibration.\\
Indeed, consider the map $\cI\into \tr $. We recall (see \ref{not}) that $\cI$ is the following groupoid $$\xymatrix{0\ar[rr]^\phi&&1}$$ with the involution that maps $\phi$ to $\phi^{-1}$, and $\tr$ is the following groupoid $$\xymatrix{0\ar[rr]^\phi\ar[rd]^{\quad\circlearrowright}&&1\ar[ld]^\psi\\&2&}$$ with the involution that maps $\phi$ to $\phi^{-1}$ and $\psi$ to $\psi\circ\phi$. Since the inclusion $\cI\into \tr $ is an equivalence of groupoids and an injective-on-objects functor, this is an injective trivial cofibration in $\GGpd$. Note also that every object being fibrant in $\GGpd$ with respect to the projective model structure, the morphism $\cI\rightarrow\mathbf{1}!$ is a projective fibration in   $\GGpd$. Now, consider the following lifting problem in $\GGpd$ :
$$\xymatrix@=1,5cm{\cI\ar@{=}[r]\ar@{^{(}->}[d]&\cI\ar[d] \\ 
\tr \ar[r] &\mathbf{1}! \,.}$$
A lift cannot exist, since the fixed point $2$ of $\tr$ should be mapped to a fixed point in $\cI$ but such a fixed point does not exist. So the morphism $\cI\rightarrow\mathbf{1}$ is not an injective fibration in $\GGpd$.
\end{rmk}\bigskip

\begin{notn}\label{notn3}
We denote by $i'$ the inclusion in $\GGpd$ of remark \ref{prop14}:
\begin{center}
 $\cI\hookrightarrow\tr\,.$ 
\end{center}
Also remember we denote by $i$ (see exemple \ref{eg3}) the generating trivial cofibration in $\Gpd$ with its natural model structure and by $S(i)$ (see section \ref{sec:tpmsme} in chapter 4) the corresponding generating  projective trivial cofibration in $\GGpd$.
\end{notn}

\subsection{An explicit set of generating trivial cofibrations}\bigskip

\begin{prop}\label{prop15}
Let $f$ be a map in the injective type-theoretic fibration category $\GGpd_{\mathbf{inj}}$, the following are equivalent :\\
\begin{enumerate}[label=(\roman*)]
\item $f$ is a trivial cofibration.
\item $f$ is a transfinite composition of pushouts with elements in $\lbrace{S(i),i'}\rbrace$.\\
\end{enumerate}
\end{prop}
\begin{proof}\label{proofprop15}
The implication $(ii)\Rightarrow(i)$ is clear since $S(i)$ and $i'$ are levelwise trivial cofibrations and so they are injective trivial cofibrations. The stability of trivial cofibrations by pushouts and transfinite compositions allows to conclude.\\
Conversely, let $f:A\rightarrow B$ be a trivial cofibration. Note that $f$ is a levelwise trivial cofibration, hence $\underline{f}$ is an injective-on-objects equivalence, meaning that $\underline{f}$ is (isomorphic to) the inclusion of a full subgroupoid of $B$ which is equivalent to $B$.\\
We define a set $S$ whose set of objects are pairs $(\la,X)$ with $\la$ any ordinal and $X$ any $\la$-sequence  in $\text{Push}(\lbrace S(i), i'\rbrace)$ with $X_0=A$, moreover for every $\kappa\leqslant\la$ one has $X_\kappa\subseteq B$ (with the involution on $X_\kappa$ being the restriction of the involution $\be$ on $B$).\\
We provide $S$ with the structure of a preordered set as follows,
\begin{center}
$(\la,X)\leqslant(\la',X')$ \textit{iff} $\la\leqslant\la'$ and $X'_{|[0,\la]} = X$.
\end{center}
Note that $S$ is a non-empty set. Indeed, $(0,X)$ with $X_0 := A$ is an element of $S$.\\
Let $C$ be a non-empty chain of $S$. By using the universal property of the colimit $\underset{(\la,X^{\la})\in C}{\text{colim}\la}$ one gets a sequence denoted $(\displaystyle{\bigcup_{(\la,X^{\la}\in C)}X^{\la}})$ from $[0,\underset{(\la,X^{\la})\in C}{\text{colim}\la}]$ to $\GGpd$,
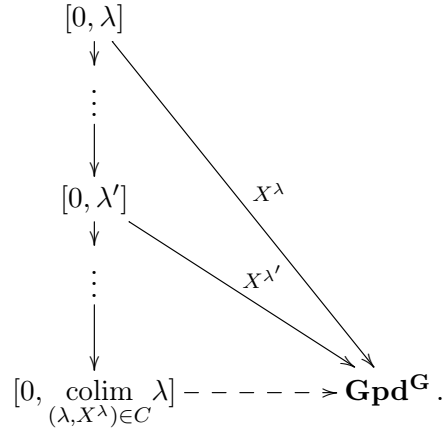
$$\xymatrix@C=2cm@R=0,5cm{[0,\la]\ar[rdddd]^{X^{\la}}\ar[d] & \\
\vdots\ar[d] & \\
[0,\la']\ar[rdd]^{X^{\la'}}\ar[d] & \\
\vdots\ar[d] & \\
[0,\underset{(\lambda,X^\la)\in C}{\text{colim}}\la]\ar@{-->}[r] & \GGpd\,.}$$
Clearly $(\underset{(\lambda,X^\la)\in C}{\text{colim}}\la, \displaystyle{\bigcup_{(\la,X^{\la})\in C}X^{\la}})$ is an element of $S$ and $(\la,X^\la)\leqslant (\underset{(\lambda,X^\la)\in C}{\text{colim}}\la, \displaystyle{\bigcup_{(\la,X^{\la})\in C}X^{\la}})$ for every $(\la,X^\la)\in C$.\\
By Zorn's lemma one gets a maximal element in $S$ denoted $(\la_{\text{max}},X^{\text{max}})$. It suffices to prove that $X^{\text{max}}_{\la_{max}}$ is (isomorphic to) $B$. Assume this is not the case. Since $X^{\text{max}}_{\la_{\text{max}}}$ is a full subgroupoid of $B$, there exists an object $x$ of $B$ such that $x$ is not an object of $X^{\text{max}}_{\la_{\text{max}}}$.\\
We define a $(\la_{\text{max}}+1)$-sequence $X$ as follows. Take $X_{|[0,\la_{\text{max}}]} := X^{\text{max}}$ and the definition of $X_{\la_{\text{max}}+1}$ depends on $x$. Indeed, by essential surjectivity of $f$ there exists $y\in A\subseteq X^{\text{max}}_{\la_{\text{max}}}$ and a morphism $\phi$ in $B$ from $y$ to $x$ .\\
If $x$ is not a fixed point then $\phi$ is not a fixed morphism and we get $X_{\la_{\text{max}}+1}$ as a subgroupoid of $B$ by taking the following pushout,
$$\xymatrix@=1,5cm{S(\b1)\ar[r]^y\ar@{^{(}->}[d]_{S(i)} & X^{\text{max}}_{\la_{\text{max}}}\ar[d] \\
S(\I)\ar[r] & X_{\la_{\text{max}}+1}\pushoutcorner\,.}$$ 
Otherwise, $x$ is a fixed point then, if $\phi$ is a fixed morphism (hence note that $y$ is a fixed point) one defines $X_{\la_{\text{max}}+1}$ by the following pushout,
$$\xymatrix@=1,5cm{\cI\ar[r]^{id_y}\ar@{^{(}->}[d]_{i'} & X^{\text{max}}_{\la_{\text{max}}}\ar[d] \\
\tr\ar[r] & X_{\la_{\text{max}}+1}\pushoutcorner\,,}$$
otherwise $\phi$ is not a fixed morphism and one defines $X_{\la_{\text{max}}+1}$ by the following pushout,
$$\xymatrix@=1,5cm{\cI\ar[r]\ar@{^{(}->}[d]_{i'} & X^{\text{max}}_{\la_{\text{max}}}\ar[d] \\
\tr\ar[r] & X_{\la_{\text{max}}+1}\pushoutcorner\,,}$$
where the upper horizontal morphism maps the unique non-trivial isomorphism in $\cI$ to $\be(\phi)^{-1}\circ \phi$.\\
In all cases $X_{\la_{\text{max}}+1}$ is a subgroupoid of $B$ and $(\la_{\text{max}}+1,X)$ is an element of $S$ satisfying $(\la_{\text{max}},X^{\text{max}})<(\la_{\text{max}}+1,X)$ contradicting the maximality of $(\la_{\text{max}},X^{\text{max}})$.\\
So one concludes that  $X^{\text{max}}_{\la_{\text{max}}}$ is (isomorphic to) $B$ and $f$ is a transfinite composition of pushouts in $\lbrace S(i), i'\rbrace$.  
\end{proof}\bigskip

\begin{prop}\label{prop16}
Let $f$ be a morphism in the injective type-theoretic fibration category $\GGpd_{\mathbf{inj}}$, the following are equivalent :\\
\begin{enumerate}[label=(\roman*)]
\item $f$ is a fibration.
\item $f$ has the right lifting property with respect to the elements of the set $\lbrace{S(i), i'}\rbrace$.\\
\end{enumerate}
\end{prop}
\begin{proof}\label{proofprop16}
Straightforward with the previous proposition.
\end{proof}\bigskip

\begin{rmk}
The above proposition \ref{prop16} will be a useful tool to prove the fibrancy of our expected universe.
\end{rmk}

\section{A univalent model in the injective type-theoretic fibration category $\GGpd_{\mathbf{inj}}$}
\label{sec:injmodel}\bigskip

\subsection{A universe}
\label{subsec:discreteuniverse}\bigskip

\begin{rmk}
Recall that, according to remark \ref{rmk14} and proposition \ref{prop14}, we need to introduce some restrictions on $V_\kappa$ in the construction of $p:\widetilde{U}\rightarrow U$ to get a universe in our injective setting.
\end{rmk}\bigskip

\begin{notn}\label{notn4}
We recall (see \ref{sec:tcou}) that when $V_\kappa$ is $\text{Gpd}_\Delta(V_\kappa)$ we denote $U_{V_\kappa}$ (respectively $\widetilde{U_{V_\kappa}}$) by $U_\Delta$ (respectively $\widetilde{U}_\Delta$) and $p_\Delta:\widetilde{U}_{\Delta}\rightarrow U_\Delta$ our candidate for universe.
\end{notn}\bigskip

\begin{lem}\label{lem15}
The morphism $p_\Delta:\widetilde{U}_{\Delta}\rightarrow U_\Delta$ in $\GGpd$ is an injective fibration.\\ 
\end{lem}
\begin{proof}\label{prooflem15}
Thanks to \ref{prop16} it suffices to prove that $p_\Delta$ has the right lifting property with respect to $S(i)$ and $i'$. We know that $p_\Delta$ has the right lifting property against $S(i)$. Indeed, $p_\Delta$ is a projective fibration and $S(i)$ is a (generating) projective trivial cofibration.\\
Assume we have a lifting problem in $\GGpd$ as follows :
$$\xymatrix@=1,5cm{\cI\ar@{^{(}->}[d]_{i'}\ar[r]^f & \widetilde{U}_{\Delta} \ar[d]^{p_\Delta}\\
 \tr \ar[r]_{g} & U_\Delta~.}$$
Assume $f(0)$ is the tuple $(A_0,A_1,x,\chi)$, then $f(1)=\tilde{u}_\Delta(f(0))=(A_1,A_0,\chi(x),\chi^{-1})$. Assume $f(\phi)$ is the tuple $(\rho_0,\rho_1,\text{id}_{\chi(x)},\text{id})$, 
it means that the following square commutes\\
$$\xymatrix@=1,5cm{ \ar[r]^\chi \ar[d]_{\rho_0}& \ar[d]^{\rho_1}\\
\ar[r]_{\chi^{-1}} & \,.}$$ 
Since $f$ is compatible with the involutions, we have\\
$f(\phi^{-1})=\tilde{u}_\Delta(f(\phi))$. But $\tilde{u}_\Delta(f(\phi))=(\rho_1,\rho_0,\text{id}_x,\text{id})$ and\\
$f(\phi^{-1})=f(\phi)^{-1}=(\rho_0,\rho_1,\text{id}_{\chi(x)},\text{id})^{-1}=(\rho_0^{-1},\rho_1^{-1},\text{id}_{\rho_1(\chi(x))},\text{id})$.\\
So one has the equalities $\rho_1=\rho_0^{-1}$ and $\rho_1(\chi(x))=x$.\\
One has $p_\Delta(f(0))= g(0),~p_\Delta(f(1))= g(1),~p_\Delta(f(\phi))= g(\phi)$.\\
Assume $g(2)$ is the tuple $(C_0,C_1,\eta)$ and $g(\psi):g(1)\rightarrow g(2)$ is the tuple $(\tau_0,\tau_1,\text{id})$. It means that the following square commutes,
$$\xymatrix@=1,5cm{ \ar[r]^{\chi^{-1}}\ar[d]_{\tau_0} & \ar[d]^{\tau_1}\\
 \ar[r]_\eta & \,.}$$
Moreover since $g$ is compatible with the involutions, it means that $g(\psi\circ\phi)= u_\Delta(g(\psi))$.\\
But $u_{\Delta}(g(\psi))= u_{\Delta}(\tau_0,\tau_1,\text{id})= (\tau_1,\tau_0,\text{id})$ and one has\\
$g(\psi\circ\phi)= g(\psi)\circ g(\phi)= (\tau_0,\tau_1,\text{id})\circ (\rho_0,\rho_1,\text{id})= (\tau_0\circ\rho_0,\tau_1\circ\rho_1,\text{id})$. So one has the equality $\tau_0\circ\rho_0= \tau_1$. Note that $g(2)$ is fixed in $U_\Delta$ (since $2$ is a fixed point), hence $C_1= C_0$ and $\eta$ is an involution. Now we define a morphism
$$ j: \textstyle\tr \longrightarrow \widetilde{U}_\Delta$$
in $\GGpd$ as follows. Take $j(\phi)= f(\phi)$. We need to define $j(2)$ and $j(\psi)$ such that $p_\Delta(j(2))= g(2),~p_\Delta(j(\psi))= g(\psi)$ and $j(2)$ has to be a fixed point and $j(\psi)$ has to satisfy the equality $j(\psi)\circ j(\phi)= j(\psi\circ\phi)= \tilde{u}_\Delta(j(\psi))$.\\
Take $j(2)=(C_0,C_0,\tau_0(\rho_0(x)),\eta)$. Indeed, one has the following equalities
\begin{align*}
\eta(\tau_0(\rho_0(x)))
&=\tau_1(\chi^{-1}(\rho_0(x)))\\
&=\tau_0(\rho_0(\chi^{-1}(\rho_0(x))))\\
&=\tau_0(\rho_0(\rho_{1}(\chi(x))))\\
&=\tau_0(\rho_0(x))~.\\
\end{align*}
So $j(2)$ is a fixed point in $\widetilde{U}_\Delta$. 
Finally, take $j(\psi)=(\eta\circ\tau_0,\tau_0\circ\chi,\text{id},\text{id})$, indeed $\tau_0\circ\chi= \eta\circ\eta\circ\tau_0\circ\chi$ and $\eta(\tau_0(\chi(x)))= \tau_1(\chi^{-1}(\chi(x)))= \tau_1(x)= \tau_0(\rho_0(x))$. So $j(\psi)$ is a map from $j(1)= f(1)= (A_0,A_1,\chi(x),\chi^{-1})$ to $j(2)= (C_0,C_0,\tau_0(\rho_0(x)),\eta)$.\\
We have the following equalities,
\begin{align*}
j(\psi\circ\phi)&=j(\psi)\circ j(\phi)\\
&=(\eta\circ\tau_0,\tau_0\circ\chi,\text{id},\text{id})\circ f(\phi)\\
&=(\eta\circ\tau_0,\tau_0\circ\chi,\text{id},\text{id})\circ (\rho_0,\rho_1,\text{id},\text{id})\\
&=(\eta\circ\tau_0\circ\rho_0,\tau_0\circ\chi\circ\rho_1,\text{id},\text{id})\\
&=(\eta\circ\tau_1,\eta\circ\tau_1\circ\rho_1,\text{id},\text{id})\\
&=(\eta\circ\tau_1,\eta\circ\tau_0,\text{id},\text{id})\\
&=(\tau_0\circ\chi,\eta\circ\tau_0,\text{id},\text{id})\\
&=\tilde{u}_\Delta(\eta\circ\tau_0,\tau_0\circ\chi,\text{id},\text{id})\\
&=\tilde{u}_\Delta(j(\psi))~.\\
\end{align*}
The map $j$ is a diagonal filler for our lifting problem.
\end{proof}\bigskip

\begin{lem}\label{lem16}
The objects $\widetilde{U_\Delta}$ and $U_\Delta$, equipped with their involutions $\widetilde{u_\Delta}$ and $u_\Delta$ respectively, are two injectively fibrant objects in $\GGpd$.\\
\end{lem}
\begin{proof}\label{prooflem16}
We start with $U_\Delta$. It suffices to prove that the unique morphism from $U_\Delta$ to the terminal object $\mathbf{1}!$ has the right lifting property with respect to $S(i)$ and $i'$.\\ 
First, assume we have the following lifting problem :
$$\xymatrix@=1,5cm{\cI\ar@{^{(}->}[d]_{i'}\ar[r]^f & U_{\Delta}\ar[d]\\ 
 \tr \ar[r] & \mathbf{1}! \,.}$$
Assume $f(0)$ is the tuple $(A_0,A_1,\chi)$, so $f(1)$ is the tuple\\
$u_\Delta(A_0,A_1,\chi)= (A_1,A_0,\chi^{-1})$. Also assume $f(\phi)$ is the tuple $(\rho_0,\rho_1,\text{id})$, it means that the following square commutes
$$\xymatrix@=1,5cm{ \ar[r]^\chi \ar[d]_{\rho_0} & \ar[d]^{\rho_1} \\
\ar[r]_{\chi^{-1}}& \,.}$$ 
By compatibility of $f$ with the involutions, one has $f(\phi^{-1})= u_\Delta(f(\phi))$.\\
But one has the equalities $f(\phi^{-1})= f(\phi)^{-1}= (\rho_0,\rho_1,\text{id})^{-1}= ({\rho_0}^{-1},{\rho_1}^{-1},\text{id})$ and\\
$u_\Delta(f(\phi))= u_\Delta(\rho_0,\rho_1,\text{id})= (\rho_1,\rho_0,\text{id})$. So one has $\rho_1= \rho_0^{-1}$. Take $j(\phi)= f(\phi)$ and $j(2)= (A_0,A_0,\rho_1\circ\chi)$. Indeed, $\rho_1\circ\chi$ is an involution since 
\begin{align*}
(\rho_1\circ\chi)^{-1}&=\chi^{-1}\circ\rho^{-1}_1\\
&=\chi^{-1}\circ\rho_0\\
&=\rho_1\circ\chi~,
\end{align*}
so $j(2)$ is a fixed point in $U_\Delta$. Now, take $j(\psi)=(\chi^{-1},\rho_1\circ\chi,\text{id})$. Indeed, this is a morphism from $j(1)= f(1)$ to $j(2)$. We need to check that $j(\psi\circ\phi)= u_\Delta(j(\psi))$. But one has the following equalities,
\begin{align*}
j(\psi\circ\phi)&=j(\psi)\circ j(\phi)\\
&=(\chi^{-1},\rho_1\circ\chi,\text{id})\circ(\rho_0,\rho_1,\text{id})\\
&=(\chi^{-1}\circ\rho_0,\rho_1\circ\chi\circ\rho_1,\text{id})
\end{align*}
and 
\begin{align*}
u_\Delta(j(\psi))&= u_\Delta(\chi^{-1},\rho_1\circ\chi,\text{id})\\
&=(\rho_1\circ\chi,\chi^{-1},\text{id})\\
&=(\chi^{-1}\circ\rho_0,\rho_1\circ\chi\circ\rho_1,\text{id}).
\end{align*}
So $j$ is a diagonal filler for our lifting problem.\\
Second, $U_\Delta$ is a projectively fibrant object and so $U_\Delta\rightarrow \mathbf{1}!$ has the right lifting problem with respect to $S(i)$ since $S(i)$ is in particular a projective trivial cofibration.\\
Now recall from lemma \ref{lem15} that $p_\Delta:\widetilde{U_\Delta}\rightarrow U_\Delta$ is an injective fibration, moreover fibrations are closed by composition, hence we deduce the fibrancy of $\widetilde{U_\Delta}$ from the fibrancy of $U_{\Delta}$,

\begin{displaymath}
    \xymatrix @!0 @R=3cm @C=3cm {\widetilde{U_{\Delta}}\ar@{-->}[rd] \ar@{->>}[d]_{p_{\Delta}} & \\
              U_{\Delta}\ar@{->>}[r] & \mathbf{1}! \,. }
\end{displaymath}
\end{proof}\bigskip

\begin{thm}\label{thm7}
The map $p_\Delta:\widetilde{U_\Delta}\rightarrow U_\Delta$ is a universe in the injective type-theoretic fibration category $\GGpd_{\mathbf{inj}}$.\\
\end{thm}
\begin{proof}\label{proofthm7}
It follows from \ref{lem15} and \ref{lem16}, plus the fact that $p_\Delta~:~\widetilde{U_\Delta}\rightarrow U_\Delta$ is a universe in the projective type-theoretic fibration category $\GGpd$ hence conditions \textit{(i),(ii) and (iii)} of definition \ref{def:univ} have already been checked since pullbacks of $p_\Delta$, namely the small fibrations, are the same in the injective setting. 
\end{proof}\bigskip

\begin{cor}\label{cor3}
The category $(\GGpd)\f$ hosts a model of type theory with $\ssum{}$, $\sprod{}$ and \Id-types plus a universe (for each inaccessible cardinal $\kappa$).
\end{cor}
\begin{proof}\label{proofcor3}
Assuming the initiality of the syntactic category of type theory, this is \ref{prop17} plus \ref{thm7}.
\end{proof}

\subsection{The univalence property of our universe}
\label{subsec:univdiscreteuniverse}\bigskip

We now move on to constructing some specific injective path objects in $\GGpd$ (more generally in the slice category $\GGpd/C$ for any $C\in\text{Ob}(\GGpd)$).\\
Let $f:A\rightarrow C$ be a morphism in $\GGpd$, by the universal property of the pullback one gets the diagonal morphism $\delta$ as follows :
$$\xymatrix{A\ar@{-->}[rd]^{\delta}\ar@/^2pc/[rrrd]^{\text{id}}\ar@/_2pc/[rddd]_{\text{id}}&&&\\&{A\times_c A}\pullbackcorner\ar[rr]\ar[dd]&&A\ar[dd]^f\\\\&A\ar[rr]_f&&C~.}$$
We define a groupoid $P_cA$ equipped with an involution $\pi_cA$ as follows. The objects of the groupoid $P_cA$ are tuples $(x,y,\varphi)$ where $x,y\in A$ and $\varphi:x\rightarrow y$ is an isomorphism in $A$ such that $f(x)= f(y)$ and $f(\varphi)= 1_{f(x)}= 1_{f(y)}$. A morphism in $P_cA$ between $(x,y,\varphi)$ and $(x',y',\varphi')$ is a couple $(\sigma,\tau)$ where $\sigma: x\rightarrow x'$ is an isomorphism in $A$, and $\tau: y\rightarrow y'$ is an isomorphism in $A$ such that $f(\sigma)= f(\tau)$ and $\varphi'= \tau\circ\varphi\circ\sigma^{-1}$. Let $(\sigma,\tau):(x,y,\varphi)\rightarrow (x',y',\varphi')$ and $(\sigma',\tau'):(x',y',\varphi')\rightarrow (x'',y'',\varphi'')$ be two composable morphisms in $P_cA$. We define the composition $(\sigma', \tau')\circ (\sigma, \tau):(x,y,\varphi)\rightarrow (x'',y'',\varphi'')$ as $(\sigma'\circ\sigma, \tau'\circ\tau)$. The reader can easily check that $f(\sigma'\circ\sigma)= f(\tau'\circ\tau)$ and $(\tau'\circ\tau)\circ\varphi\circ(\sigma'\circ\sigma)^{-1}=\varphi''$. Of course this composition is associative by the associativity of the composition in $A$. The inverse of the morphism $(\sigma,\tau)$ is $(\sigma^{-1},\tau^{-1})$. Define the involution $\pi_cA$ on $P_CA$ as follows : 
$$\fsix{\pi_cA:\qquad P_CA}{P_CA}{(x,y,\varphi)}{(\alpha(x),\alpha(y),\alpha(\varphi))}{(\sigma,\tau)}{(\alpha(\sigma),\alpha(\tau))}$$
where $\alpha$ is the involution on $A$. We define $\delta_1,\delta_2$ in $\GGpd$ as the following morphisms, 
$$\fsix{\delta_1:\qquad A}{P_CA}{x}{(x,x,1_x)}{\varphi}{(\varphi,\varphi)}$$
$$\fsix{\delta_2:\qquad P_CA}{A\times_c A}{(x,y,\varphi)}{(x,y)}{(\sigma,\tau)}{(\sigma,\tau)\,.}$$
The morphisms $\delta_1$ and $\delta_2$ are compatible with the involutions ($\alpha$ on $A$, $\pi_cA$ on $P_cA$, $\alpha\times \alpha$ on $A\times_c A$) and so they are really morphisms in $\GGpd$. Moreover we have the factorization $\delta= \delta_2\circ\delta_1$.\bigskip

\begin{prop}\label{prop20}
Being given $f:A\rightarrow C$ in $\GGpd$, $P_cA$ with its involution $\pi_cA$ as constructed above is a path object in the injective type-theoretic fibration category $\GGpd_{\mathbf{inj}}$.\\
\end{prop}
\begin{proof}\label{proofprop20}
It suffices to prove that $\delta_1$ is an injective trivial cofibration and $\delta_2$ is an injective fibration. We start with $\delta_1$. It suffices to prove that $\underline{\delta_1}$ is a trivial cofibration in $\Gpd$. Clearly it is an injective-on-objects functor. The morphism $\underline{\delta_1}$ is essentially surjective. Indeed let $(x,y,\varphi)$ be an element of $P_cA$ then $(\varphi^{-1},1_y)$ is a morphism in $P_cA$ between $\delta_1(y)= (y,y,1_y)$ and $(x,y,\varphi)$.\\
It remains to prove that $\underline{\delta_1}$ is a fully faithful functor. Since for every morphism $(\sigma,\tau)$ in $P_cA$ between $(x,x,1_x)$ and $(y,y,1_y)$ one has $1_y= \tau\circ 1_x \circ \sigma^{-1}$, we conclude that $\sigma= \tau$. So for every $x,y\in A$ the morphism $A(x,y)\rightarrow P_cA(\delta_1(x),\delta_1(y))$ induced by $\delta_1$ is a bijection.\\
It remains to prove that $\delta_2$ is an (injective) fibration. First, the morphism $\delta_2$ has the right lifting property with respect to $S(i)$. Indeed, $\delta_2$ is a projective fibration, namely $\underline{\delta_2}$ is an isofibration between groupoids. In order to prove this, let $(\sigma,\tau):(x,y)\rightarrow (x',y')$  be a morphism in $A\times_c A$ and $(x,y,\varphi)$ be an element of $P_cA$ such that $\delta_2(x,y,\varphi)= (x,y)$, then $(\sigma,\tau):(x,y,\varphi)\rightarrow (x',y',\tau\circ\varphi\circ\sigma^{-1})$ is an isomorphism in $P_cA$ that lies above $(\sigma,\tau)$ in $A\times_c A$ by $\delta_2$.\\
Last, consider the following lifting problem : 
$$\xymatrix@=1,5cm{\cI\ar@{^{(}->}[d]_{i'}\ar[r]^g & P_cA\ar[d]^{\delta_2}\\ \tr \ar[r]_h & A\times_c A \,.}$$
We define a morphism $j: \textstyle\tr\longrightarrow P_cA$ in $\GGpd$ as follows.
Take $j(\phi)= g(\phi)$. Let us assume that $g(\phi)$ is $(\sigma,\tau):(x,y,\varphi)\rightarrow (\alpha(x),\alpha(y),\alpha(\varphi))$.\\
Note that $(\sigma,\tau)^{-1}= (\sigma^{-1},\tau^{-1})= (\alpha(\sigma),\alpha(\tau))$ and $\alpha(\varphi)= \tau \circ \varphi \circ \sigma^{-1}$, since\\
$g(\phi^{-1})= \pi_cA(g(\phi))$.
Moreover, denote $h(2)$ by $(x'',y'')$ and $h(\psi): h(1)\rightarrow h(2)$ with $\delta_2(g(1))= h(1)$ by $(\sigma'',\tau''): (\alpha(x),\alpha(y),\alpha(\varphi))\rightarrow (x'',y'',\tau''\circ \alpha(\varphi)\circ \sigma''^{-1})$. 
Also, by compatibility of $h$ with the involutions one has 
$$(\sigma'',\tau'') \circ (\sigma,\tau)= h(\psi) \circ h(\phi)= h(\psi\circ \phi)= (\alpha\times \alpha)(h(\psi)) \,.$$ 
Hence one has $(\sigma''\circ \sigma,\tau''\circ \tau)= (\alpha(\sigma''),\alpha(\tau''))$, so $\sigma'' \circ \sigma= \alpha(\sigma'')$ and $\tau''\circ \tau= \alpha(\tau'')$.\\
In particular, $\alpha(x'')= x''$ and $\alpha(y'')= y''$, \textit{i.e. }$x''$, $y''$ are two fixed points in $A$. Take $j(2)= (x'',y'',\tau''\circ \alpha(\varphi)\circ \sigma''^{-1})$, this is a fixed point in $P_cA$ since
\begin{align*}
\alpha(\tau''\circ \alpha(\varphi)\circ \sigma''^{-1})&= \alpha(\tau'')\circ \varphi \circ \alpha(\sigma'')^{-1}\\
&= \tau''\circ \tau \circ \varphi \circ \sigma^{-1} \circ \sigma''^{-1}\\
&= \tau'' \circ \alpha(\varphi) \circ \sigma''^{-1} \,.
\end{align*}
Finally, take $j(\psi)= (\sigma'',\tau''):(\alpha(x),\alpha(y),\alpha(\varphi))\rightarrow (x'',y'',\tau'' \circ \alpha(\varphi) \circ \sigma''^{-1})$, indeed one has $j(\psi \circ \phi)= \pi_cA(j(\psi))$, since $(\sigma'' \circ \sigma, \tau'' \circ \tau)= (\alpha(\sigma''),\alpha(\tau''))= \pi_cA(\sigma'',\tau'') \,.$\\
So $j$ is a diagonal filler for the above lifting problem in $\GGpd$. Hence thanks to \ref{prop16} $\delta_2$ is an injective fibration.
\end{proof}\bigskip

\begin{prop}\label{pathfibrancy}
If $f:A\rightarrow C$ is an injective fibration in $\GGpd$ and $A$ is injectively fibrant then $P_CA$ is injectively fibrant.
\end{prop}
\begin{proof}
First, we prove that the first projection $pr_1:A\times_C A\to A$ is an injective fibration. Consider  the following lifting problem,
$$\xymatrix@=1,5cm{B\ar[r]^g\ar@{ >->}[d]^{\rotatebox[origin=c]{90}{$\sim$}}_{k} & A\times_C A\ar[d]^{pr_1} \\
D\ar[r]_h & A \,,}$$ 
where $k$ is any injective trivial cofibration. We can display the following lifting problem with respect to $f$,
$$\xymatrix@=1,5cm{B\ar[r]^g\ar@{ >->}[d]^{\rotatebox[origin=c]{90}{$\sim$}}_{k} & A\times_C A\pullbackcorner\ar[d]^{pr_1}\ar[r]^{pr_2}  & A\ar@{->>}[d]^f \\
D\ar[r]_h & A\ar@{->>}[r]_f & C\,.}$$
Since $f$ is a fibration there exists a diagonal filler $j_1$ as follows,
$$\xymatrix@=1,5cm{B\ar[r]^g\ar@{ >->}[d]^{\rotatebox[origin=c]{90}{$\sim$}}_{k} & A\times_C A\pullbackcorner\ar[d]^{pr_1}\ar[r]^{pr_2}  & A\ar@{->>}[d]^f \\
D\ar[r]_h\ar@{-->}[rru] & A\ar@{->>}[r]_f & C\,.}$$
Now using $h$, $j_1$ and the universal property of the pullback square, one gets a morphism $j_2$ satisfying $pr_1\circ j_2 = h$,
$$\xymatrix@=1,5cm{B\ar[r]^g\ar@{ >->}[d]^{\rotatebox[origin=c]{90}{$\sim$}}_{k} & A\times_C A\pullbackcorner\ar[d]^{pr_1}\ar[r]^{pr_2}  & A\ar@{->>}[d]^f \\
D\ar[r]_h\ar[rru]\ar@{-->}[ru]^{j_2} & A\ar@{->>}[r]_f & C\,.}$$
Using the uniqueness part of the universal property of the pullback square (with the universal problem involving $pr_1\circ g$ and $pr_2\circ g$), ones proves that the morphisms $g$ and $j_2\circ k$ are equal. So $j_2$ is a diagonal filler for our first lifting problem.\\
Last, using that $A$ is fibrant, $\delta_2$ and $pr_1$ are fibrations, and the fact that fibrations are closed under compositions, one concludes that $P_CA$ is fibrant,
$$\xymatrix@=1,5cm{P_CA\ar@{->>}[d]_{\delta_2}\ar@{-->}[rdd] & \\
A\times_C A\ar@{->>}[d]_{pr_1} & \\
A\ar@{->>}[r] & \b1! \,.}$$
\end{proof}\bigskip

\begin{rmk}\label{rmk16}
Note that the above proposition \ref{pathfibrancy} proves, under the assumptions that $f:A\to C$ is a fibration and $A$ is (injectively) fibrant, that $P_CA$ is an object of $(\GGpd)\f$, hence it is really a path object in the injective type-theoretic fibration category $\GGpd_{\mathbf{inj}}$, namely the category $(\GGpd)\f$ equipped with the subcategory given by the injective fibrations (between injectively fibrant objects).  
\end{rmk}\bigskip

\begin{lem}\label{lemmaforuniv}
The space of equivalences $E$ in our model is (isomorphic to) the path object $P_{C}A$ of \ref{prop20} where we take $C=\mathbf{1}!$, $A= U_\Delta$ and $f$ is the unique morphism from $U_\Delta$ to $\mathbf{1}!$.
\end{lem}
\begin{proof}
Note that $U_\Delta \to \b1!$ is a fibration since we proved in \ref{lem16} that $U_\Delta$ is injectively fibrant. So by \ref{pathfibrancy} $PU_\Delta$ is a fibrant object, namely is an element of $(\GGpd)\f$.
Also note that we have already computed $E$ in the proof of proposition \ref{prop:fiberspaceofeq}  of chapter 4, moreover this interpretation has not changed with the change of the type-theoretic fibration category under consideration (from the projective one to the injective one) since the morphism $p_\Delta$ is the same (hence small fibrations are the same, thus the interpretations of the identity types involved in the type of equivalences are still the same) as well as the categorical interpretations of $\Sigma$-types and $\Pi$-types.
\end{proof}\bigskip

\begin{thm}\label{thm8}
The universe $p_\Delta:\widetilde{U_\Delta}\rightarrow U_\Delta$ in the injective type-theoretic fibration category $(\GGpd)\f$ satisfies the univalence property.\\ 
\end{thm}
\begin{proof}\label{proofthm8}
Recall from \ref{sec:catunivaxiom} that it suffices to prove the upper horizontal arrow in the following commutative diagram,
$$\xymatrix{\underset{~}{U_\Delta}\ar[rr]\ar@{ >->}[dd]_*[@]{\hbox{$\sim$}}&&E\ar@{->>}[dd]\\\\PU_\Delta\ar@{->>}[rr]&&**[r]{U_\Delta\times U_\Delta}\,,}$$
that maps a small type to its identity equivalence is a homotopy equivalence. But this morphism is $\delta_1$, hence it is a weak equivalence (even a trivial cofibration see \ref{prop20}). Moreover in $(\GGpd)\f$ all objects are fibrant and cofibrant, thus the homotopy equivalences are the weak equivalences.

\end{proof}

\section{Conclusion}

In contrast with the previous chapter and the projective setting this chapter reinforces the suitability of the injective setting with respect to univalence (indeed, note that in the case of elegant Reedy categories as index categories the Reedy model structure on functors categories, used with success in \cite{shul13}, coincides with the injective model structure). However, the injective setting is not very tractable due to the difficulty of making fibrations and fibrant objects explicit.

%\setchapterpreamble{
%\begin{quote}
%In previous chapters the author gave two models of intensional type theory in the presheaf category $\GGpd$. %The first model is built from the projective model structure and the second one from the injective model %structure. These two models share a common universe (of discrete groupoids equipped with an involution), %which is non-univalent in the first case and univalent in the second case. Since the projective and %injective model structure give rise to Quillen equivalent model categories, we conclude that Quillen %equivalent model categories can have inequivalent interpretation of type theory. We outline some conclusions %on a \textit{Model Invariance Principle} (to the best of our knowledge the idea of a model invariance %principle was initially suggested by Michael Shulman here %\url{http://ncatlab.org/homotopytypetheory/show/model+invariance+problem}). Last, we propose a specific %conjecture for an enhanced model invariance principle. 
%\end{quote}
%}

%\chapter{A model invariance principle for Univalent Foundations}
%\label{sec:chp6}

%\input{chp6}

%\setchapterpreamble{
%\begin{quote}
%In this final short chapter we outline some directions of future work.
%\end{quote}
%}
%\chapter{Future work}
%\label{sec:futurework}

%\input{futurework}

\nocite{clef}
%\bibliographystyle{plain}
%\bibliography{Bibliography}
\addcontentsline{toc}{chapter}{Bibliography}
\printbibliography
\end{document}